\documentclass{amsart}
\usepackage{graphicx}
\usepackage{amsthm,amssymb,amsmath,xcolor}
\usepackage{tikz,hyperref,cleveref}
\usepackage{algorithm}
\usepackage{algpseudocode}

\newtheorem{theorem}{Theorem}[section]
\newtheorem*{theorem*}{Theorem}
\newtheorem{lemma}{Lemma}[section]
\newtheorem{corollary}{Corollary}[section]
\theoremstyle{remark}
\newtheorem{remark}{Remark}[section]

\DeclareMathOperator*{\argmin}{argmin}
\DeclareMathOperator{\sign}{sign}

\DeclareMathOperator{\diag}{diag}

\title{Randomized Kaczmarz with geometrically smoothed momentum}

\author[S. J. Alderman]{Seth J. Alderman}
\address{Department of Mathematics, University of Maryland, College Park, Maryland}
\email{ald@umd.edu}
\author[R. W. Luikart]{Roan W. Luikart}
\address{Department of Mathematics, Oregon State University, Corvallis, Oregon}
\email{luikartr@oregonstate.edu}
\author[N. F. Marshall]{Nicholas F. Marshall}
\address{Department of Mathematics, Oregon State University, Corvallis, Oregon}
\email{marsnich@oregonstate.edu}



\begin{document}

\begin{abstract}
This paper studies the effect of adding geometrically smoothed momentum to the randomized Kaczmarz algorithm, which is an instance of stochastic gradient descent on a linear least squares loss function.  We prove a result about the expected error in the direction of singular vectors of the matrix defining the least squares loss. We present several numerical examples illustrating the utility of our result and pose several questions.
\end{abstract}

\maketitle

\section{Introduction}
\subsection{Introduction} \label{sec:intro}
Let $A \in \mathbb{R}^{m \times n}$ be a tall ($m \ge n$) matrix,  $b \in \mathbb{R}^m$ be a column vector, and $f : \mathbb{R}^n \rightarrow \mathbb{R}$ be the least squares loss function
\begin{equation} \label{lossfunc}
f(x) = \frac{1}{2} \|A x - b\|_2^2 = \frac{1}{2} \sum_{i=1}^m ( \langle a_i, x \rangle - b_i)^2 = \sum_{i=1}^m f_i(x),
\end{equation}
where $\|\cdot\|_2$ is the $\ell^2$-norm, $\langle \cdot, \cdot \rangle$ is the standard inner product, $a_i^\top$ is $i$-th row of $A$, $b_i$ is the $i$-th entry of $b$, and $f_i(x) := (1/2) (\langle a_i, x \rangle - b_i)^2$.
For a positive probability vector $p \in \mathbb{R}^m$,
stochastic gradient descent, with batch size 1, is the iteration
\begin{equation} \label{sgdformula}
x_{k+1} = x_k - \alpha_k \nabla \frac{1}{p_{i_k}}  f_{i_k}(x_{k}) = x_k - \alpha_k \frac{\langle a_{i_k}, x_k \rangle - b_{i_k}}{p_{i_k}} a_{i_k},
\end{equation}
where $\alpha_k$ is the learning rate, and $i_k$ is chosen independently and randomly from $\{1,\ldots,m\}$ such that each $i$ is chosen with probability $p_i$. Setting $p_i = \|a_i\|_2^2/\|A\|_F^2$ and $\alpha_k= 1/\|A\|_F^2$, where $\|\cdot\|_F$ is the Frobenius norm, results in the randomized Kaczmarz algorithm
\begin{equation} \label{kaczmarz}
x_{k+1} = x_k  +  \frac{ b_{i_k} - \langle a_{i_k}, x_k \rangle}{\|a_{i_k}\|_2^2} a_{i_k},
\end{equation}
see \cite{Needell2015}.
The randomized Kaczmarz algorithm has a geometric interpretation: at iteration $k$, we perform an orthogonal projection of $x_k$ on the affine hyperplane $\{ y \in \mathbb{R}^n: \langle a_{i_k}, y \rangle = b_{i_k} \}$ defined by the $i_k$-th equation. If $A x = b$ is a consistent linear system of equations, meaning that there exists a solution $x$ satisfying $\langle a_i, x \rangle = b_i$ for all $i \in \{1,\ldots,m\}$, then the randomized Kaczmarz algorithm has the following convergence rate result established in 2009 by Strohmer and Vershynin: by  \cite[Theorem 2]{StrohmerVershynin2009} we have
\begin{equation} \label{kaczmarzresult}
\mathbb{E} \|x_k - x\|_2^2 \le ( 1 - \eta)^k \|x - x_0\|_2^2,
\end{equation}
where $\eta := \sigma_n^2/\|A\|_F^2$, where $\sigma_n$ denotes the smallest singular value of $A$. In 2021, Steinerberger \cite{Steinerberger2021} proved several results about how the error decays in different directions. Let $v_l$ be the right singular vector of $A$ associated with the $l$-th largest singular value $\sigma_l$ of $A$. Then
by \cite[Theorem 1.1]{Steinerberger2021} we have
\begin{equation} \label{kaczmarzstef}
\mathbb{E} \langle x_k - x, v_l \rangle  = ( 1 - \eta_{ l})^k \langle x_0 - x, v_l \rangle,
\end{equation}
where $\eta_l := \sigma_l^2/\|A\|_F^2$.
 Moreover, the direction in which the vectors approach the solution remains roughly constant: By \cite[Theorem 1.3]{Steinerberger2021}, we have
 \begin{equation} \label{stefanthm3}
\mathbb{E} \left\langle \frac{x_k - x}{\|x_k - x\|_2}, \frac{x_{k+1} - x}{\|x_{k+1} - x\|_2} \right\rangle^2 = 1 - \frac{1}{\|A\|_F^2} \left\| A \frac{x_k - x}{\|x_k - x\|_2} \right\|_2^2,
\end{equation}
assuming that $x_k,x_{k+1} \not = x$. Moreover, Steinerberger poses the following question based on these results:
\begin{quote}
``Once $x_k$ is mainly comprised of small singular vectors, its direction does
not change very much anymore  \ldots Could this property be used for
convergence acceleration?''  \cite[Page 4]{Steinerberger2021}.
\end{quote}

Our main result addresses this question: We show that adding geometrically smoothed momentum to the randomized Kaczmarz algorithm can accelerate convergence in the direction of small singular vectors. 
More precisely, we define the randomized Kaczmarz with geometrically smoothed momentum (KGSM) algorithm, see \eqref{eq:our-method}, and establish Theorem \ref{thm1}, which extends
\eqref{kaczmarzstef} to a setting with geometrically smoothed momentum. Our method is interpretable and illustrated by several numerical examples in \S \ref{numerics}.

Before going into the details of the algorithm and results, we present a numerical example 
to illustrate the accelerated convergence and interesting dynamics of KGSM \eqref{eq:our-method} compared to the randomized Kaczmarz \eqref{kaczmarz}, see Figure \ref{fig01}.

\begin{figure}[h!]
    \centering
    \includegraphics[width=.65\textwidth]{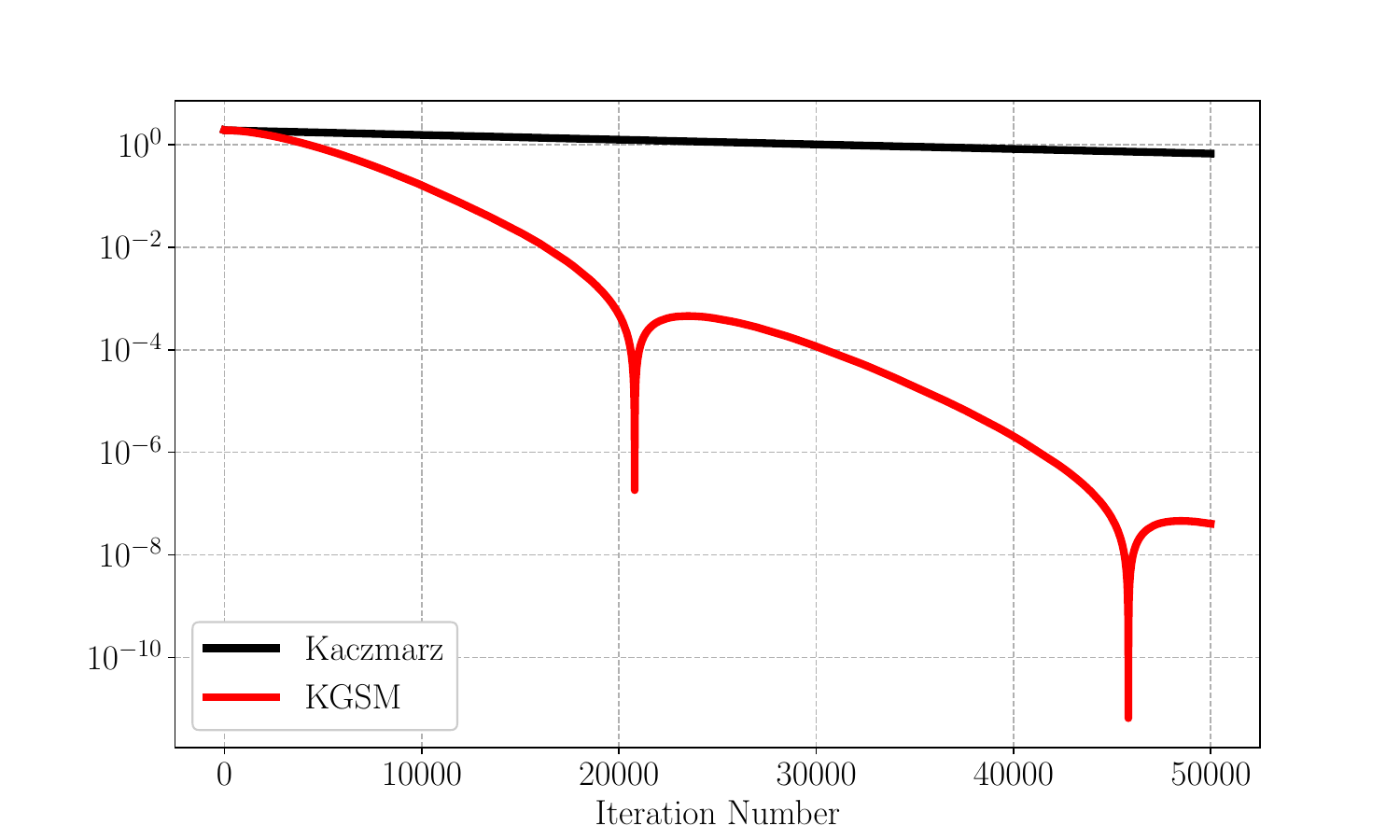}
    \caption{The error $|\langle x_k - x, v_n \rangle|$ in the direction of the smallest singular vector $v_n$ for randomized Kaczmarz \eqref{kaczmarz} and KGSM \eqref{eq:our-method} (see \S \ref{numericscomplex} for a precise description of this numerical example).}
    \label{fig01}
\end{figure}


\subsection{Motivation}
We are motivated by Polyak's Heavy Ball Momentum \cite{polyak1964}, which is the iteration
\begin{equation} \label{polyak}
x_{k+1} = x_k - \alpha_k \nabla f(x_k)  + \beta_k (x_k - x_{k-1}),
\end{equation}
where $f$ is the loss function; the name momentum comes from interpreting $\beta_k$ and $x_k - x_{k-1}$  as mass and velocity, respectively. One potential idea is to translate \eqref{polyak} directly into the following modified randomized Kaczmarz algorithm
\begin{equation} \label{beta0}
\left\{
\begin{array}{ccl}
   x_{k+1}& = & x_k +\displaystyle\frac{{b}_{i_k} - \langle x_k, a_{i_k} \rangle}{||a_{i_k}||^2_2}a_{i_k}   +M y_{k} \\[10pt]
   y_{k+1}& = & x_{k+1}-x_k,
\end{array} \right.
\end{equation}
where $y_k$ represents velocity and $M$ represents mass. However, 
numerical results, see  \S \ref{heavyballmomentumbatchsize1fails},
and theoretical results, see \cite{bollapragada2023fast,Kidambi2018}, indicate that Heavy Ball Momentum with batch size $1$  is not effective; in particular, \cite{bollapragada2023fast} explains:
\begin{quote}
``We study minibatching in the randomized Kaczmarz method as a necessary algorithmic structure for unlocking the fast convergence rate of HBM. \ldots Minibatch-HBM with a batch size of $B = 1$ provably fails to achieve faster convergence than SGD'' 
\cite[Page 5]{bollapragada2023fast}.
\end{quote}

Roughly speaking, the issue is that the momentum term is too noisy. One potential way to address this issue is to consider a minibatch version of Heavy Ball Momentum, which is studied by Bollapragada, Chen, and Ward \cite{bollapragada2023fast}, who prove that the convergence is accelerated when the minibatch size is large enough but at an increased computational cost per iteration, where minibatch size refers to the number of equations used to form the stochastic gradient, see \cite[Eq. 1.2]{bollapragada2023fast}.

In this paper, we take a different approach: instead of using a larger minibatch, we geometrically smooth the momentum term. For $\beta \in [0,1)$ and $M \in [0,1]$, we define
Kaczmarz with geometrically smoothed momentum (KGSM) by
\begin{equation} \label{eq:our-method}
\left\{
\begin{array}{ccl}
   x_{k+1}& = & x_k +\displaystyle\frac{{b}_{i_k} - \langle x_k, a_{i_k} \rangle}{||a_{i_k}||^2_2}a_{i_k}   +M y_{k} \\[10pt]
   y_{k+1}& = & \beta y_k + (1-\beta)(x_{k+1}-x_k) .
\end{array} \right.
\end{equation}
Note that setting $\beta = 0$ results in \eqref{beta0}, while setting $\beta$ close to $1$ results in a high level of smoothing. Our main result is an analog of \eqref{kaczmarzstef} 
which has an improved convergence rate for the signed error in the direction of small singular vectors; see Theorem \ref{thm1}. We demonstrate numerically that this expected signed error in the direction of the singular vectors is an effective model for the absolute numerical error in the direction of the singular vectors 
in some situations. Furthermore, we pose several questions about the dynamics of KGSM in \S \ref{discussion}.

\subsection{Related work}
The Kaczmarz algorithm first appeared in the literature in a 1937 paper of Kaczmarz \cite{Kaczmarz1937}. It was rediscovered 
in 1970 in the context of electron microscopy under the name algebraic reconstruction techniques (ART) \cite{GordonBenderHerman1970}. 
In 2009, Strohmer and Vershynin \cite{StrohmerVershynin2009} established the linear convergence result \eqref{kaczmarzresult} for the randomized Kaczmarz algorithm for consistent linear systems of equations; this analysis was extended to noisy linear systems by Needell  \cite{Needell2010}. Moreover, the block Kaczmarz methods corresponding to minibatch stochastic gradient descent have been considered by several authors, see \cite{bollapragada2023fast, derezinski2023sharp, Necoara2019,needell2014paved, Needell2015,
Needell2015b}.

In 2023, Marshall and Mickelin \cite{Marshall2023}, considered the problem of determining an optimal learning rate schedule $\alpha_k$ for the randomized Kaczmarz algorithm 
$$
x_{k+1} = x_k + \alpha_k \frac{\tilde{b}_{i_k}  - \langle a_{i_k}, x_k \rangle}{\|a_{i_k}\|^2_2} a_{i_k},
$$
for the case where $A x = b$ is a consistent linear system of equations, and $\tilde{b} = b + \varepsilon$, where $\varepsilon$ has independent mean zero random entries. The main result of \cite{Marshall2023} shows that there is an optimal learning rate schedule that depends on two hyperparameters. A similar learning rate problem for block Kaczmarz problems was subsequently studied in \cite{Tondji2023}. 

Part of the motivation for the current paper was to understand adaptive learning rates in the context of the randomized Kaczmarz algorithm. In particular, we were motivated by trying to understand ADAM \cite{Kingma2014}, a popular stochastic optimization method that adaptively controls the learning rate. The geometric smoothing of the momentum in \eqref{eq:our-method} is somewhat reminiscent of the geometric averaging used in ADAM. We note that the idea of using weighted averaging (without momentum) was considered in the context of the randomized Kaczmarz algorithm by \cite{Moorman2020}. 

Several authors have considered Kaczmarz algorithms that can be interpreted in terms of adaptive learning rates in the context of solving linear systems of equations with sparse corrupted equations
\cite{coria2024quantile,Haddock2020,jarman2021quantilerk,Steinerberger2022}.
 Here, the challenge is to find a way to ignore the corrupted equations in a hard or soft way, which can be interpreted as using an adaptive learning rate that de-emphasizes certain equations.

Another line of related research is the accelerated randomized Kaczmarz (ARK) \cite{liu2015}, which applies Nesterov's accelerated procedure, which defines $x_k,y_k,v_k$ iteratively by
$$
\left\{
\begin{array}{ccl}
y_k &=& \alpha_k v_k + (1-\alpha_k) x_k, \\[5pt]
x_{k+1} &=& y_k - \theta_k \nabla f(y_k), \\[5pt]
v_{k+1} &=& \beta_k v_k + (1-\beta_k) y_k - \gamma_k \nabla f(y_k), \\[5pt]
\end{array} \right.
$$
see \cite{Nesterov2004}.
Modifications and extensions of ARK have been considered by several authors; see \cite{gupta2023achieving,  zeng2023adaptive,  zuo2023fast}. While our motivation is different, the iteration considered in this paper has some similarities to Nesterov's accelerated procedure. 
In the general setting, several authors have considered momentum in stochastic iteration in different contexts; in particular, see \cite{Loizou2020,Morshed2021}.

Another related work by Han, Su, Xie \cite{han2022randomized}, 
incorporates momentum into randomized Douglas-Rachford methods, which are related to the randomized Kaczmarz algorithm and use iterative reflections. In this work, the authors use a sequence of reflections followed by one momentum-based step, which can be viewed as another way to achieve smoothing of the momentum term.

\subsection{Main contributions} Before presenting our main result and its corollaries, we outline the contributions of the paper. The main contributions are:
\begin{enumerate}
\item Introducing KGSM
\eqref{eq:our-method}, a Kaczmarz algorithm with geometrically smoothed momentum: The simplicity of this algorithm helps isolate the effect of geometrically smoothed momentum, while our theoretical and numerical results demonstrate that, despite its simplicity, the iterates of KGSM have a rich structure. 

\item Addressing a question raised by Steinerberger \cite{Steinerberger2021} (see \S \ref{sec:intro} above) about the possibility of using results on the convergence of the Kaczmarz algorithm along singular vectors to accelerate the convergence: Our main result extends  \cite[Theorem 1.1]{Steinerberger2021}. To the best of our knowledge, our work is the first to study the interplay between momentum and singular vectors in the context of the Kaczmarz algorithm. Our results provide insight into the $(M,\beta)$ parameter space, which governs the dynamical system of iterates produced by KGSM.

\item Building towards a precise understanding of stochastic optimization methods in the context of linear algebra: Our results describe the behavior of geometrically smoothed momentum in relation to the singular values and vectors of the matrix defining a linear least squares loss function. These results represent a step towards developing a complete and precise understanding of the algorithms used in practice (such as ADAM \cite{Kingma2014}) in the linear algebra setting. Our work opens several new directions for possible inquiry, including extending our analysis to more sophisticated momentum algorithms; see \S \ref{discussion} for further discussion.
\end{enumerate}

\subsection{Main result} \label{sec:mainresult}
Our main result establishes a convergence result for the signed error in the direction of singular vectors for randomized Kaczmarz with geometrically smoothed momentum (KGSM). Let $\beta \in [0,1)$  and $M \in [0,1]$ be given parameters, suppose that $x_0 \in \mathbb{R}^n$ is a given initial vector, and let $y_0 = 0 \in \mathbb{R}^n$ be the zero vector. Define
\begin{equation} \label{eq:our-method2}
\left\{
\begin{array}{ccl}
   x_{k+1}& = & x_k +\displaystyle\frac{{b}_{i_k} - \langle x_k, a_{i_k} \rangle}{||a_{i_k}||^2_2}a_{i_k}   +M y_{k} \\[10pt]
   y_{k+1}& = & \beta y_k + (1-\beta)(x_{k+1}-x_k) ,
\end{array} \right. 
\end{equation}
for $k = 0,1,\ldots$ 
where $i_k$ is chosen randomly and independently from $\{1,\ldots,m\}$ such that $i$  is chosen with probability $\|a_i\|_2^2/\|A\|_F^2$.

\begin{theorem}[Main result] \label{thm1}
Fix $\beta \in [0,1)$, $M \in [0,1]$, and ${ l} \in \{1,\ldots,n\}$. 
Suppose that $x_k$ is defined by  
KGSM \eqref{eq:our-method2}. For all $k \ge 0$ we have
\begin{equation} \label{maineq}
 \mathbb{E} \langle x_{k+1} - x, v_l \rangle =   
    \begin{bmatrix}
        r \\
        \zeta \\
    \end{bmatrix}^\top
    \begin{bmatrix}
        r & \zeta \\
        -1 & \beta \\
    \end{bmatrix}^k
    \begin{bmatrix}
   1 \\ 
   \frac{-1}{1-\beta}
\end{bmatrix}
 \langle x_0 -x, v_l \rangle,
\end{equation}
where
\begin{equation*}
r := 1 - \frac{\sigma_l^2}{\|A \|_F^2}+M(1-\beta), \quad \text{and} \quad
\zeta := M(1-\beta)^2.
\end{equation*}
\end{theorem}

The proof of Theorem \ref{thm1} is given in \S \ref{proofmainresult}. 
We illustrate the result with several corollaries and numerical examples.
First, in the following corollary, we show that setting $M = 0$ recovers Theorem \ref{thm1} of Steinerberger \cite{Steinerberger2021}.

\begin{corollary}[$M = 0$] \label{corstefan} Additionally, assume $M = 0$, then \eqref{eq:our-method2} reduces to randomized Kaczmarz \eqref{kaczmarz}, and 
\eqref{maineq} reduces to
\begin{equation} \label{resultM0} 
 \mathbb{E} \langle x_{k+1} - x, v_l \rangle =   
    \begin{bmatrix}
        1-\eta_l \\
        0 \\
    \end{bmatrix}^\top
    \begin{bmatrix}
        1-\eta_l & 0 \\
        -1 & \beta \\
    \end{bmatrix}^{k}
    \begin{bmatrix}
   1 \\ 
   -\frac{1}{1-\beta}
\end{bmatrix}
 \langle x_0 -x, v_l \rangle,
\end{equation}
where we used the notation $\eta_l := \sigma_l^2/\|A\|_F^2$. Recall that 
$$
\begin{bmatrix}
a & 0 \\
b & c 
\end{bmatrix}^n 
= 
\begin{bmatrix}
a^n & 0 \\
x_n & c^n
\end{bmatrix},
$$
where $x_n = b (a^n + a^{n-1} c + \cdots + a c^{n-1} + c^n)$. Thus, \eqref{resultM0} reduces to
\begin{equation*} 
 \mathbb{E} \langle x_{k+1} - x, v_l \rangle =   (1-\eta_l)^{k+1} \langle x_0 - x, v_l \rangle,
\end{equation*} which is Theorem \ref{thm1} of Steinerberger \cite{Steinerberger2021}.
\end{corollary}

%

Our next corollary optimizes the value of the smoothing parameter $\beta$ by minimizing the maximum magnitude eigenvalue of the $2 \times 2$ matrix
\begin{equation} \label{eqB}
B :=
    \begin{bmatrix}
        r & \zeta \\
        -1 & \beta \\
    \end{bmatrix}
\end{equation}
from  \eqref{maineq} in Theorem \ref{thm1}. The eigenvalues of $B$ are
\begin{equation} \label{eqeigenvalues}
\lambda_1 := \frac{r + \beta + \sqrt{(r-\beta)^2 - 4 \zeta}}{2}
\quad \text{and} \quad
\lambda_2 := \frac{r + \beta - \sqrt{(r-\beta)^2 - 4 \zeta}}{2}.
\end{equation}
Since $r + \beta > 0$, it follows that $|\lambda_1| \ge |\lambda_2|$ for all $\beta,M \in [0,1]$.
In the following, we consider the problem of minimizing $|\lambda_1|$. Fix $M \in [0,1]$, and consider the eigenvalue $\lambda_1 = \lambda_1(\beta)$ as a function of the smoothing parameter $\beta$.

\begin{corollary}[Minimizing $\lambda_1$] \label{cormin}  Fix $M \in [0,1]$. We have
\[
\argmin_{\beta \in [0,1]} {|\lambda_1(\beta)|} =
\left\{\begin{array}{cl}
\displaystyle
     1 - \frac{\eta_l}{(1-\sqrt{M})^2} & \text{if} \quad 0 \le M \le (1-\sqrt{\eta_l})^2, \\[10pt]
     0 & \text{if} \quad  (1-\sqrt{\eta_l})^2 < M \le 1-\eta_l, \\[10pt]
\displaystyle
     1 - \frac{\eta_l}{(1+\sqrt{M})^2}  & \text{if} \quad 1-\eta_l \le M \le 1.
\end{array}\right.
\] 
\end{corollary}

The proof of Corollary \ref{cormin} is given in \S \ref{proofcors}.
Our next corollary restates Theorem \ref{thm1} for the case where $\beta$ is chosen using the optimized value of Corollary \ref{cormin} for the case $0 \le M \le (1-\sqrt{\eta_l})^2$, which is the numerical region of interest, see \S \ref{explorembeta}.

\begin{corollary} \label{coropt} In addition to the hypotheses of Theorem \ref{thm1} assume that $0 \le M \le (1-\sqrt{\eta_l})^2$ and set
\begin{equation} \label{eqbetaf}
\beta = 1 - \frac{\eta_l}{(1-\sqrt{M})^2}. 
\end{equation}
Then, 
$$
 \mathbb{E} \langle x_{k+1} - x, v_l \rangle =   
\left(1 - \frac{\eta_l}{1-\sqrt{M}} \right)^k \left( 1 +  \eta_l\frac{\sqrt{M}(k+1))-1}{1- \sqrt{M}}  \right)
 \langle x_0 -x, v_l \rangle.
 $$
\end{corollary}

The proof of Corollary \ref{coropt} is given in \S \ref{proofcors}.
Several numerical experiments illustrating Corollary \ref{coropt},
and its effectiveness at modeling the behavior of KGSM
are included in \S \ref{numerics}. Next, we consider the case where the matrix $B$ defined in \eqref{eqB} has complex eigenvalues.

\begin{corollary} \label{corarg} In addition to the assumptions in Theorem \ref{thm1}, suppose that the eigenvalue $\lambda_1$ defined in  \eqref{eqeigenvalues} is complex with a non-zero imaginary part; that is,
$\lambda_1 = \rho e^{i \theta}$ for $\rho>0$ and $0<\theta < \pi$. Then,
\[
\mathbb{E}\langle x_{k+1}-x,v_l \rangle = C \rho^k \cos(k \theta + \theta_0),
\]
where $C$ and $\theta_0$ are constants that depend on $r,\zeta,\beta,v_l,x,$ and $x_0.$
\end{corollary}

The proof of Corollary \ref{corarg} is given in \S \ref{proofofcorarg}. This corollary is further discussed and illustrated with numerical examples in 
\S \ref{numericscomplex} and \S  \ref{periodicspikingbehavior}. Before presenting these numerical examples, we emphasize the limitations of our analysis.

\begin{remark}[Limitations] \label{limitations}
We emphasize that Theorem \ref{thm1} and its corollaries provide formulas for the expected value of the signed quantity $\langle x_k - x, v_l \rangle$. If the sign of this quantity was known to remain constant $\sign \langle x_k - x, v_l \rangle = \sign \langle x_0 - x, v_l \rangle$ for some range of $k$, then these results become effective in describing the error $|\langle x_{k+1} - x, v_l \rangle|$ of $x_k$ to the solution $x$ in direction $v_l$ on that range of $k$. Informally speaking, the sign remaining constant corresponds to $x_k$ not overshooting the solution $x$ in the direction $v_l$ and keeping a constant approach direction. On the other hand, if $\sign \langle x_k - x, v_l \rangle$ is not determined by $\sign \langle x_0 - x, v_l \rangle$, then these results may or may not be effective in describing the error $|\langle x_{k+1} - x, v_l \rangle|$. For example, the $\sign \langle x_k - x, v_l \rangle$ could oscillate between positive and negative values or have an equal chance of becoming a large negative or positive value. One possible approach to understanding this issue would be to establish an absolute convergence result or an analog of \cite[Theorem 1.3]{Steinerberger2021}; see \S \ref{discussion} for further discussion.
\end{remark}

In the following section, we present various numerical examples that illustrate the dynamics of KGSM, including cases where the dynamics are explained by Theorem \ref{thm1} and cases that raise questions, which we discuss in \S \ref{discussion}.

\section{Numerical Examples} \label{numerics}
This section is organized as follows: in \S \ref{numericsprelim}, we introduce notation; in \S \ref{basicex}, we illustrate Corollary \ref{coropt} for a system with one small singular value; in \S \ref{numericscomplex}, we illustrate Corollary \ref{corarg} for a system with one small singular value; in \S 
\ref{explorembeta}, we explore the $(M,\beta)$ parameter space; 
in \S \ref{periodicspikingbehavior}, we discuss the periodic spiking behavior that the error in the direction of a singular vector sometimes exhibits; in \S \ref{exlinear}, we consider a system whose singular values decay linearly; and in \S \ref{manybad}, we consider a system with many small singular values.

Additional examples are included in Appendix \ref{furtherexamplesappendix}. In particular, \S \ref{gaussiansystem} presents an elementary example involving a Gaussian linear system; this example includes detailed pseudocode to make it simple to implement in any numerical programming environment. In \S \ref{heavyballmomentumbatchsize1fails},  \S \ref{additionalspectraldecays}, \S \ref{additionalbetaplots}, we provide supplementary numerical examples referenced throughout the text.

\subsection{Notation and preliminaries} \label{numericsprelim}
Suppose that $A$ is an $m \times n$ matrix with $m \ge n$. Let $\sigma_1 \ge \cdots \ge \sigma_n$ denote the singular vectors of $A$. Using this notation, the Frobenius norm $\|A\|_F$ can be expressed by
$$
\|A\|_F = \sqrt{\sum_{j=1}^n \sigma_j^2}.
$$
In the following, we generate random $m \times n$ matrices with specified singular values  $\sigma = (\sigma_1,\ldots,\sigma_n)$ as follows. First, we construct a random $m \times n$ matrix $U$ with orthonormal columns by starting with an $m \times n$ matrix $G$ with random independent standard normal entries and applying the Gram-Schmidt procedure to obtain an orthonormal basis for the column space. This construction guarantees that the column space of $U$ is sampled uniformly (with respect to the Haar measure) on the Grassmann manifold ${\bf{Gr}}_{n}(\mathbb{R}^m)$ of all $n$-dimensional subspaces of $\mathbb{R}^m$, see \cite[\S 5.2.6]{vershynin2018high}. Second, using the same procedure, we choose a random $n \times n$ matrix $V$ with orthonormal columns and set
$$
A := U \diag(\sigma) V,
$$
where $\diag(\sigma)$ is the $n \times n$ matrix whose $i$-th diagonal entry is $\sigma_i$, and which is zero off the diagonal.

\subsection{Basic Example} \label{basicex}
We generate a matrix with only one small singular value to create a simple example illustrating Corollary \ref{coropt}. More precisely, using the procedure described in \S \ref{numericsprelim}, we generate a random $100 \times 20$ matrix $A$ whose singular values are
$$
\sigma_1 = \cdots = \sigma_{19} = 1 ,\quad \text{and} \quad \sigma_{20} = 1/50.
$$
We choose a random solution vector $x \in \mathbb{R}^n$ (with independent standard normal entries) and define $b:= Ax$. We set
\begin{equation} \label{basicexparam}
M = 0.9, \quad \text{and} \quad \beta = 1 - \frac{\eta_{20}}{(1 - \sqrt{M})^2},
\end{equation}
where $\eta_{20} := \sigma_{20}^2/\|A\|_F^2$, see Remark \ref{settingM} for a discussion about setting the momentum parameter $M$. 

We choose a random initial vector $x_0$ (with independent standard normal entries) and run randomized Kaczmarz \eqref{kaczmarz} and KGSM \eqref{eq:our-method2}. At each iterate,  we compute the absolute error in the direction of the smallest right singular vector $|\langle x - x_k, v_{20} \rangle|$, see Figure \ref{fig02}. For comparison, we plot the theoretical estimates for 
$|\mathbb{E} \langle x_k - x, v_{20} \rangle|$ from \eqref{kaczmarzstef} and Corollary \ref{coropt}, respectively.
\begin{figure}[h!]
\centering
\includegraphics[width=.65\textwidth]{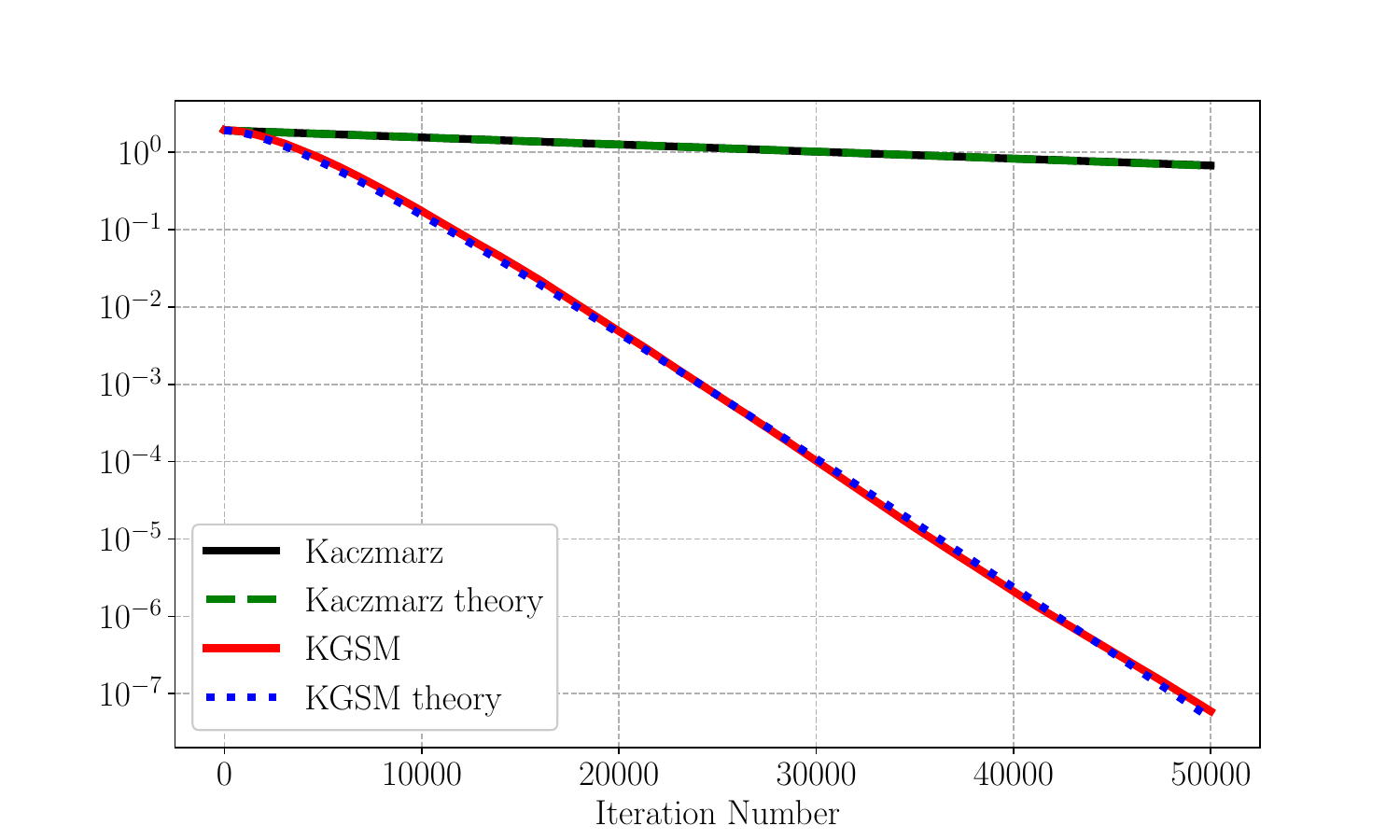}
\caption{The numerical error $|\langle x_k -x,v_{20}\rangle|$ 
for randomized Kaczmarz  \eqref{kaczmarz} and KGSM \eqref{eq:our-method2}, and 
the theoretical estimates for $|\mathbb{E} \langle x_k - x, v_{20} \rangle|$
from  \eqref{kaczmarzstef} and Corollary \ref{coropt}, for the example of \S \ref{basicex}.}
\label{fig02}
\end{figure}

The numerical results in Figure \ref{fig02} are an example of where the theoretical results of 
 \eqref{kaczmarzstef} and Corollary \ref{coropt} for the expected signed error $\mathbb{E} \langle x_k - x, v_{20}\rangle$ provide an accurate model for the absolute numerical error $|\langle x_k - x, v_{20} \rangle|$. Moreover, in this case, the addition of geometrically smoothed momentum drastically increases the convergence rate.

\subsection{Complex pertubation} \label{numericscomplex}
 Recall that from \eqref{eqeigenvalues} that
\begin{equation} \label{eql1d}
\lambda_1 := \frac{r + \beta + \sqrt{(r-\beta)^2 - 4 \zeta}}{2}
\end{equation}
is the largest magnitude eigenvalue of the $2 \times 2$ matrix \eqref{eqB} 
from Theorem \ref{thm1}, where $r = 1 - \sigma_l^2/\|A \|_F^2+M(1-\beta)$ 
and $\zeta = M(1-\beta)^2$. The formula for $\beta$ from Corollary \ref{coropt}
$$
\beta = 1 - \frac{\eta_l}{(1-\sqrt{M})^2},
$$
is the curve where the discriminant $D=(r-\beta)^2 - 4 \zeta$ in \eqref{eql1d} vanishes. 
Above the curve, the eigenvalues are complex, while below the curve, there are two distinct real eigenvalues.

In the following, we repeat the same experiment as in \S \ref{basicex}, but perturb $\beta$ by $0.001$ into the region where the eigenvalues are complex; that is, we set
$$
M = 0.9 \quad \text{and} \quad \beta = 1 - \frac{\eta_{20}}{(1 - \sqrt{M})^2} + 0.001.
$$
In this case, the theoretical estimate for
$|\mathbb{E} \langle x - x_k, v_{20} \rangle|$ for KGSM is computed using 
Theorem \ref{thm1}, and the complex eigenvalues of the matrix in \eqref{maineq} predict
oscillatory behavior, as explained in  Corollary \ref{corarg}, where the error dips at regular intervals,
see Figure \ref{fig03}.

\begin{figure}[h!]
    \centering
    \includegraphics[width=.65\textwidth]{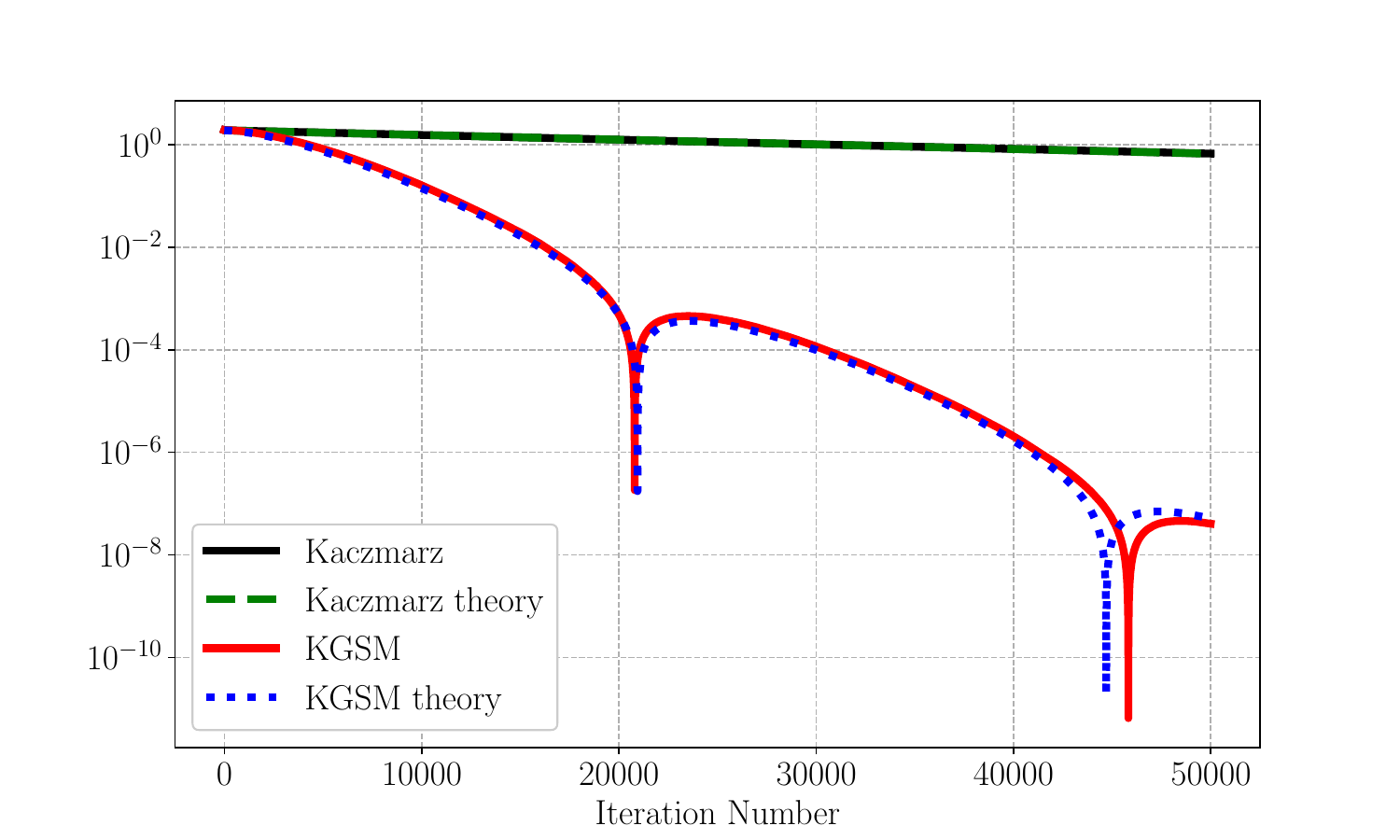}
    \caption{The numerical error $|\langle x_k -x,v_{20}\rangle|$ 
for randomized Kaczmarz 
\eqref{kaczmarz} and KGSM \eqref{eq:our-method2}, and
the theoretical estimates for $|\mathbb{E} \langle x_k - x, v_{20} \rangle|$ from  \eqref{kaczmarzstef} and Theorem \ref{thm1} for the example in \S \ref{numericscomplex}}
\label{fig03}
\end{figure}

Informally speaking, Figure \ref{fig03} illustrates a case where too much momentum is built in a given direction. 
 We further explore this periodic spiking behavior of the error in \S \ref{periodicspikingbehavior}, where we show that the spikes in the numerical error correspond to changes in $\sign \langle x_k - x, v_n \rangle$. 

\subsection{Exploring the $(M,\beta)$ parameter space} \label{explorembeta}
This section further explores the $(M,\beta)$ parameter space. For the same linear system described in \S \ref{basicex}, we run KGSM for four different values of $(M,\beta)$; in particular, we consider
\begin{equation} \label{parameterspace}
(M,\beta) =  (0.85,0.992), (0.9,0.992),\, (0.95,0.992), \text{ and }(0.965,0.932),
\end{equation}
see Figure \ref{fig04}.
\begin{figure}[h!]
\centering
\includegraphics[width = .6\textwidth]{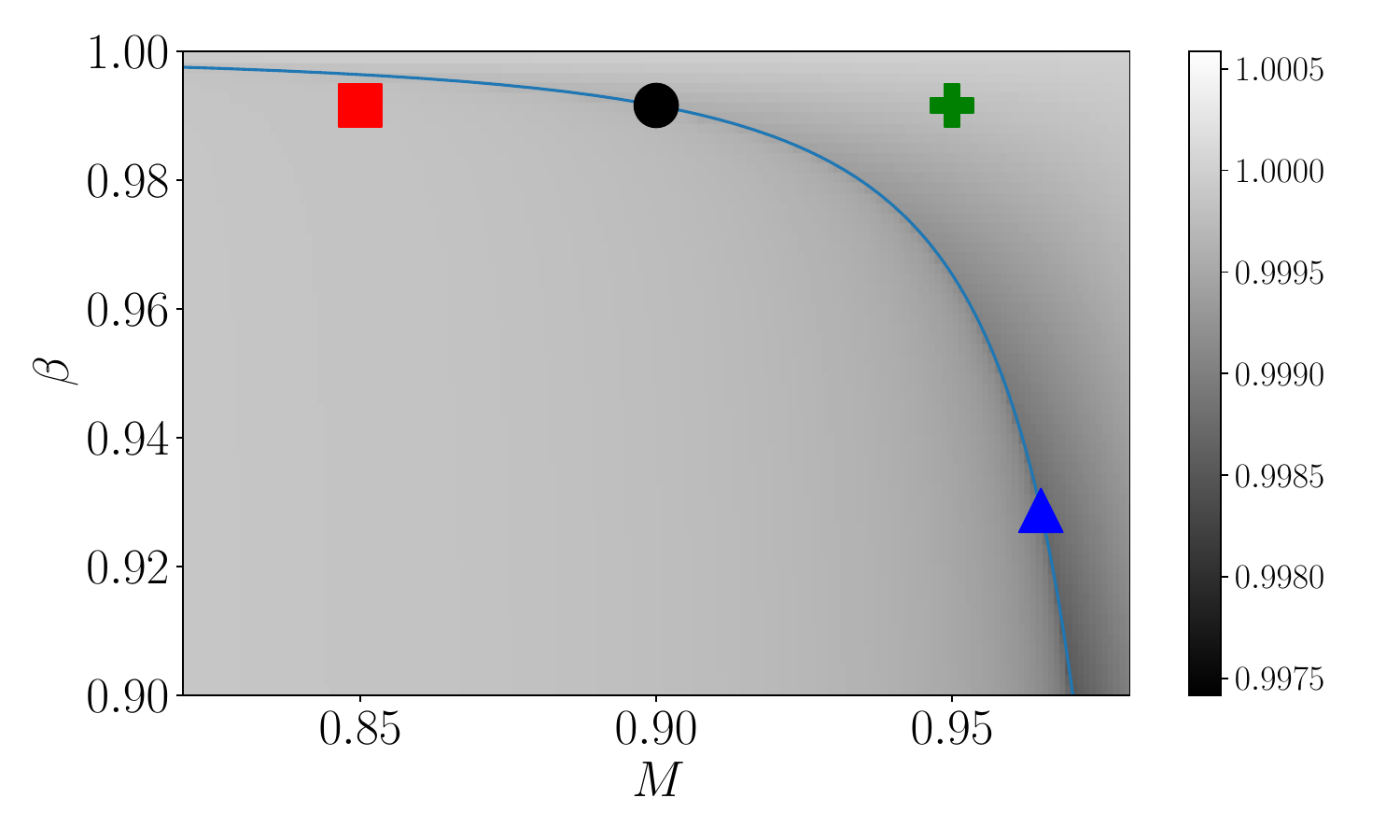}
\caption{Visualization of the values of $(M,\beta)$ from \eqref{parameterspace}. The curve  $\beta = 1 - \eta_{20}/(1 - \sqrt{M})^2$ is plotted for reference in blue.} \label{fig04}
\end{figure}

For each of the four values of $(M,\beta)$ illustrated in Figure \ref{fig04}, we run KGSM \eqref{eq:our-method2}, compute the absolute error in the direction of the smallest right singular vector $|\langle x - x_k, v_{20} \rangle|$, and plot the results in Figure \ref{fig05}.

To interpret the results in Figures \ref{fig05}, recall that the curve $\beta = 1 - \eta_{20}/(1 - \sqrt{M})^2$ in Figure \ref{fig04} divides the $(M,\beta)$ parameter space: above the curve, the eigenvalues of the $2 \times 2$ matrix in Theorem \ref{thm1} are complex; below the curve, the eigenvalues are distinct real values; and on the curve, the eigenvalues are repeated real values, see \S \ref{numericscomplex}.

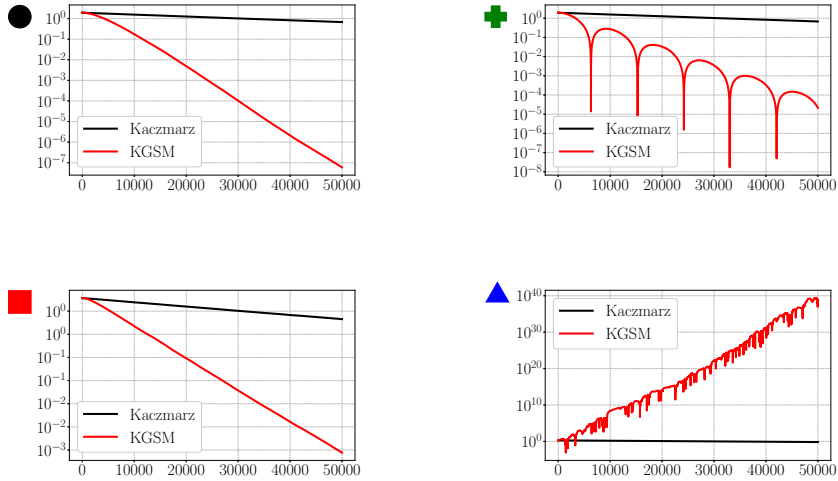
\begin{figure}[h!]
\centering
\begin{tikzpicture}
\node[anchor=north west] at (0,0) {\includegraphics[width=0.355\textwidth,trim=30 0 30 0]{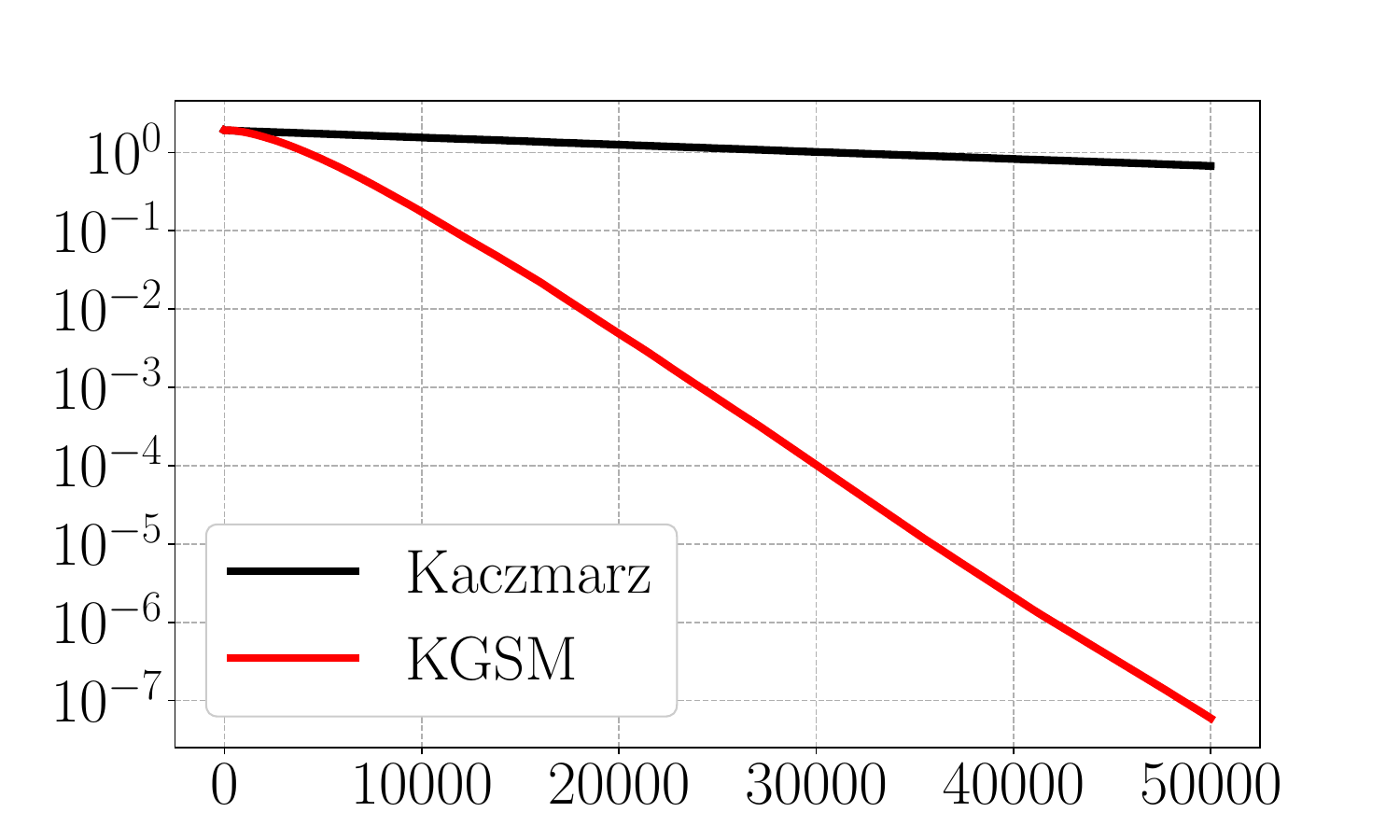}};
\node[anchor=north west] at (0.5\textwidth,0) {\includegraphics[width=0.355\textwidth,trim=30 0 30 0]{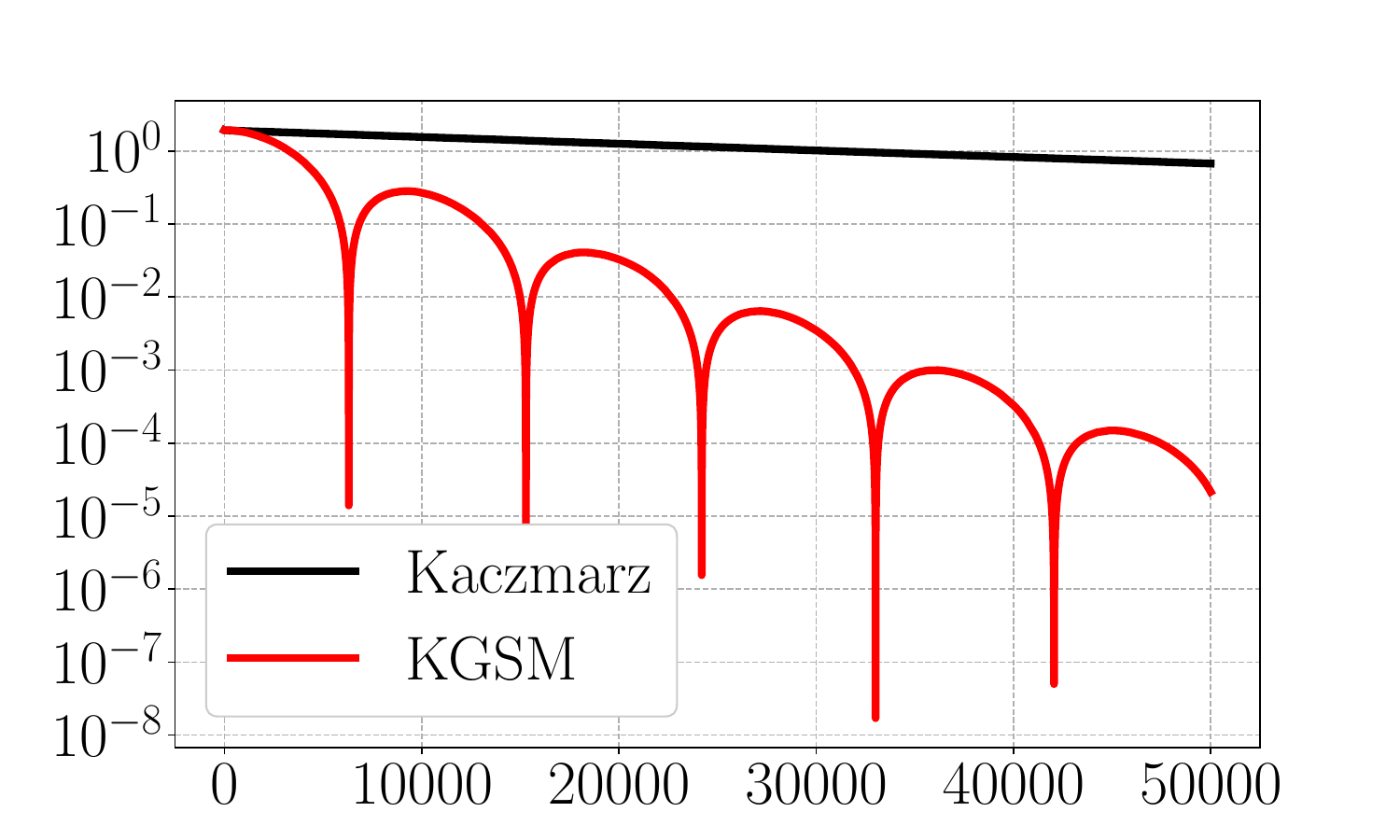}};
\node[anchor=north west] at (0,-0.3\textwidth) {\includegraphics[width=0.355\textwidth,trim=30 0 30 0]{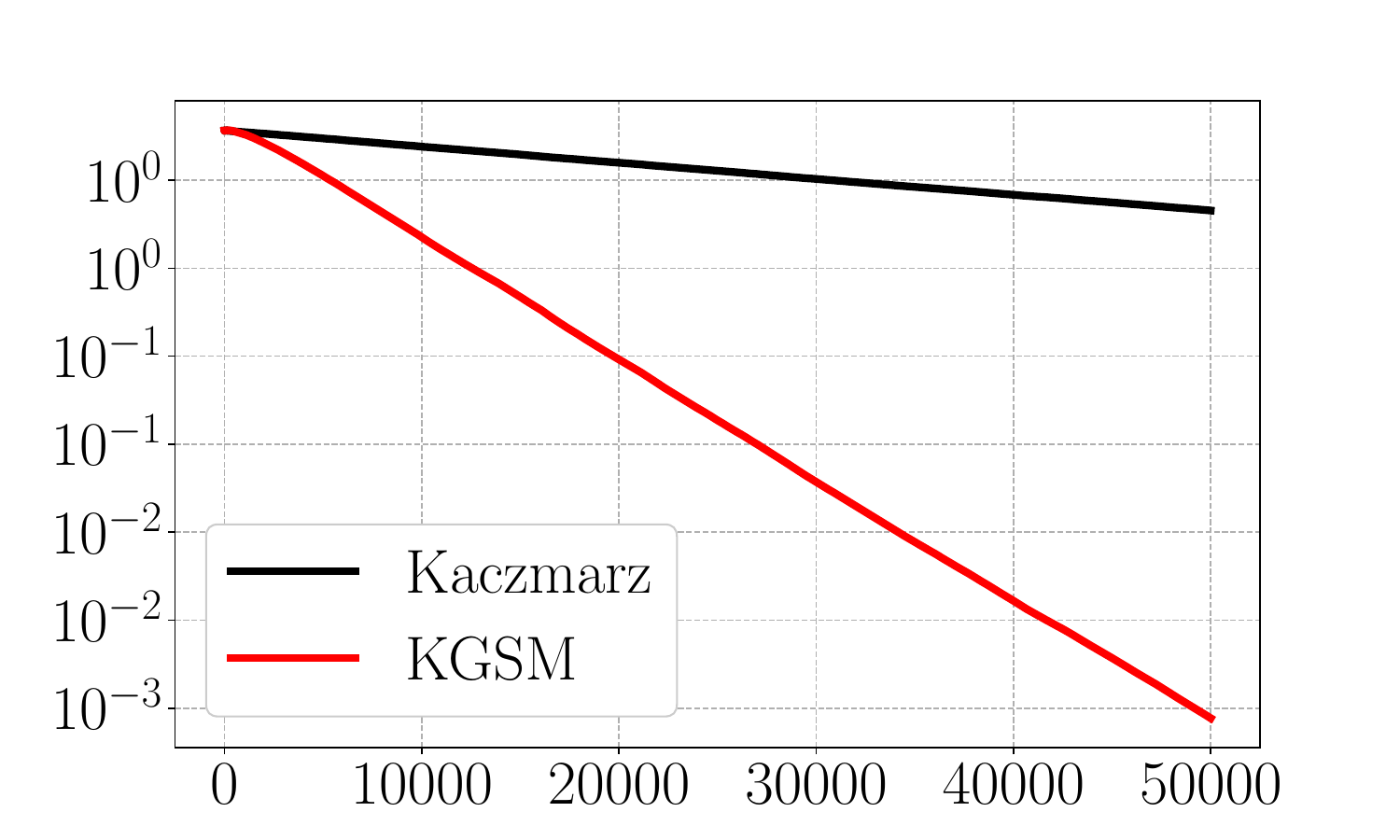}};
\node[anchor=north west] at (0.5\textwidth,-0.3\textwidth) {\includegraphics[width=0.355\textwidth,trim=30 0 30 0]{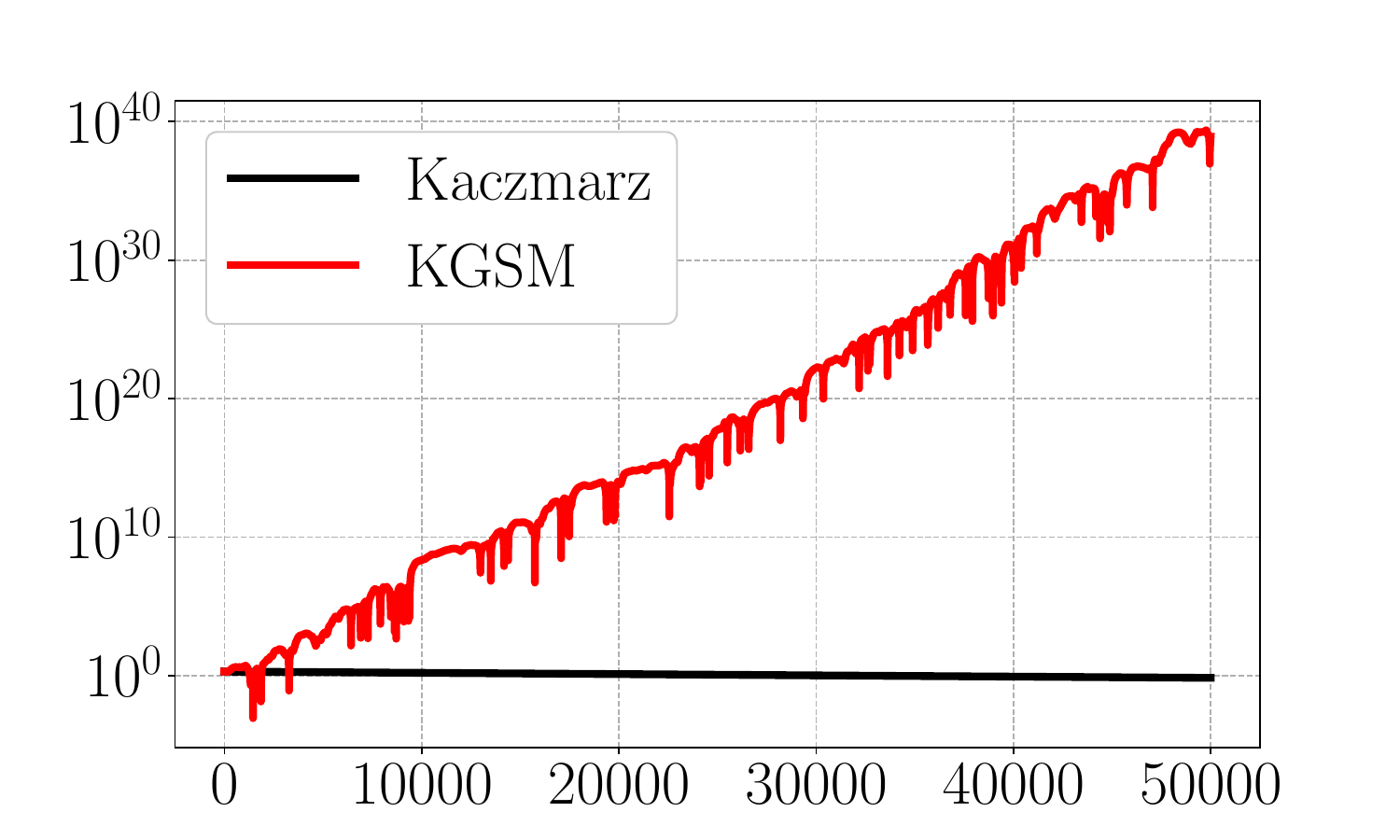}};

\node[draw=black, circle, fill=black, inner sep=0pt, minimum size=3mm] at (-0.01\textwidth, -0.05\textwidth) {};

\definecolor{mygreen}{HTML}{008000}
\begin{scope}[shift={(0.49\textwidth, -0.05\textwidth)}]
    \draw [mygreen, line width=1.3mm] (-0.15,0) -- (0.15,0); 
    \draw [mygreen, line width=1.3mm] (0,-0.15) -- (0,0.15); 
\end{scope}

\definecolor{myred}{HTML}{FF0000}
\node[draw=myred, rectangle, fill=myred, inner sep=0pt, minimum size=3mm] at (-0.01\textwidth, -0.35\textwidth) {};

\definecolor{myblue}{HTML}{0000FF}
\draw[draw=myblue, fill=myblue] (0.48\textwidth, -0.35\textwidth) -- ++(60:3mm) -- ++(-60:3mm) -- cycle;
\end{tikzpicture}
\caption{The error $|\langle x_k -x,v_{20}\rangle|$  for the randomized Kaczmarz \eqref{kaczmarz} and KGSM \eqref{eq:our-method2} for parameters $(M,\beta)$ indicated by markers labeling each plot, which correspond to the markers in Figure \ref{fig04}. For corresponding plots of $\ell^2$-norm error
$\|x_k - x \|_2$ and $|\langle x_k -x, v_{19} \rangle|$, see  \S \ref{additionalbetaplots}.}
\label{fig05}
\end{figure}

In the following, we briefly discuss the plot corresponding to each marker:
\subsubsection*{Black circle}
The parameters denoted by the black circle marker are those from 
\eqref{basicexparam} previously discussed in
\S \ref{basicex} and are included in Figure \ref{fig05} to provide context.

\subsubsection*{Red square} The parameters denoted by the red square marker are in the region where the eigenvalues $\lambda_1$ and $\lambda_2$ are real. The momentum parameter $M$ is smaller compared to the black circle marker, and as a result, the rate of convergence of KGSM is relatively slower; see Figure \ref{fig05}.

\subsubsection*{Green plus} The parameters denoted by the green plus marker are in the region where the eigenvalues $\lambda_1$ and $\lambda_2$ are complex, which results in oscillatory behavior; see Figure \ref{fig05} and the discussion in \S \ref{numericscomplex}. 

\subsubsection*{Blue triangle} The parameters denoted by the blue triangle marker are on the curve $\beta = 1 - \eta_{20}/(1 - \sqrt{M})^2$. The momentum parameter $M$ is larger compared to the black circle marker. Here, one might expect that KGSM converges faster relative to the black circle experiment. Instead, the method breaks down, and the iteration diverges; see Figure \ref{fig05}.  The following remark discusses setting the momentum parameter $M$.

\begin{remark}[Setting the momentum parameter $M$] \label{settingM}
Recall that by Corollary  \ref{coropt}, the expected signed error in the direction of the smallest singular vector $\mathbb{E} \langle x - x_k, v_{20} \rangle$ decreases by a factor of $1 - \eta_l/(1-\sqrt{M})$ each iteration. We can make this factor arbitrarily small by setting $M$ close to $(1-\sqrt{\eta_l})^2$. Numerical results indicate that each linear system has a critical value of $M$ beyond which the KGSM iteration fails to converge. Numerically, an effective value of $M$ can be determined by increasing $M$ until the methods diverge or using bisection to identify such a value. We pose several questions related to the critical value of $M$ in \S \ref{discussion}.
\end{remark}

\subsection{Periodic Spiking behavior} \label{periodicspikingbehavior} In this section, we discuss the periodic spiking behavior of the error $|\langle x_k -x, v_l\rangle|$ observed in Figure \ref{fig03} and Figure \ref{fig05}--Green plus. The existence of this periodic spiking behavior is explained by Corollary \ref{corarg} whose formula for $\mathbb{E} \langle x_k -x, v_l\rangle$ exhibits periodic spikes. Informally speaking, this behavior occurs when too much momentum is built in the direction $v_l$ (or $-v_l$), causing the error $|\langle x_k -x, v_l\rangle|$  to rapidly decrease and then subsequently rapidly increase as $x_k$ overshoots the solution. This overshooting behavior can be visualized by plotting the error $|\langle x_k - x, v_l \rangle|$ together with $\sign \langle x_k - x,v_l \rangle$, see Figure \ref{fig_sign}.

\begin{figure}[h!]
\centering
\begin{tabular}{cc}
\includegraphics[width = 0.45\textwidth,trim=30 0 0 0]{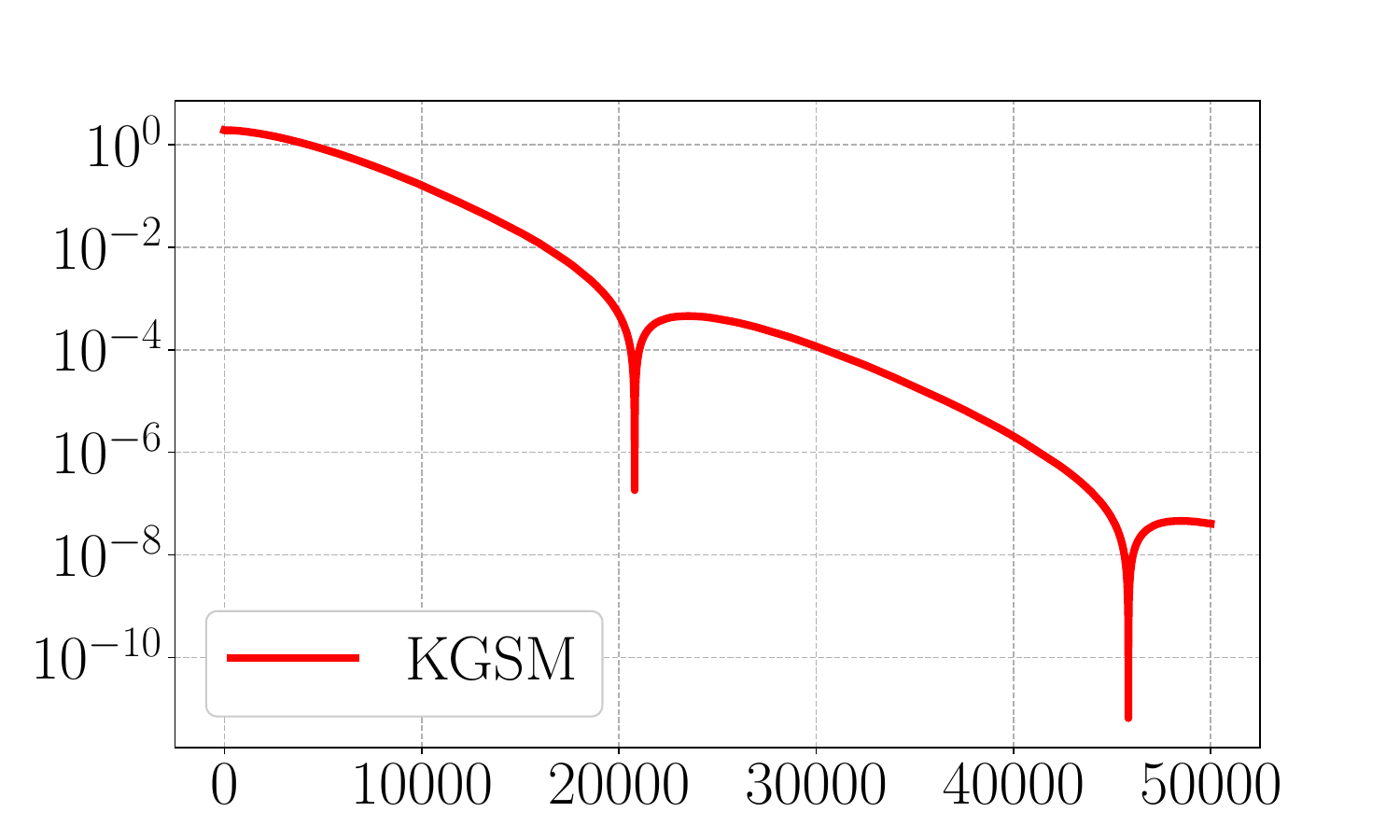}  &
\includegraphics[width = 0.45\textwidth,trim=30 0 0 0]{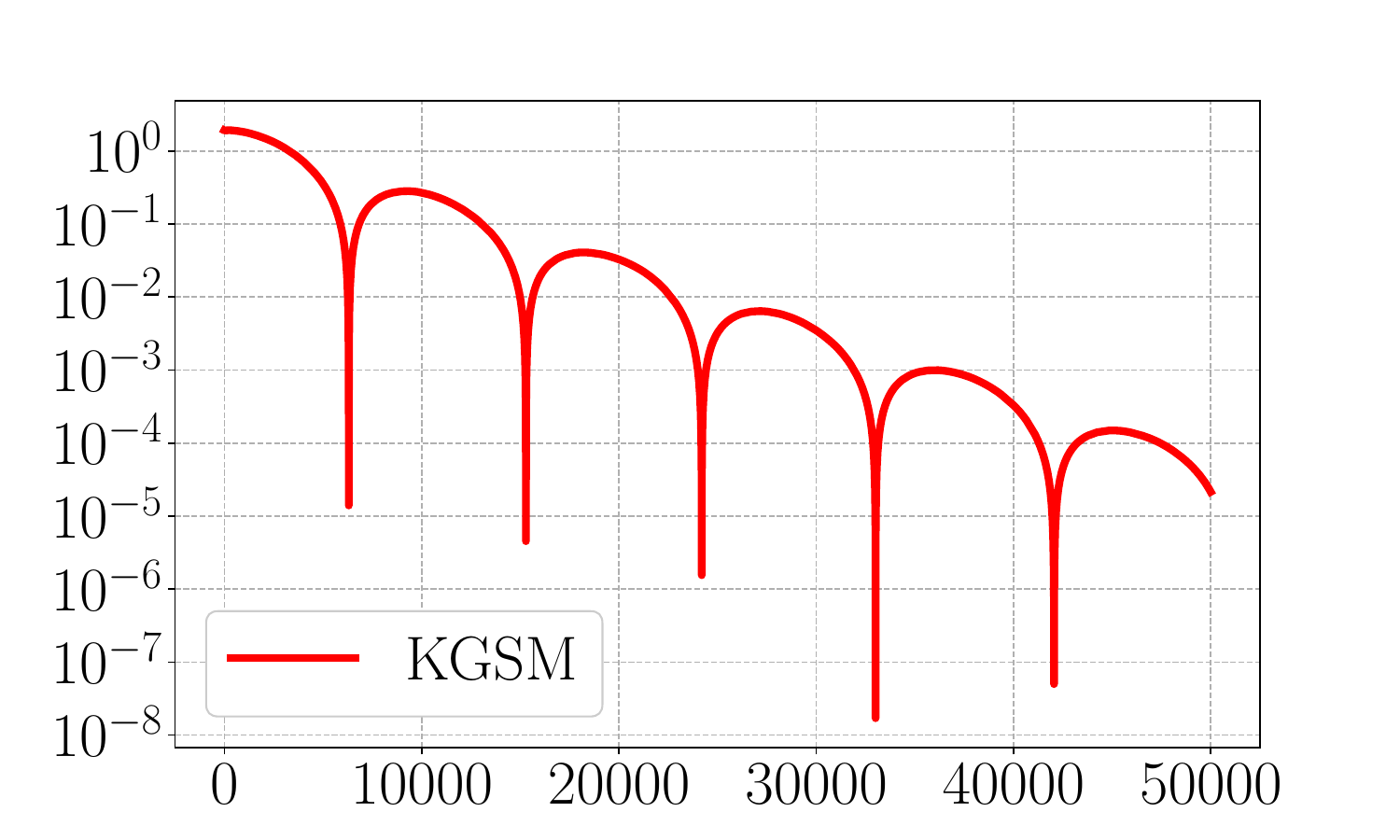} \\
\includegraphics[width = 0.45\textwidth,trim=30 0 0 0]{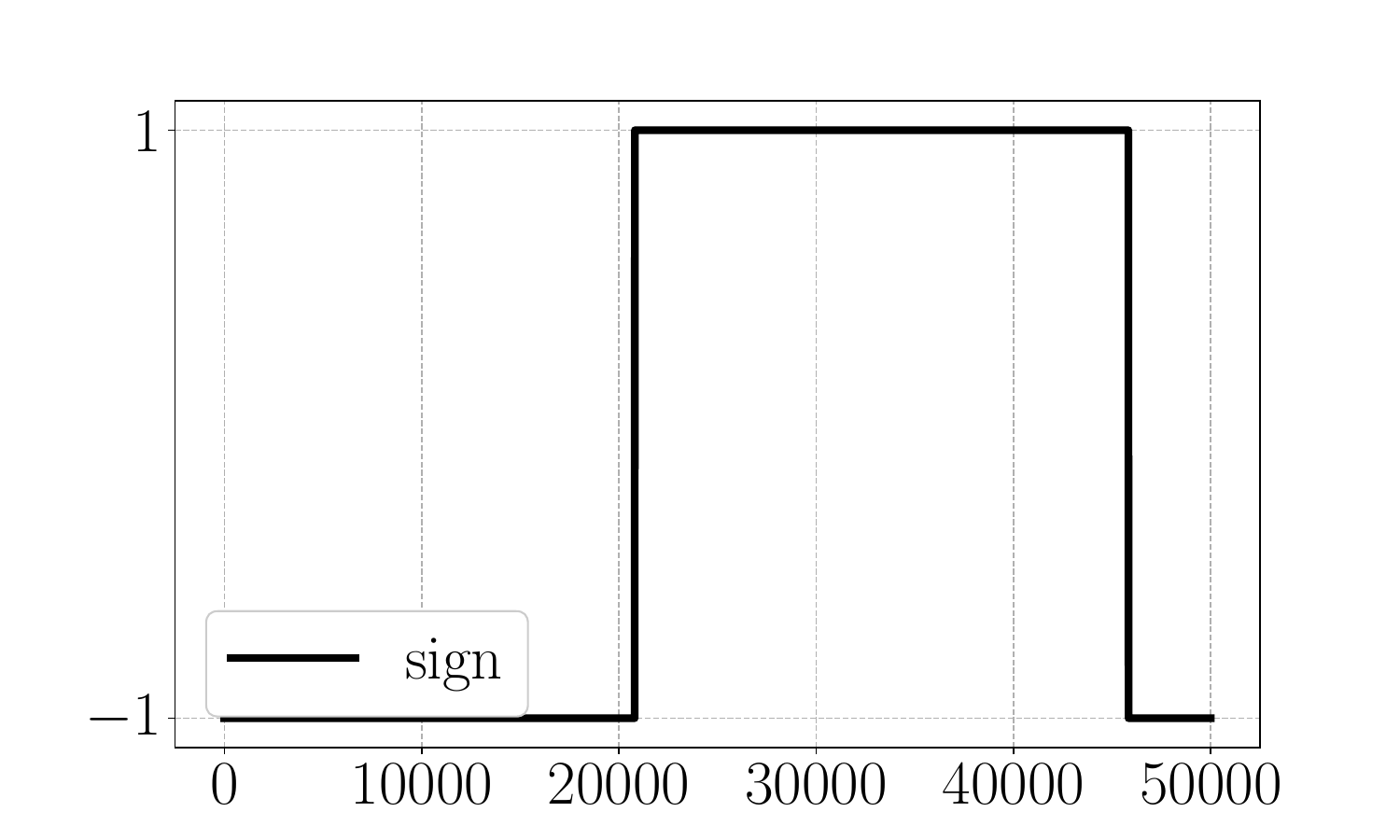} &
\includegraphics[width = 0.45\textwidth,trim=30 0 0 0]{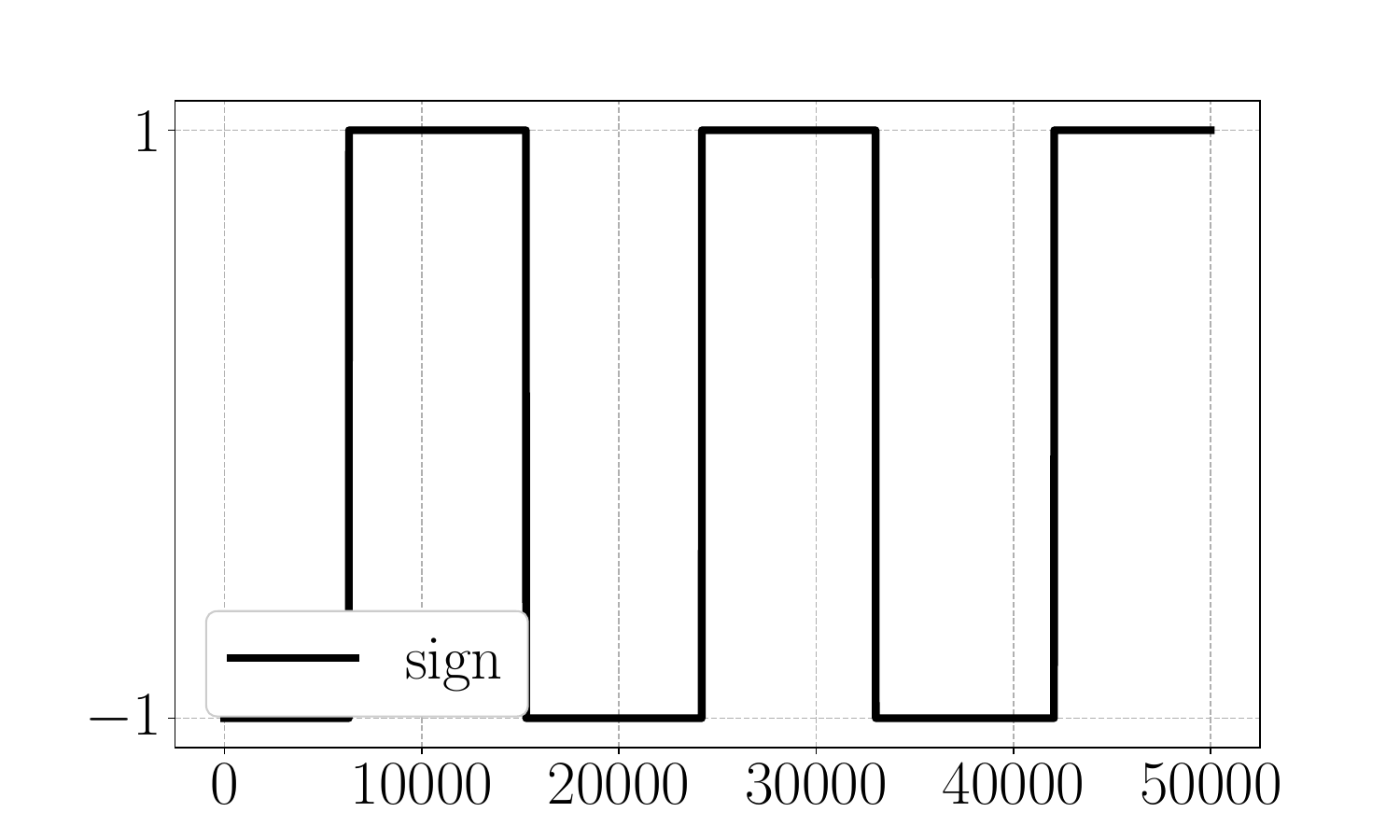} \\
\end{tabular}
\caption{ On the left, we plot  $| \langle x_k - x, v_{20} \rangle|$  for KGSM from the example of Figure \ref{fig03}, and the corresponding value of $\sign  \langle x_k - x, v_{20} \rangle$ . On the right, we plot 
$| \langle x_k - x, v_{20} \rangle|$ for KGSM on the example of Figure \ref{fig05}--Green plus, and the corresponding value of $\sign  \langle x_k - x, v_{20} \rangle$. }
\label{fig_sign}
\end{figure}

Observe that in each example illustrated in Figure \ref{fig_sign}, spikes 
in $| \langle x_k - x, v_{20} \rangle|$  corresponds to a change in 
$\sign  \langle x_k - x, v_{20} \rangle$. Note that this periodic spiking behavior is predicted by Theorem \ref{thm1}, and in turn, Corollary \ref{corarg}, in expectation. However, determining the locations of the spikes using Corollary \ref{corarg} would require the accurate computation of $\theta$ and $\theta_0$, which may be prohibitively computationally expensive. Additionally, it may be difficult to compute the location of the spikes due to numerical error.  For example, in Figure \ref{fig03}, observe that the KGSM theory curve accurately predicts the first spike but has become less accurate by the second spike. It may be possible and interesting to predict the spikes in the error during run time; see \S \ref{discussion} for further discussion.

\subsection{Linear distribution of singular values} \label{exlinear} So far, we have considered the case where the matrix $A$ only has one small singular value. Here, we consider cases where the singular values decay linearly. In this case, the KGSM method is still effective in accelerating the convergence rate, but the dynamics are no longer precisely predicted by Theorem \ref{thm1}.

Using the procedure described in \S \ref{numericsprelim}, we generate a random $100 \times 20$ matrix $A$ whose singular values are
$$
\sigma_1 = 1 , \ \sigma_2 = \frac{19}{20}, \quad \cdots \quad ,\sigma_{20} = \frac{1}{20}.
$$
We set
$$
M = 0.85, \quad \text{and} \quad \beta = 1 - \frac{\eta_{20}}{(1 - \sqrt{M})^2}.
$$
We choose a random initial vector $x_0$ (with independent standard normal entries) and run randomized Kaczmarz \eqref{kaczmarz} and KGSM \eqref{eq:our-method2}. At each iterate,  we compute the absolute error in the direction of the smallest right singular vector $|\langle x - x_k, v_{20} \rangle|$  and the $\ell^2$-norm error $\|x - x_k\|_2$,
see Figure \ref{fig06}. For comparison, we plot the theoretical estimates for 
$|\mathbb{E} \langle x_k - x, v_{20} \rangle|$ from \eqref{kaczmarzstef} and Corollary \ref{coropt}, respectively.

\begin{figure}[h!]
    \centering
    \includegraphics[width=.45\textwidth,trim=30 0 0 0]{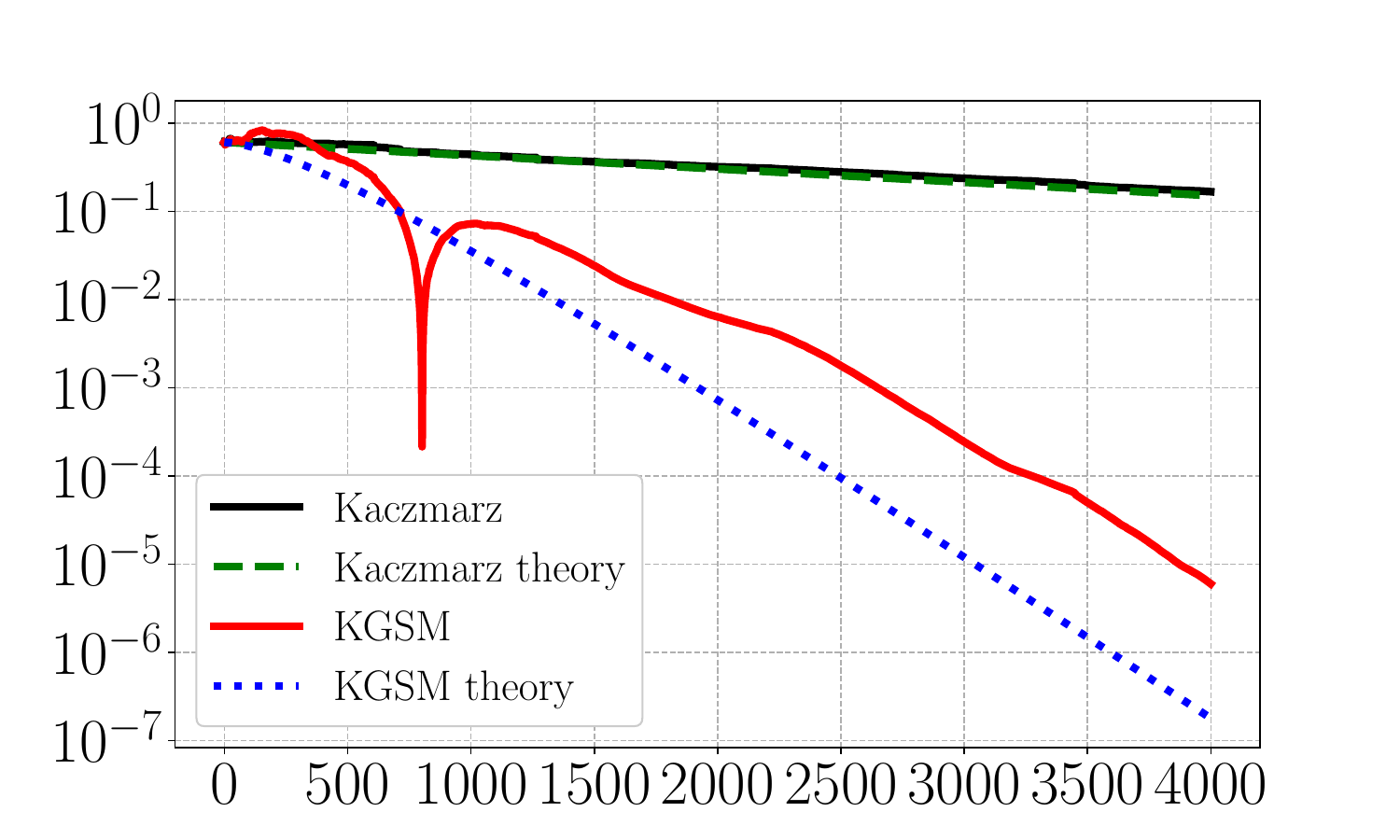}
    \includegraphics[width=.45\textwidth,trim=30 0 0 0]{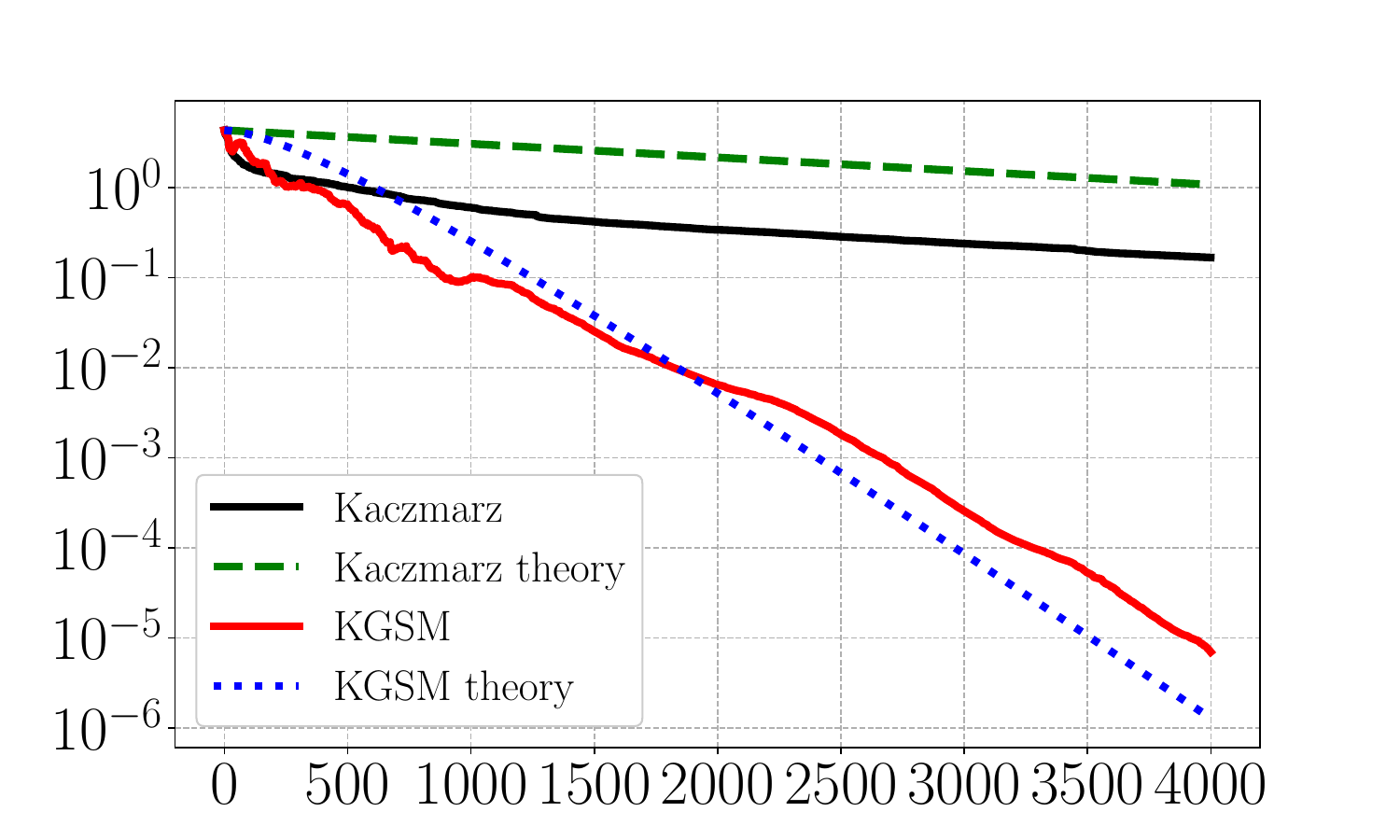}
    \caption{ The numerical error $|\langle x_k -x,v_{20}\rangle|$ (left) and $\|x_k -x\|_2$ (right) 
for randomized Kaczmarz  \eqref{kaczmarz} and KGSM \eqref{eq:our-method2}, and the theoretical estimates for $|\mathbb{E} \langle x_k - x, v_{20} \rangle|$
from  \eqref{kaczmarzstef} and Corollary \ref{coropt}, for the example of \S \ref{exlinear}.}
    \label{fig06}
\end{figure}

\begin{figure}[h!]
    \centering
    \includegraphics[width=.45\textwidth,trim=30 0 0 0]{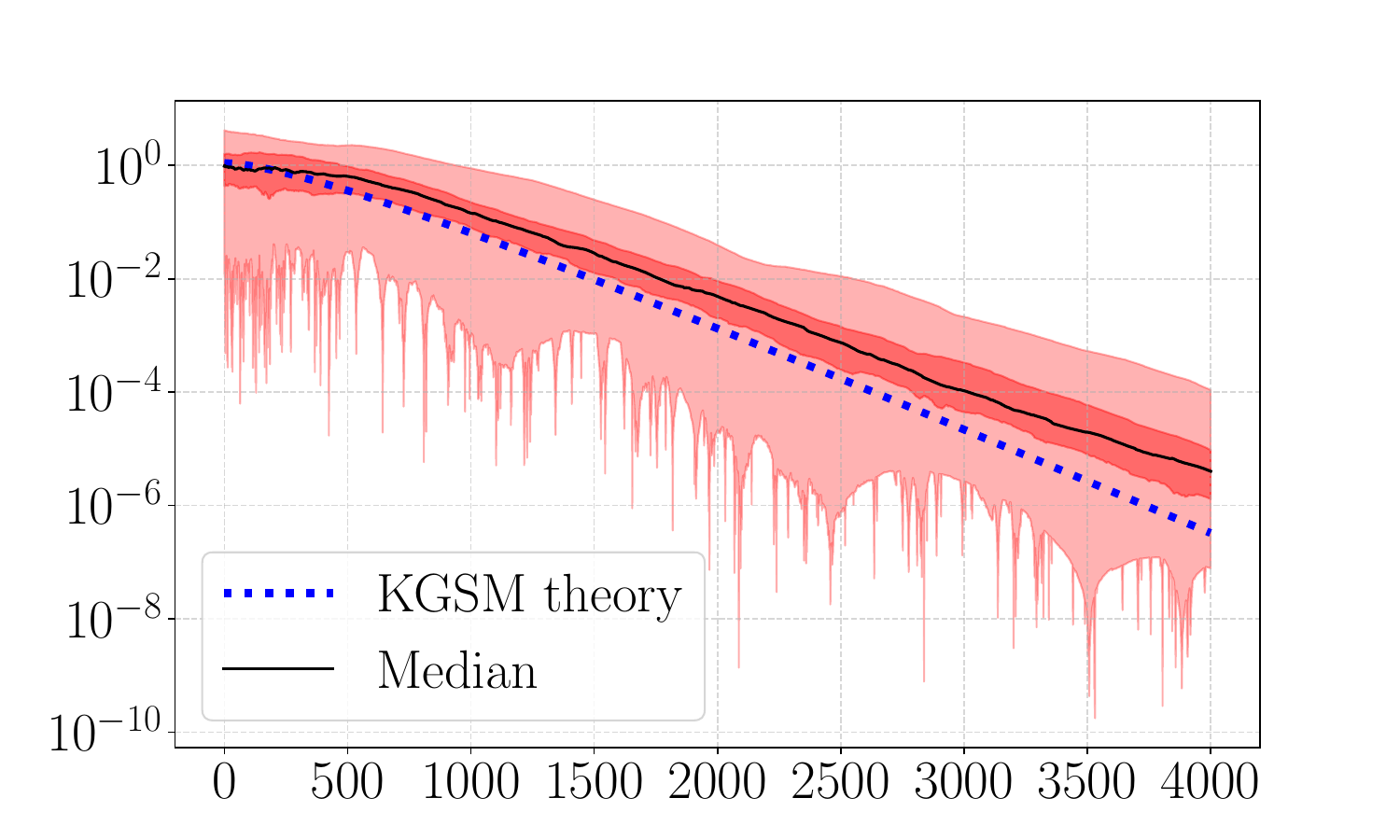}
    \includegraphics[width=.45\textwidth,trim=30 0 0 0]{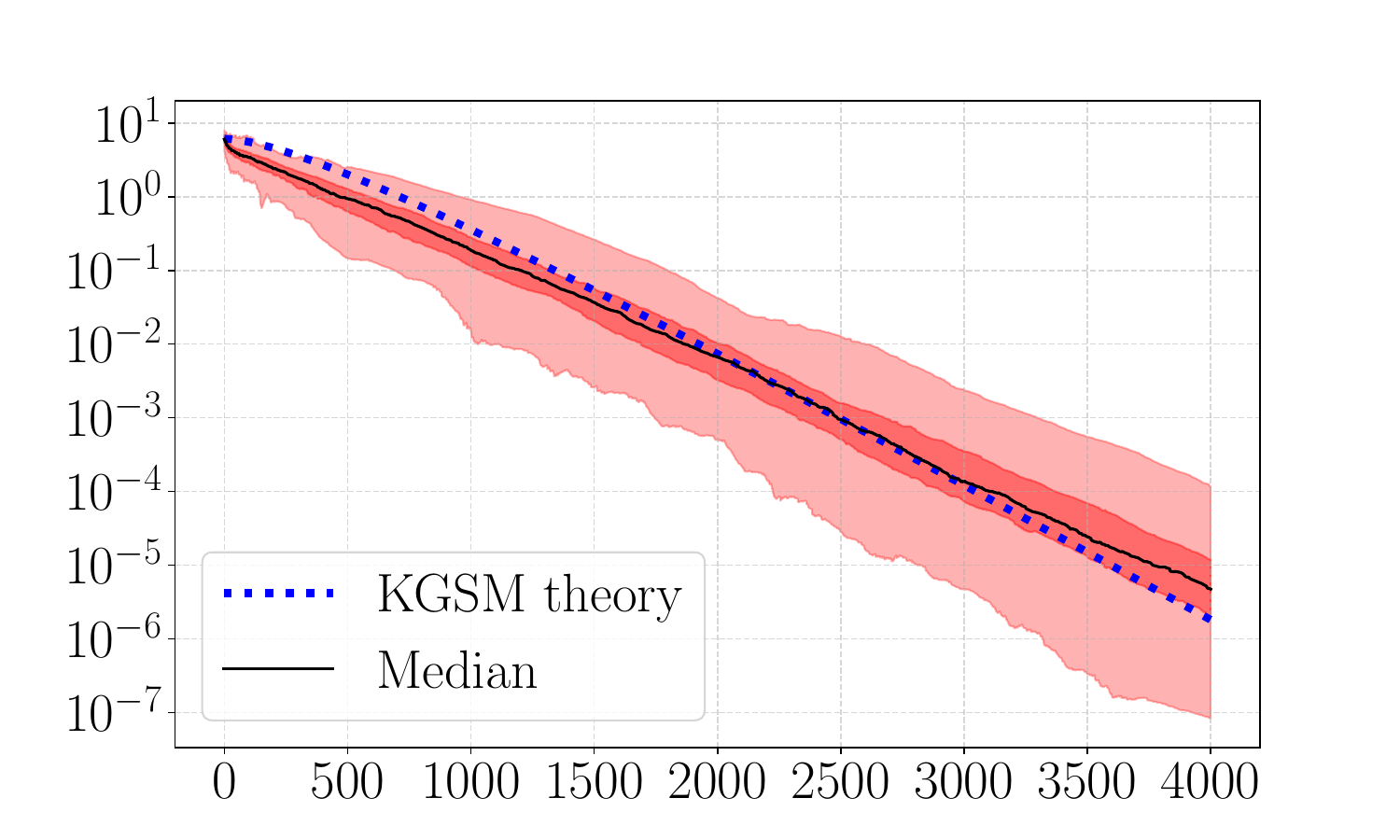}
    \caption{ We plot the 1st, 2nd, 3rd, and 4th quartiles as shaded regions 
   for $|\langle x_k - x, v_{20}\rangle|$ (left) and $\|x-x_k\|_2$ (right)  over $100$ trials of KGSM \eqref{eq:our-method2} on systems described in \S \ref{exlinear}. }
    \label{fig_averaged}
\end{figure}

The error $|\langle x_k -x,v_{20}\rangle|$ in Figure \ref{fig06} exhibits a notable spiking behavior. 
While parameters $(M,\beta)$ are set so that the eigenvalues $\lambda_1$ and $\lambda_2$ are real, the initial behavior of $|\langle x_k -x,v_{20}\rangle|$ is similar to that found in the complex eigenvalue case. In this case, the spiking is not periodic and does not resemble Corollary \ref{corarg}. This indicates that the spiking is due to either a numerical error or some initial behavior of the algorithm. To further explore this phenomenon, we run KGSM $100$ times and plot error quartiles; see Figure \ref{fig_averaged}.

Observe that the 1st quartile in Figure \ref{fig_averaged} shows that spikes in  $|\langle x_k - x, v_{20}\rangle|$ seem to appear more or less randomly. On the other hand, these spikes do not appear in the 4th quartile, indicating that the spikes do not affect the worst-case error.

\subsection{Many small singular values} \label{manybad} 
In the following, we consider the extreme case of a matrix where all of the singular values are small except for one. We repeat the same numerical experiment (with the parameters) described in \S \ref{exlinear} but construct the matrix $A$ to have singular values
$$
\sigma_1 = 1 ,\quad \text{and} \quad \sigma_{2} = \cdots = \sigma_{20} = 1/50.
$$
The results are plotted in Figure \ref{fig07}. 

\begin{figure}[h!]
    \centering
    \includegraphics[width=.45\textwidth,trim=30 0 0 0]{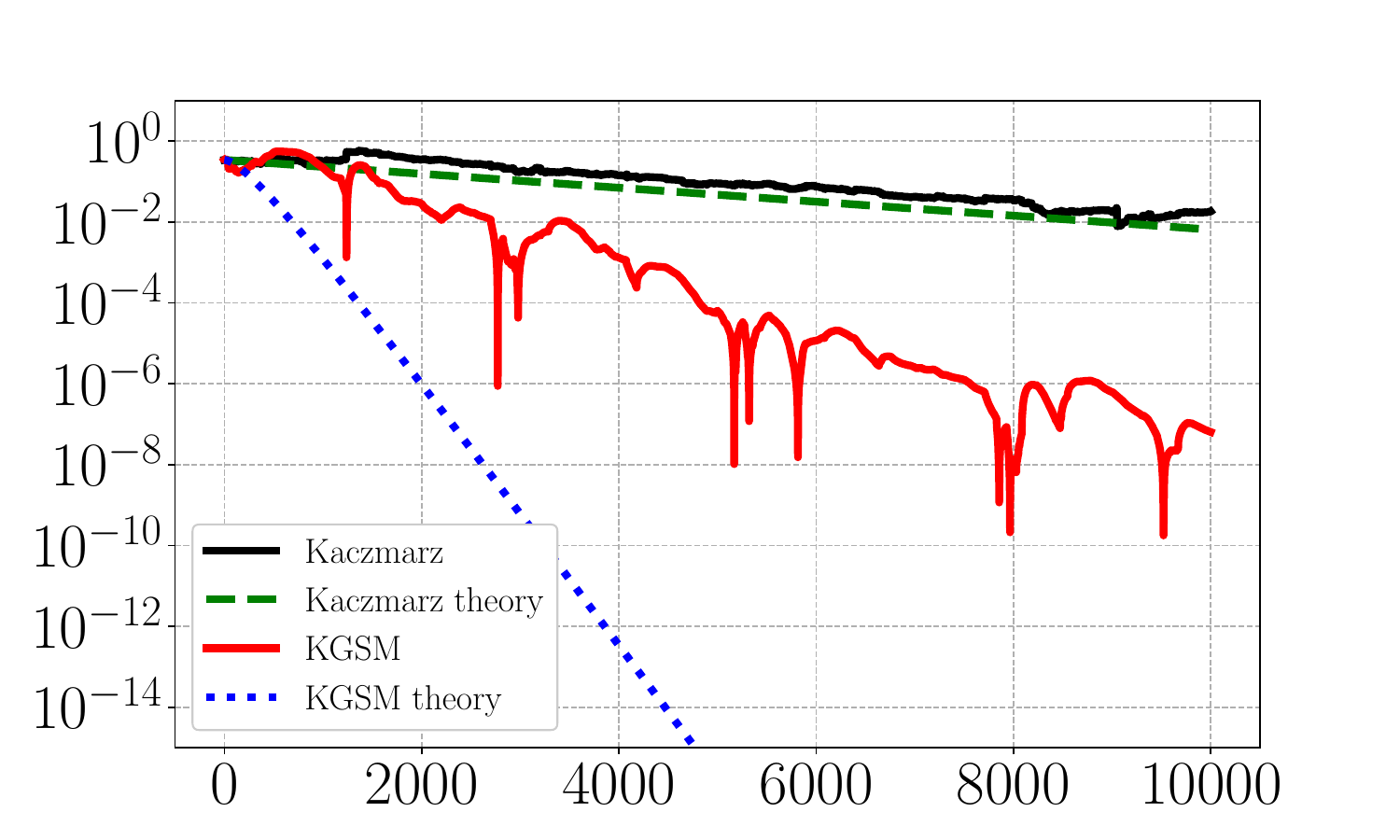}
        \includegraphics[width=.45\textwidth, trim=30 0 0 0]{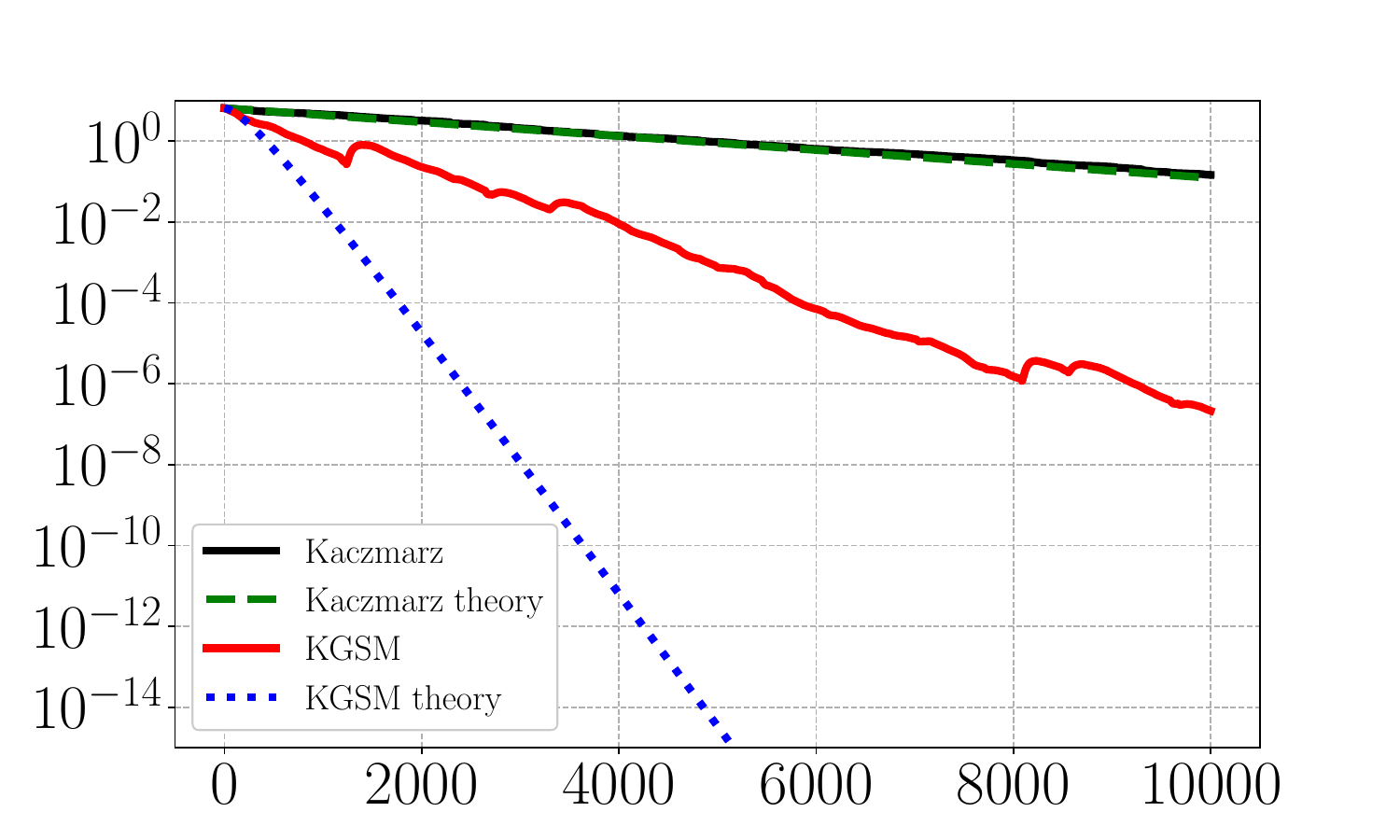}
    \caption{The numerical error $|\langle x_k -x,v_{20}\rangle|$ (left) and $\|x-x_k\|_2$ (right)
for randomized Kaczmarz  \eqref{kaczmarz} and KGSM \eqref{eq:our-method2}, and 
the theoretical estimates for $|\mathbb{E} \langle x_k - x, v_{20} \rangle|$
from  \eqref{kaczmarzstef} and Corollary \ref{coropt}, for the example of \S \ref{manybad}.}
    \label{fig07}
\end{figure}

The numerical example in Figure \ref{fig07} exhibits complicated dynamics. Observe that the theoretical estimate for the signed expected error from Corollary \ref{coropt} no longer seems to match the numerical error $|\langle x - x_k, v_{20} \rangle|$, even for a large number of iterations. Furthermore, the error includes multiple dips similar to the complex dynamics of \S \ref{numericscomplex}, but without a defined period.
This example was designed as a worst-case situation for KGSM.

\begin{remark}[Extending analysis to $\ell^2$-norm error]
Our proof technique does not seem to directly extend to the standard expected $\ell^2$-norm error analysis of the randomized Kaczmarz algorithm \cite{StrohmerVershynin2009}. Figure \ref{fig07} indicates a potential reason why extending the analysis might be challenging: the performance of KGSM seems to depend on the distribution of the singular values of the matrix $A$ rather than just the condition number. That being said, even in this challenging case, KGSM still converges significantly faster compared to the standard Kaczmarz algorithm in the $\ell^2$-norm and along singular vectors associated with the smallest singular value. It could be the case that the expected $\ell^2$-norm error decreases at a similar rate to $|\mathbb{E} \langle x_k -x, v_n \rangle|$ up to a constant for some region of parameters $(M,\beta)$, and that an $\ell^2$-norm error result can be established by identifying the correct way to introduce this constant factor. Another approach to study the $\ell^2$-norm error would be to study the sequence of errors $(\mathbb{E} |\langle x_k - x, v_l \rangle|)_{l=1}^n$ in the direction of all singular vectors simultaneously, which would lead to a detailed understanding of the dynamics, and may be easier than studying $\mathbb{E} |\langle x_k -x, v \rangle|$ for fixed $l$ in isolation. Further discussion of possible directions for extending our results is given in \S \ref{discussion}.
\end{remark}

\subsection{Comparison to $\ell^2$-norm error}
So far, we have been plotting the error 
$|\langle x_k - x, v_n \rangle |$ in the direction of the smallest singular value $v_n$. In this section, we compare 
the singular vector error $|\langle x_k - x, v_n \rangle |$ to the $\ell^2$-norm error $\|x - x_k\|_2$.
We start by considering the numerical example described in \S \ref{numericscomplex} and provide a side-by-side comparison of error in the direction of the smallest singular vector and the $\ell^2$-norm error in Figure \ref{fig08}.

\begin{figure}[h!]
\centering
\begin{tabular}{cc}
\includegraphics[width=.45\textwidth,trim= 30 0 30 0]{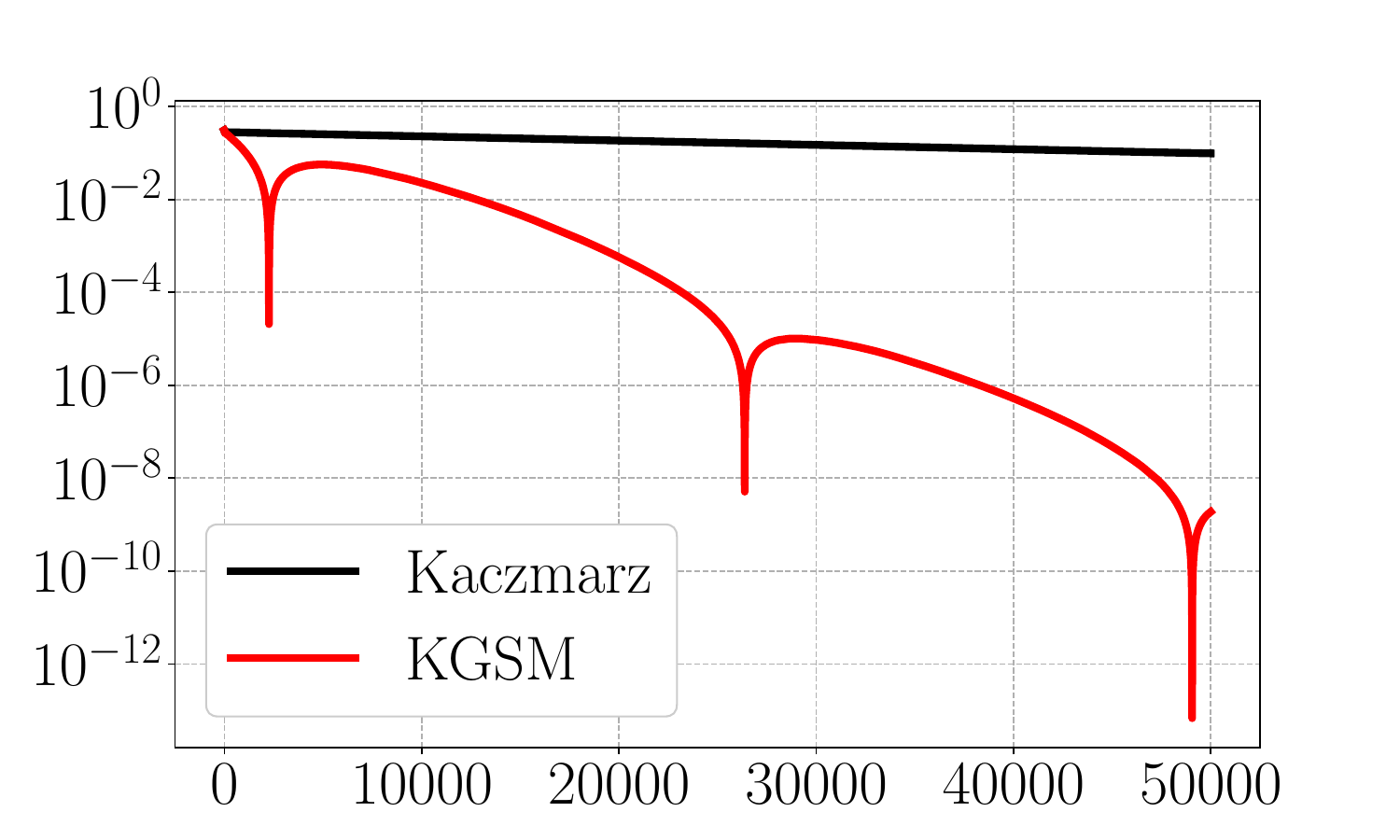} & 
\includegraphics[width=.45\textwidth,trim= 30 0 30 0]{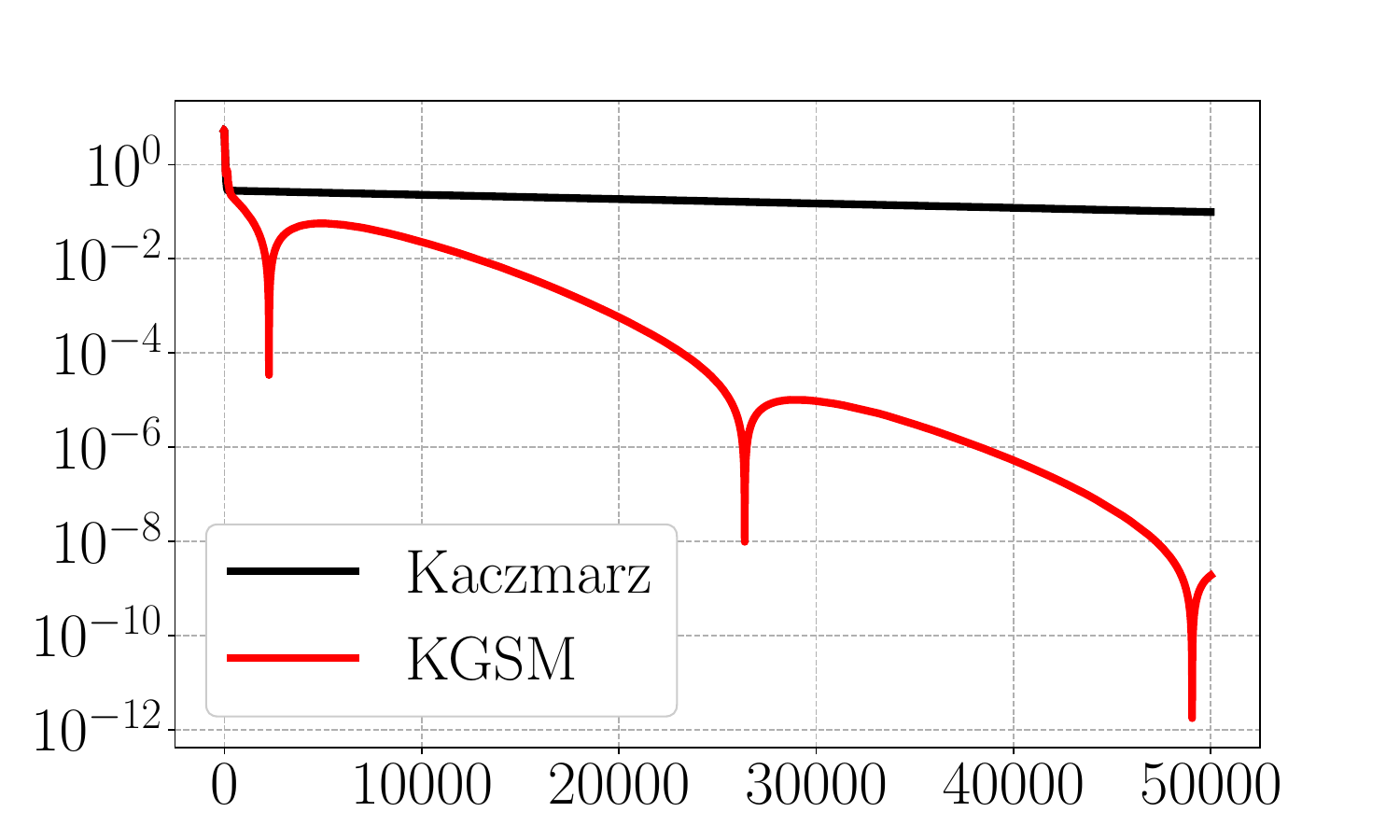} 
\end{tabular}
\caption{The error $|\langle x_k - x, v_{20}\rangle |$   (left) and 
$\|x_k-x\|_2$ (right) for  randomized Kaczmarz \eqref{kaczmarz} and KGSM \eqref{eq:our-method2} for a linear system with one small singular vector.}
\label{fig08}
\end{figure}

In the numerical example of Figure \ref{fig08}, where the underlying matrix only has one small singular value, the square error is roughly a scaled version of the error in the direction of the smallest singular vector. Next, we modify this numerical example by adding another small singular value. More precisely, we repeat the same numerical experiment with the same parameters but now set:
$$
\sigma_1 = \cdots = \sigma_{18} = 1 ,\quad \text{and} \quad \sigma_{19} = \sigma_{20} = 1/50.
$$
We plot the error in the direction of one of the smallest singular vectors in comparison to the $\ell^2$-norm error in Figure \ref{fig09}.

\begin{figure}[h!]
\centering
\begin{tabular}{cc}
\includegraphics[width=.45\textwidth,trim= 30 0 30 0]{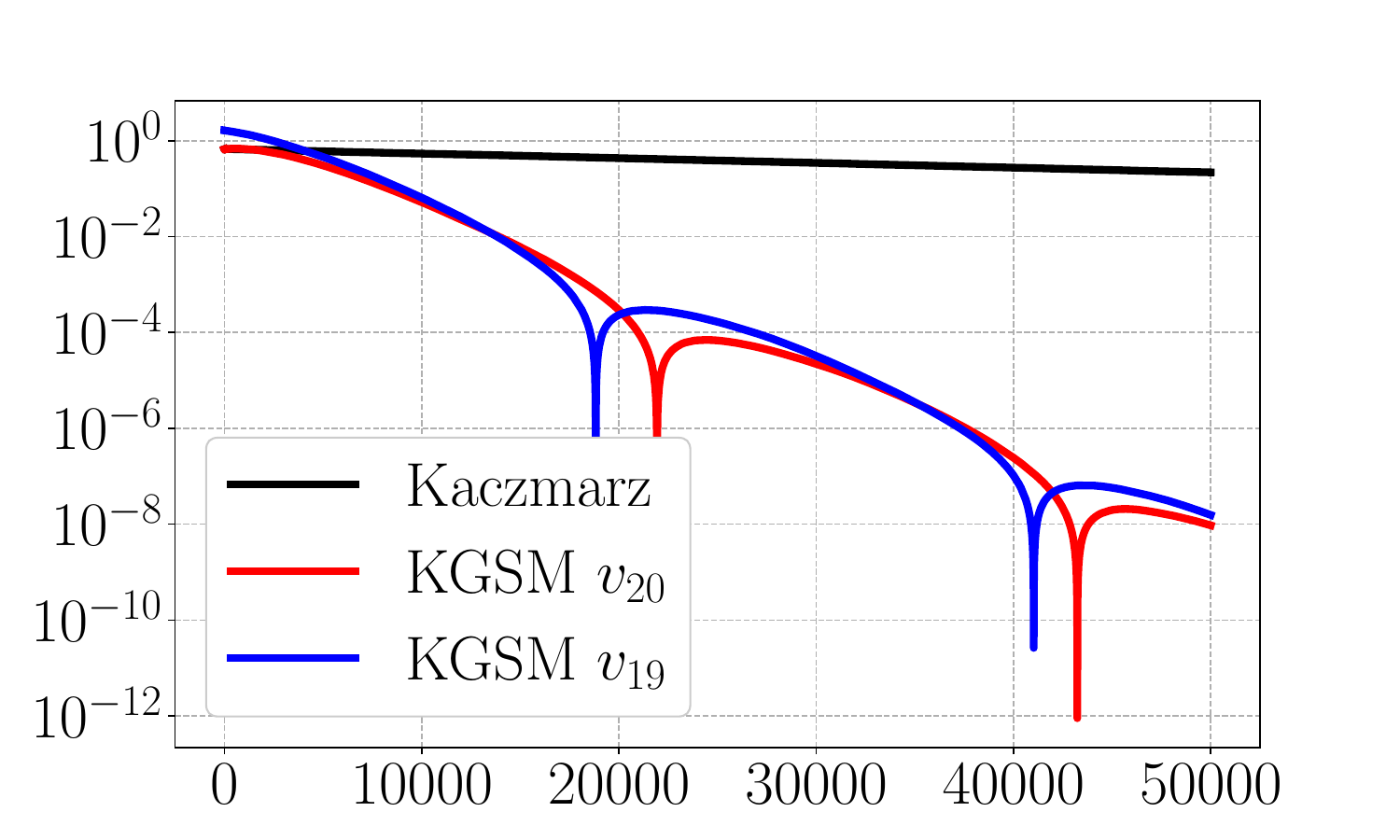} &
\includegraphics[width=.45\textwidth,trim= 30 0 30 0]{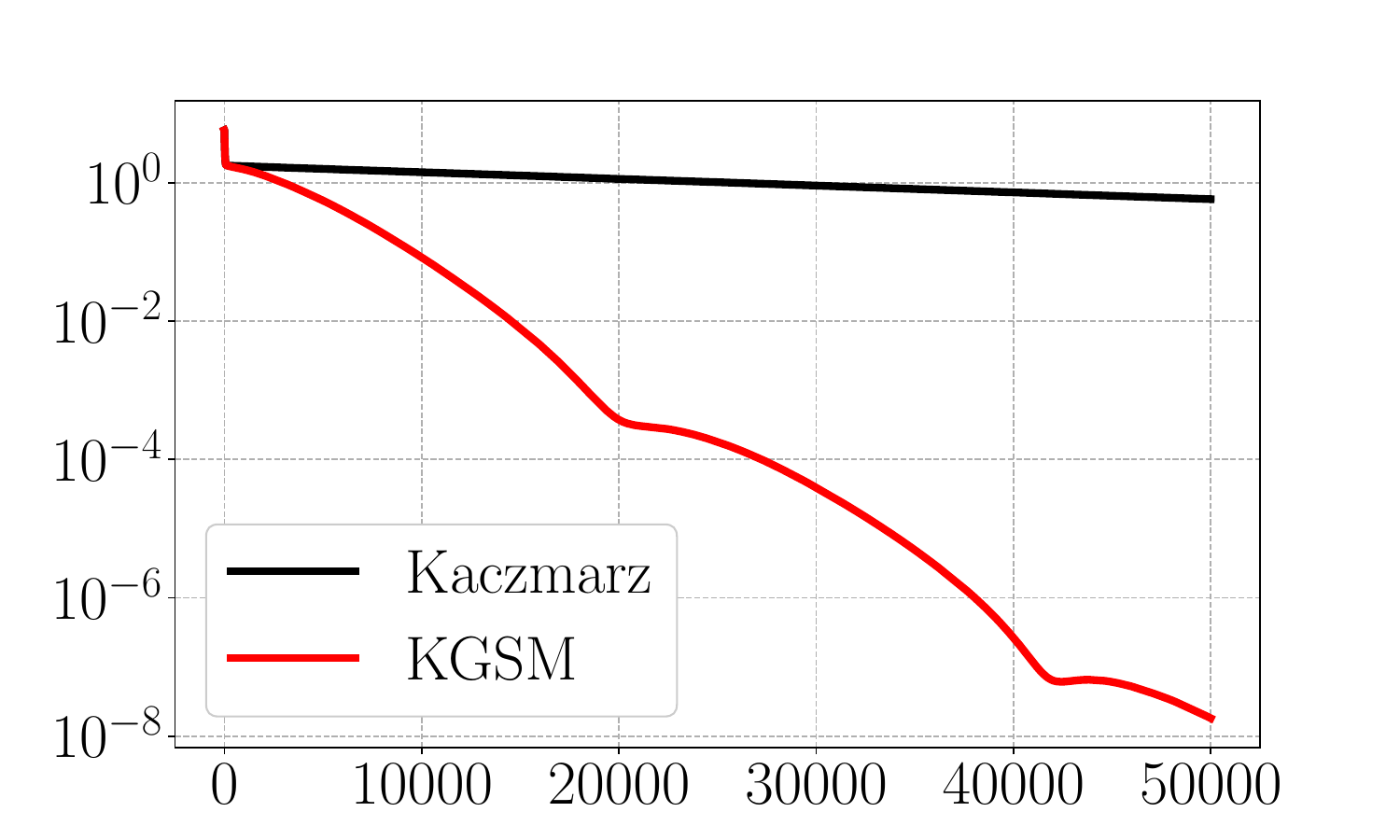} 
\end{tabular}
\caption{ The errors $|\langle x_k - x, v_{20}\rangle |$ and $|\langle x_k - x, v_{19}\rangle |$ (left) and 
$\|x_k-x\|_2$ (right) for  randomized Kaczmarz \eqref{kaczmarz} and KGSM \eqref{eq:our-method2} for a linear system with two small singular vectors.}
\label{fig09}
\end{figure}

In the numerical example in Figure \ref{fig09}, the $\ell^2$-norm error does not capture the finer features of the singular vector error since there are two small singular vectors whose errors are not perfectly synchronized due to numerical error.
The general rate of decrease is similar, but the finer features are lost. 

\section{Proof of main result} \label{proofmainresult}
This section gives the proof of Theorem \ref{thm1} and its corollaries.
\subsection{Notation and preliminaries}  \label{prelimproof}
Let $A$ be $m \times n$ matrix that is tall in the sense that $m \ge n$. Let $v_l$ denote the right singular vector of $A$ associated with the $l$-th largest singular value $\sigma_l$. The proof of the results uses the following lemma, which is the first step in the proof of Steinerberger \cite[Theorem 1.1]{Steinerberger2021}.
\begin{lemma} \label{lem1}
Suppose that $a^\top$ is a row of $A$ chosen with probability $\|a\|_2^2/ \|A\|_F^2$. Then, 
$$
\mathbb{E} \frac{\langle x, a \rangle \langle a, v_l \rangle}{\|a\|_2^2} = \frac{\sigma_l^2}{\|A\|_F^2} \langle x , v_l \rangle,
$$
for any $x \in \mathbb{R}^n$.
\end{lemma}
We give the proof for completeness.
\begin{proof}[Proof of Lemma \ref{lem1}] Direct calculation gives
\begin{equation*}
\mathbb{E} \frac{\langle x, a \rangle \langle a, v_l \rangle}{\|a\|_2^2} =  \sum_{i=1}^m \frac{\|a_i\|_2^2}{\|A\|_F^2} \frac{x^\top a_{i} a_{i}^\top v_l}{\|a_i\|_2^2} = \frac{x^\top A^\top A v_l}{\|A\|_F^2}  =  \frac{\sigma_l^2 x^\top v_l }{\|A\|_F^2},
\end{equation*}
which completes the proof.
\end{proof}

\subsection{Proof of main result}
\begin{proof}[Proof of Theorem \ref{thm1}]
Given an initial $x_0$ and $\beta \in [0,1)$ and $M \in [0,1]$ let $y_0 = 0 \in \mathbb{R}^n$ be the zero vector. Suppose that $x_k$ and $y_k$ are defined iteratively by \eqref{eq:our-method2}. 
Without loss of generality, by the possibility of replacing $x_k$ with $x_k - x$, we can assume the solution $x$ is the zero vector. 
In this case, \eqref{eq:our-method2} reduces to
\begin{equation} \label{eq:our-method-00}
\left\{
\begin{array}{ccl}
   x_{k+1}& = & x_k -\displaystyle \frac{\langle x_k , a_{i_k} \rangle a_{i_k}}{\|a_{i_k}\|^2_2}   +M y_{k} \\[10pt]
   y_{k+1}& = & \beta y_k + (1-\beta)(x_{k+1}-x_k) .
\end{array} \right. 
\end{equation}
Inducting on $k$ gives
\begin{equation} \label{eq:yk}
y_k  = (1 - \beta) \sum_{j=1}^k\beta^{k-j}(x_j-x_{j-1}).
\end{equation}
 Indeed, the base case $k=0$ of the induction holds since an empty sum is equal to zero, and assuming the formula holds for $y_k$ and applying \eqref{eq:our-method-00} gives
\begin{equation} \label{eq:yk2}
\begin{split}
y_{k+1} &= \beta\left( (1-\beta) \sum_{j=1}^k\beta^{k-j}(x_j-x_{j-1}) \right) + (1-\beta)(x_{k+1} - x_k) , \\
&= (1-\beta) \sum_{j=1}^{k+1} \beta^{k+1-j}(x_j-x_{j-1}).
\end{split}
\end{equation}
Observe that splitting the sum over $j$ in  \eqref{eq:yk} into two sums, shifting the indices, and combining terms gives
\begin{equation} \label{eq:expandsum}
\begin{split}
\sum_{j=1}^k\beta^{k-j} (x_j-x_{j-1}) &= \sum_{j=1}^k\beta^{k-j}x_j - \sum_{j=1}^k\beta^{k-j} x_{j-1} , \\
&= \sum_{j=1}^k\beta^{k-j} x_j - \sum_{j=0}^{k-1}\beta^{k-j-1} x_{j}, \\
&= x_k - (1-\beta) \left( \sum_{j=1}^{k-1}  \beta^{k-j-1}  x_j \right)  - \beta^{k-1} x_0 .
\end{split}
\end{equation}
It follows that
\begin{equation} \label{ykformula}
y_k = (1-\beta) x_k - (1-\beta)^2 \left( \sum_{j=1}^{k-1}  \beta^{k-j-1}  x_j \right) - 
(1-\beta) \beta^{k-1} x_0.
\end{equation}
Combining \eqref{eq:our-method-00}, \eqref{eq:yk}, \eqref{eq:yk2}, and \eqref{eq:expandsum} gives
\begin{equation} \label{expandxkk}
\begin{split}
x_{k+1} =& \big(1 + M(1-\beta) \big)  x_k - \frac{\langle x_k , a_{i_k} \rangle}{\|a_{i_k}\|_2^2} a_{i_k} \\   &- M(1 - \beta)^2 \left( \sum_{j=1}^{k-1} \beta^{k-j-1} x_j \right) -
M(1-\beta) \beta^{k-1} x_0.
\end{split}
\end{equation}
Recall that $v_l$ denotes the right singular vector associated with the $l$-th largest singular value $\sigma_l$ of $A$. Taking the inner product of \eqref{expandxkk} with $v_l$ gives
\begin{equation*}
\begin{split}
\langle x_{k+1}, v_l \rangle =& 
\big(1 + M(1-\beta)\big) \langle x_k, v_l \rangle - \frac{\langle x_k , a_{i_k} \rangle \langle a_{i_k}, v_l \rangle}{\|a_{i_k}\|^2_2}  \\ 
& -M(1 - \beta)^2 \left(  \sum_{j=1}^{k-1} \beta^{k-j-1} \langle x_j, v_l \rangle  \right)
- M(1-\beta) \beta^{k-1} \langle x_0, v_l \rangle .
\end{split}
\end{equation*}
Let $\mathbb{E}_{i_0,\ldots,i_k-1}$ denote the expectation conditional on $i_0,\ldots,i_{k-1}$. We have
\begin{equation} \label{mainlong}
\begin{split}
\mathbb{E}_{i_0,...,i_{k-1}}\langle x_{k+1} , v_l \rangle = & \langle x_{k} , v_l \rangle\left( 1 - \frac{\sigma_l^2}{\|A \|_F^2}+M(1-\beta)\right) \\ &-M(1-\beta)^2 \left(\sum_{j=1}^{k-1}\beta^{k-j-1}{\langle x_j,v_l \rangle} \right)- M(1-\beta) \beta^{k-1} \langle x_0, v_l \rangle,
\end{split}
\end{equation}
where the term $\sigma_l^2/\|A\|_F^2$ results from applying Lemma \ref{lem1}. Set
\begin{equation} \label{defrmsn}
\begin{split}
r &:= 1 - \frac{\sigma_l^2}{\|A \|_F^2}+M(1-\beta), \\
\zeta &:= M(1-\beta)^2, \\
S_n &:= \left( \sum_{j=1}^{n}\beta^{n-j}{\langle x_j,v_l \rangle} \right) + \frac{\beta^n}{1-\beta} \langle x_0, v_l \rangle.
\end{split}
\end{equation}
With this notation, \eqref{mainlong} can be written as 
\begin{equation} \label{base}
\mathbb{E}_{i_0,...,i_{k-1}}\langle x_{k+1} , v_l \rangle = r\langle x_k, v_l \rangle - \zeta S_{k-1}.
\end{equation}
For any fixed integer $k \ge 1$, and $j = 1,2,\ldots,k$, we claim that
\begin{equation} \label{eq:key-expansion}
  \mathbb{E}_{i_0,...,i_{k-j}}\langle x_{k+1} , v_l \rangle =   
    \begin{bmatrix}
        r \\
        \zeta \\
    \end{bmatrix}^\top
    \begin{bmatrix}
        r & \zeta \\
        -1 & \beta \\
    \end{bmatrix}^{j-1}
        \begin{bmatrix}
        \langle x_{k+1-j}, v_l \rangle \\ -S_{k-j}
    \end{bmatrix}.
\end{equation}
For fixed integer $k \ge 1$,  we prove this expression by induction on $j$. In base case $j=1$ of the induction, \eqref{eq:key-expansion} reduces to \eqref{base}. Assume that \eqref{eq:key-expansion} holds up to $j$. By the tower property of conditional expectation, we have $\mathbb{E}_{i_0,...,i_{k-j-1}}
\mathbb{E}_{i_0,...,i_{k-j}} = \mathbb{E}_{i_0,...,i_{k-j-1}}$. Thus,
taking the conditional expectation $\mathbb{E}_{i_0,...,i_{k-j-1}}$ of both sides gives
\begin{equation*} 
\begin{split}
  \mathbb{E}_{i_0,...,i_{k-j-1}}
  \langle x_{k+1} , v_l \rangle &=   
  \mathbb{E}_{i_0,...,i_{k-j-1}}
     \begin{bmatrix}
        r \\
        \zeta \\
    \end{bmatrix}^\top
    \begin{bmatrix}
        r & \zeta \\
        -1 & \beta \\
    \end{bmatrix}^{j-1}
     \begin{bmatrix}
        \langle x_{k-j+1}, v_l \rangle \\ -S_{k-j}
    \end{bmatrix},\\
    &=   
      \begin{bmatrix}
        r \\
        \zeta \\
    \end{bmatrix}^\top
    \begin{bmatrix}
        r & \zeta \\
        -1 & \beta \\
    \end{bmatrix}^{j-1}
      \mathbb{E}_{i_0,...,i_{k-j-1}}
        \begin{bmatrix}
        \langle x_{k-j+1}, v_l \rangle \\ -S_{k-j}
    \end{bmatrix}, \\
    &=   
      \begin{bmatrix}
        r \\
        \zeta \\
    \end{bmatrix}^\top
    \begin{bmatrix}
        r & \zeta \\
        -1 & \beta \\
    \end{bmatrix}^{j-1}
        \begin{bmatrix}
        r \langle x_{k-j}, v_l \rangle - \zeta S_{k-j-1} \\ 
      -\langle x_{k-j}, v_l \rangle - \beta S_{k-j-1}
    \end{bmatrix}, \\
        &=   
      \begin{bmatrix}
        r \\
        \zeta \\
    \end{bmatrix}^\top
    \begin{bmatrix}
        r & \zeta \\
        -1 & \beta \\
    \end{bmatrix}^{j}
        \begin{bmatrix}
         \langle x_{k-j}, v_l \rangle \\ 
       -  S_{k-j-1}
    \end{bmatrix}, \\
\end{split}
\end{equation*}
where the second equality follows from the linearity of expectation, the third equality follows from \eqref{base} and the definition of $S_k$, and the fourth equality follows from factoring out
\[
    \begin{bmatrix}
        r & \zeta \\
        -1 & \beta \\
    \end{bmatrix}
\]
 from the right vector. Setting $j = k-1$ gives
\begin{equation} \label{eqqad}
  \mathbb{E}_{i_0} \langle x_{k+1} , v_l \rangle =   
    \begin{bmatrix}
        r \\
        \zeta \\
    \end{bmatrix}^\top
    \begin{bmatrix}
        r & \zeta \\
        -1 & \beta \\
    \end{bmatrix}^{k-1}
        \begin{bmatrix}
        \langle x_1, v_l \rangle \\ 
       -S_0
    \end{bmatrix}.
\end{equation}
Taking the full expectation of \eqref{eqqad} gives
\begin{equation}
\begin{split}
 \mathbb{E} \langle x_{k+1} , v_l \rangle &=   
    \begin{bmatrix}
        r \\
        \zeta \\
    \end{bmatrix}^\top
    \begin{bmatrix}
        r & \zeta \\
        -1 & \beta \\
    \end{bmatrix}^{k-1}
       \mathbb{E} 
    \begin{bmatrix}
        \langle x_1, v_l \rangle \\ 
       -S_0
    \end{bmatrix}, \\
    &=   
    \begin{bmatrix}
        r \\
        \zeta \\
    \end{bmatrix}^\top
    \begin{bmatrix}
        r & \zeta \\
        -1 & \beta \\
    \end{bmatrix}^{k-1}
\begin{bmatrix}
   (1-\sigma_l^2/\|A\|_F^2) \langle x_0,v_l \rangle \\ 
   -\frac{1}{1-\beta} \langle x_0, v_l \rangle
\end{bmatrix}, \\
    &=   
    \begin{bmatrix}
        r \\
        \zeta \\
    \end{bmatrix}^\top
    \begin{bmatrix}
        r & \zeta \\
        -1 & \beta \\
    \end{bmatrix}^{k-1}
    \begin{bmatrix}
   r - \frac{\zeta}{1-\beta} \\ 
   -\frac{1}{1-\beta}
\end{bmatrix}
 \langle x_0, v_l \rangle,
\end{split}
\end{equation}
where the final equality follows from the definition of $r$ and $\zeta$, see \eqref{defrmsn}.  In addition since 
\[
\begin{bmatrix}
  r & \zeta \\ 
  -1 & \beta 
\end{bmatrix}
\begin{bmatrix}
    1 \\ \frac{-1}{1-\beta}
\end{bmatrix}
=
\begin{bmatrix}
    r - \frac{\zeta}{1-\beta} \\ -\frac{1}{1-\beta}
\end{bmatrix},
\]
we have proved that when $k \ge 1$, we have
\begin{equation} \label{formula}
 \mathbb{E} \langle x_{k+1} , v_l \rangle =   
    \begin{bmatrix}
        r \\
        \zeta \\
    \end{bmatrix}^\top
    \begin{bmatrix}
        r & \zeta \\
        -1 & \beta \\
    \end{bmatrix}^{k}
    \begin{bmatrix}
   1 \\ 
   -\frac{1}{1-\beta}
\end{bmatrix}
 \langle x_0, v_l \rangle.
\end{equation}
Note that so far, we have assumed $k \ge 1$. To complete the proof, it remains to consider the case when $k=0$. On one hand computing the expectation  $\mathbb{E} \langle x_1, v_l \rangle$ using \eqref{eq:our-method-00} gives
\begin{equation} \label{way1}
\mathbb{E} \langle x_1, v_l \rangle = \langle x_0, v_l \rangle\left( 1 - \frac{\sigma_l^2}{\|A\|_F^2} \right) = \left( r - \frac{\zeta}{1-\beta} \right)
 \langle x_0, v_l \rangle ;
 \end{equation}
on the other hand, setting $k=0$ on the right-hand side of \eqref{formula} gives
\begin{equation}  \label{way2}
    \begin{bmatrix}
        r \\
        \zeta \\
    \end{bmatrix}^{\top}
    \begin{bmatrix}
   1 \\ 
   -\frac{1}{1-\beta}
\end{bmatrix}
 \langle x_0, v_l \rangle
 = \left( r - \frac{\zeta}{1 - \beta} \right)
 \langle x_0, v_l \rangle.
\end{equation}
Since the right-hand side of \eqref{way1} and \eqref{way2} are equal, we conclude that the case $k=0$ holds, and thus we have established
\begin{equation} 
 \mathbb{E} \langle x_{k+1} , v_l \rangle =   
    \begin{bmatrix}
        r \\
        \zeta \\
    \end{bmatrix}^\top
    \begin{bmatrix}
        r & \zeta \\
        -1 & \beta \\
    \end{bmatrix}^{k}
    \begin{bmatrix}
   1 \\ 
   -\frac{1}{1-\beta}
\end{bmatrix}
 \langle x_0, v_l \rangle,
\end{equation}
for all $k \ge 0$. This completes the proof.
\end{proof}

\subsection{Proof of corollaries} \label{proofcors}
\begin{proof}[Proof of Corollary \ref{cormin}] 
Let $M \in [0,1]$ and $0 < \eta_l \le 1$ be fixed. By \eqref{eqeigenvalues} we have
\begin{equation} \label{lambda1eqpf}
\lambda_1(\beta) := \frac{r + \beta + \sqrt{(r-\beta)^2 - 4 \zeta}}{2}.
\end{equation}
 Recall that $r,\zeta,$ and $\eta_l$ are defined by 
$$
r = 1 - \eta_l +M(1-\beta), \quad
\zeta = M(1-\beta)^2, \quad \text{and} \quad 
\eta_l = \frac{\sigma_l^2}{\|A \|_F^2}.
$$
The discriminant $D:= (r - \beta)^2 - 4 \zeta$ has two zeros
$$
\beta_0 := 1 - \frac{\eta_l}{(1-\sqrt{M})^2}, \quad \text{and} \quad \beta_1 := 1 - \frac{\eta_l}{(1+\sqrt{M})^2}.
$$
The discriminant $D$ is nonnegative when $\beta \in (-\infty,\beta_0]$, negative when $\beta \in (\beta_0,\beta_1)$, and nonnegative when $\beta \in [\beta_1,\infty)$. 

\subsection*{Case 1: $M \in [0,(1-\sqrt{\eta_l})^2]$} In this case, $\beta_0\geq 0$ and we claim
$$
\beta_0 = \argmin_{\beta \in [0,1]} |\lambda_1(\beta)|.
$$
We prove this claim by considering three subcases.
\subsubsection*{Case 1a: $\beta \in [0,\beta_0]$} In this case, the discriminant $D$ is nonnegative so
$$
|\lambda_1(\beta)| = \frac{1}{2} \left( \beta + r + \sqrt{(\beta+r)^2 - 4(r \beta + \zeta)} \right) = \frac{1}{2} \left( \beta + r + \sqrt{D} \right).
$$
We claim that $\partial_\beta |\lambda_1(\beta)| \le 0$. Indeed,
$$
\partial_\beta |\lambda_1(\beta)| = \frac{1}{2}\left( 1 - M + \frac{\partial_\beta D}{2\sqrt{D}} \right),
$$
where
$$
\frac{\partial_\beta D}{2}  = -(1-\beta)(1-M)^2 + \eta_l (1+M).
$$
The assumption that $\beta \in [0,\beta_0]$ and $M \in [0,(1-\sqrt{\eta_l})^2]$ implies that $\partial_\beta D/2 \le 0$.
Thus, to show that $\partial_\beta |\lambda_1(\beta)| \le 0$ it suffices to show that
$$
(1 - M)^2   \le \left( \frac{\partial_\beta D}{2 \sqrt{D}} \right)^2.
$$
That is, it suffices to show that
$$
(1 - M)^2 D   - \left( \frac{\partial_\beta D}{2} \right)^2 \le 0.
$$
Using the definition of $D$ and expanding terms gives
$$
(1 - M)^2 D   - \left( \frac{\partial_\beta D}{2} \right)^2 = - 4 \eta_l^2 M \le 0,
$$
so we conclude that $\partial_\beta |\lambda_1(\beta)| \le 0$, which implies that $|\lambda_1(\beta)| \ge |\lambda_1(\beta_0)|$ for $\beta \in [0,\beta_0]$.

\subsubsection*{Case 1b: $\beta \in (\beta_0,\beta_1)$}
In this case, the discriminant is negative so
$$
|\lambda_1(\beta)|^2 = \frac{(r+\beta) + 4\zeta - (r-\beta)^2}{4} = 
(1- \eta_l + M(1-\beta)) \beta + M(1-\beta)^2.
$$
Observe that the derivative of $|\lambda_1(\beta)|^2$ with respect to $\beta_0$ is positive
$$
\partial_\beta |\lambda_1(\beta)|^2 = 1 - \eta_l - M \ge 1 - \eta_l - (1-\sqrt{\eta_l})^2 = 2(\sqrt{\eta_l} - \eta_l) > 0,
$$
which implies $|\lambda_1(\beta_0)| < |\lambda_1(\beta)|$ for all $\beta \in (\beta_0,\beta_1)$.

\subsubsection*{Case 1c: $\beta \in [\beta_1,1)$}
In this case, $r \le \beta$, so we can estimate
$$
|\lambda_1(\beta)| \ge \frac{r+\beta}{2} \ge r \ge 1- \eta_l \ge |\lambda(\beta_0)|,
$$
which completes the proof of Case 1.

\subsection*{Case 2: $M \in ((1-\sqrt{\eta_l})^2,1)$}
In this case, we claim
\[
\argmin_{\beta \in [0,1]}{|\lambda_1(\beta)|} =
\left\{\begin{array}{cl}
     0 & \text{if } M \in [(1-\sqrt{\eta_l})^2,1-\eta_l] \\
     \beta_1 & \text{if } M \in (1-
     \eta_l,1] 
\end{array}\right..
\]

We prove the claim by considering two subcases.
\subsubsection*{Case 2a : $\beta \in [0,\beta_1)$ } Since the assumption on $M$ implies $[0,\beta_1) \subset (\beta_0,\beta_1)$ the discriminant is negative and 
$$
|\lambda_1(\beta)|^2 = \frac{(\beta+r)^2 + 4( r \beta + \zeta) - (\beta+r)^2}{4} = (1- \eta_l + M(1-\beta)) \beta + M(1-\beta)^2.
$$
Taking the derivative with respect to $\beta$ gives
$$
\partial_\beta |\lambda_1(\beta)|^2 = 1 - \eta_l - M.
$$
We conclude that the minimum occurs at $\beta = 0$ when $M \in [(1-\sqrt{\eta_l})^2,1-\eta]$  or at $\beta_1$ when $M \in [1-\eta_l,1]$.

\subsubsection*{Case 2b: $\beta \in [\beta_1,1)$}  Since $\beta \in [\beta_1,1)$ the discriminant is nonnegative, the eigenvalues are real, and
$$
\partial_\beta |\lambda_1(\beta)| = \frac{1}{2} \left( 1 - M + \frac{\partial_\beta D}{2\sqrt{D}} \right)
$$
as in Case 1a. Observe that, by our assumption on $\beta$, 
$$
\partial_\beta D/2  =  -(1-\beta)(1-M)^2 + \eta_l (1+M) > \left(\frac{\eta_l}{(1+\sqrt{M})^2}\right)(1-M)^2 + \eta_l (1+M)
$$
\[
= \eta_l\left[\frac{1-M^2}{(1+\sqrt{M})^2}+(1+M) \right].
\]
Since $M\in ((1-\sqrt{\eta_l})^2,1)$, all terms above are positive, which shows $\partial_\beta D/2 > 0$. Furthermore, $\partial_\beta|\lambda_1(\beta)|>0$ implying the minimum occurs at $\beta_1$. This means that for $M\in [(1-\sqrt{\eta})^2,1]$  the minimum reduces to Case 2a by continuity of $|\lambda_1(\beta)|$, completing our proof.
\end{proof}

\begin{proof}[Proof of Corollary \ref{coropt}] \label{corspectralproof}
Let $B$ denote the $2 \times 2$  matrix
$$
B = \begin{bmatrix}
r & \zeta \\
-1 & \beta \\
\end{bmatrix}
$$
from \eqref{maineq} in Theorem \ref{thm1} and recall the eigenvalues of $B$ are
$$
\lambda_1 := \frac{r + \beta + \sqrt{(r-\beta)^2 - 4 \zeta}}{2}
\quad \text{and} \quad
\lambda_2 := \frac{r + \beta - \sqrt{(r-\beta)^2 - 4 \zeta}}{2}.
$$
Set
$$
\beta = 1 - \frac{\eta_l}{(1-\sqrt{M})^2}.
$$
In this case
$$
B = \begin{bmatrix}
1 - \eta_l \frac{1 - 2 \sqrt{M}}{(1-\sqrt{M})^2} & \eta_l^2\frac{M }{(1-\sqrt{M})^4} \\[5pt]
-1 & \eta_l\frac{M }{(1-\sqrt{M})^2}
\end{bmatrix}
$$
has the eigenvalue
$$
\lambda =  1 - \frac{\eta_l}{1 - \sqrt{M}}
$$
of algebraic multiplicity two. It is straightforward to verify that the Jordan decomposition of this matrix is
\begin{equation*} 
B =
\begin{bmatrix}
\frac{\eta_l^2 M}{(1-\sqrt{M})^4} & 0 \\[5pt] 
-\frac{\eta_l \sqrt{M}}{(1-\sqrt{M})^2} & 1 
\end{bmatrix}
\begin{bmatrix}
1 - \frac{\eta_l}{1-\sqrt{M}} & 1 \\[5pt]
0 & 1 - \frac{\eta_l}{1 - \sqrt{M}}
\end{bmatrix}
\begin{bmatrix}
\frac{(1-\sqrt{M})^4}{\eta_l^2 M} & 0 \\[5pt]
\frac{(1-\sqrt{M})^2}{\eta_l \sqrt{M}} & 1
\end{bmatrix}.
\end{equation*}
It follows that
\begin{equation*} 
B^k =
\begin{bmatrix}
\frac{\eta_l^2 M}{(1-\sqrt{M})^4} & 0 \\[5pt] 
-\frac{\eta_l \sqrt{M}}{(1-\sqrt{M})^2} & 1 
\end{bmatrix}
\begin{bmatrix}
\left(1 - \frac{\eta_l}{1-\sqrt{M}}\right)^k &  k \left(1 - \frac{\eta_l}{1-\sqrt{M}}\right)^{k-1} \\[5pt]
0 & \left(1 - \frac{\eta_l}{1 - \sqrt{M}} \right)^k
\end{bmatrix}
\begin{bmatrix}
\frac{(1-\sqrt{M})^4}{\eta_l^2 M} & 0 \\[5pt]
\frac{(1-\sqrt{M})^2}{\eta_l \sqrt{M}} & 1
\end{bmatrix}.
\end{equation*}
By Theorem \ref{thm1} we have
\begin{multline*}
 \mathbb{E} \langle x_{k+1} - x, v_l \rangle =  
    \begin{bmatrix}
        1 - \eta_l + M \frac{\eta_l}{(1-\sqrt{M})^2} \\[5pt]
        M \frac{\eta_l^2}{(1-\sqrt{M})^4}
    \end{bmatrix}^\top
 \begin{bmatrix}
\frac{\eta_l^2 M}{(1-\sqrt{M})^4} & 0 \\[5pt] 
-\frac{\eta_l \sqrt{M}}{(1-\sqrt{M})^2} & 1 
\end{bmatrix}   \\
 \begin{bmatrix}
\left(1 - \frac{\eta_l}{1-\sqrt{M}}\right)^k &  k \left(1 - \frac{\eta_l}{1-\sqrt{M}}\right)^{k-1} \\[5pt]
0 & \left(1 - \frac{\eta_l}{1 - \sqrt{M}} \right)^k
\end{bmatrix}  
 \begin{bmatrix}
\frac{(1-\sqrt{M})^4}{\eta_l^2 M} & 0 \\[5pt]
\frac{(1-\sqrt{M})^2}{\eta_l \sqrt{M}} & 1
\end{bmatrix}  
    \begin{bmatrix}
   1 \\[5pt] 
   \frac{-(1-\sqrt{M})^2}{\eta_l}
\end{bmatrix}
 \langle x_0 -x, v_l \rangle;
 \end{multline*}
performing matrix multiplication gives
\begin{equation*}
 \mathbb{E} \langle x_{k+1} - x, v_l \rangle =   
\left(1 - \frac{\eta_l}{1-\sqrt{M}} \right)^k \left( 1 +   \frac{\eta_l\left((\sqrt{M}(k+1))-1\right)}{1- \sqrt{M}}  \right)
 \langle x_0 -x, v_l \rangle,
\end{equation*}
which completes the proof.
\end{proof}

\begin{proof}[Proof of Corollary \ref{corarg}] \label{proofofcorarg}
In addition to the assumptions in Theorem \ref{thm1}, suppose that the eigenvalue $\lambda_1$ defined in  \eqref{eqeigenvalues} is complex with a non-zero imaginary part; that is,
$\lambda_1 = \rho e^{i \theta}$ for $\rho>0$ and $0<\theta < \pi$. Applying Theorem \ref{thm1} we obtain
\[
\mathbb{E} \langle x_{k+1} - x, v_l\rangle = \begin{bmatrix} r \\ \zeta\end{bmatrix}^T Q \Lambda^k Q^{-1} \begin{bmatrix} 1 \\ \frac{-1}{1-\beta}\end{bmatrix}\langle x_0 - x, v_l \rangle ,
\]
where $Q \Lambda Q^{-1}$ is the eigenvalue decomposition of $B=\left[\begin{smallmatrix} r & \zeta \\ -1 & \beta \end{smallmatrix}\right]$ and $\zeta,\beta,r$ are defined as in the proof of Theorem \ref{thm1}. Since $B$ is a real matrix and $\lambda_1$ is complex we know that $\lambda_2$, the second eigenvalue of $B$, is conjugate to $\lambda_1$. Hence
\[
\mathbb{E} \langle x_{k+1} - x, v_l\rangle = \begin{bmatrix} r \\ \zeta\end{bmatrix}^T Q \begin{bmatrix}
\rho^k e^{ik \theta} & 0 \\ 0 & \rho^k e^{-ik\theta}    
\end{bmatrix}  
Q^{-1} \begin{bmatrix} 1 \\ \frac{-1}{1-\beta}\end{bmatrix}\langle x_0 - x, v_l \rangle .
\]
Defining $\begin{bmatrix}
    c_1 \\ c_2
\end{bmatrix} := Q^{-1}\begin{bmatrix} 1 \\ \frac{-1}{1-\beta}\end{bmatrix}\langle x_0 - x, v_l \rangle$ we see that
\[
\mathbb{E} \langle x_{k+1} - x, v_l\rangle = \rho^k \begin{bmatrix} r \\ \zeta\end{bmatrix}^T Q \left(\begin{bmatrix} c_1 \\ 0\end{bmatrix} e^{ik \theta}+ \begin{bmatrix} 0 \\ c_2\end{bmatrix} e^{-i k\theta}\right) = \rho^k(C_1 e^{ik \theta} + C_2 e^{-ik \theta}) ,
\]
where $C_1 = \begin{bmatrix} r \\ \zeta\end{bmatrix}^T Q \begin{bmatrix} c_1 \\ 0\end{bmatrix}$ and $C_2 = \begin{bmatrix} r \\ \zeta\end{bmatrix}^T Q \begin{bmatrix} 0 \\ c_2 \end{bmatrix}.$ Since $\mathbb{E} \langle x_{k+1} - x, v_l\rangle$ is a real number for all $k$, we see that
\[
\text{Re}(C_1)\sin(k\theta) + \text{Im}(C_1)\cos(k\theta) - \text{Re}(C_2)\sin(k\theta) + \text{Im}(C_2)\cos(k\theta) = 0, 
\]
for all $k.$ Since $0<\theta<\pi$, it follows that $\text{Re}(C_1)=\text{Re}(C_2)$ and $\text{Im}(C_1) = -\text{Im}(C_2).$ Thus,  
\[
\mathbb{E} \langle x_{k+1} - x, v_l\rangle = \rho^k(C_1e^{i k \theta} + \overline{C_1 e^{i k \theta}}).
\]
Writing $C_1= \frac{C}{2}e^{i \theta_0}$ for some $C>0$ and $\theta_0\in [0,2\pi)$, we obtain
\[
\mathbb{E} \langle x_{k+1} - x, v_l\rangle = \rho^k C \cos(k\theta + \theta_0).
\]
This completes the proof.

\end{proof}

\section{Discussion} \label{discussion}
\subsection{Summary}
This paper introduced the randomized Kaczmarz with geometrically smoothed momentum (KGSM) algorithm. The method depends on two hyperparameters
\begin{itemize}
\item the momentum parameter $M$, and
\item the geometric smoothing parameter $\beta$.
\end{itemize}
We analyzed the behavior of this method across different distributions of singular vectors and different values of $M$ and $\beta$. Our main result, Theorem \ref{thm1}, provides a formula for the expected signed error $\mathbb{E} \langle x - x_k, v_l \rangle$ of KGSM in the direction of the singular vectors $v_l$ of the matrix $A$ that defines the least squares loss. In Corollary \ref{corstefan}, we observed how our result extends Theorem 1.1 of Steinerberger \cite{Steinerberger2021}, while Corollaries
\ref{cormin} and \ref{coropt} optimize the smoothing parameter $\beta$ for
a fixed momentum parameter $M$. Our theoretical results are illustrated by several numerical examples in \S \ref{numerics}. In addition, we provided examples that show the limitations of our analysis and pose questions regarding the dynamics of KGSM. We detail some of these questions below and discuss potential avenues to extend our analysis.

\subsection{Limitations and questions}
The main limitation of our theoretical analysis is that it does not provide insight into the value of the critical momentum parameter $M$ for a given linear system. We discuss this issue and related questions in the following.

\subsubsection{Critical momentum} 
Our numerical results indicate that there is a critical momentum along the curve illustrated in Figure \ref{fig04} where the algorithm's behavior breaks down. More generally, there is a region of the $(M,\beta)$ parameter space where the method fails to converge. What is the critical $M$ or region in $(M,\beta)$ parameter space where this failure occurs? 

\subsubsection{Estimating $M$ and $\beta$ from data} 
We did not investigate methods for estimating effective parameters $M$ and $\beta$ from data. Even if formulas for $M$ and $\beta$ are known, these would need to be estimated from data to implement the method for practical problems. Is it possible to estimate suitable parameters as the algorithm runs adaptively? A related question involves estimating the locations of the periodic spikes in the error
discussed in \S \ref{periodicspikingbehavior}. Is it possible to estimate the locations of spikes in the error in an adaptive way?

\subsubsection{Nesterov acceleration} Another method to accelerate the convergence of randomized Kaczmarz is Nesterov acceleration \cite{liu2015}. Could our analysis be extended to analyze Kaczmarz methods that use Nesterov acceleration? It may be possible to adapt our analysis to the setting of Liu and Wright \cite{liu2015} or to study convergence along singular vectors for other accelerated gradient methods.

\subsubsection{Direction of convergence and convergence in $\ell^2$-norm}
We did not establish an analog of \cite[Theorem 1.3]{Steinerberger2021}, which would help to explain any directional change in KGSM. Can this result be generalized to the setting with geometrically smoothed momentum? In particular, can the change in direction be controlled as a function of $M$?
Another related question is establishing convergence for the expected absolute error in the direction of a singular vector  $|\langle x-x_k,v_l \rangle|$ or for the $\ell^2$-norm $\|x-x_k\|_2$. By Jensen's inequality, we have $|\mathbb{E} \langle x - x_k, v_l \rangle| \le \mathbb{E} |\langle x - x_k, v_l \rangle|$, so our results provide a lower bound that our numerical results indicate is sharp in some cases.  Thus, at least for some special cases, proving a result about the absolute error may be possible.

\subsubsection{Block methods} Our analysis considered the case of block size one (or equivalent minibatch size). Can the analysis be extended to Kaczmarz methods with variable block or minibatch size? One potential approach would be to combine our analysis with that of Bollapragada, Chen, and Ward \cite{bollapragada2023fast} or to directly approach the problem using the methodology of other block methods.

\section*{Acknowledgments}
 We thank the anonymous reviewers whose suggestions significantly improved the exposition of the results. NFM was supported in part by a start-up grant from Oregon State University.

\begin{appendix}

\section{Futher Examples} \label{furtherexamplesappendix}

\subsection{Elementary example: Gaussian system} \label{gaussiansystem}
In this section, we present an example of KGSM on a linear system defined by a $60 \times 50$ matrix with i.i.d. standard Gaussian entries. This example is straightforward to implement in any programming environment that supports random number generation and basic linear algebra operations; detailed pseudocode is included in Algorithm \ref{figgauss}.

\begin{algorithm}[h!]
\caption{Elementary example: KGSM vs. Kaczmarz on a Gaussian system}
\label{figgauss}
\begin{tabular}{cc}
\begin{minipage}{0.5\textwidth}
\begin{algorithmic}[1] 
\State $m \gets 60, \quad n \gets 50, \quad N \gets 15000$
\State $M \gets 0.8, \quad \beta \gets 0.98$ 
\State $A \gets \text{random\_normal}(m,n)$
\State $x \gets  \text{random\_normal}(n,1)$
\State $b \gets A x$ 
\State $x_0, y_0, z_0 \gets \text{zeros}(n,1)$
\State $u, v \gets \text{zeros}(N,1)$
\For{$k =1,\ldots,N$}
    \State $i \gets \text{unif\_random}(\{1,\ldots,m\})$ 
    \State $a_i \gets (\text{$i$-th row $A$})^\top$ (dim $n \times 1$)
    \State $x_1 = x_0 + \frac{b_i -  a_i^\top x_0}{\|a_i\|_2^2} a_i + M y_0$
    \State $y_1 = \beta y_0 + (1-\beta) (x_1 - x_0)$ 
    \State $z_1 = z_0 + \frac{b_i - a_i^\top z_0}{\|a_i\|_2^2} a_i$ 
    \State $x_0 \gets x_1, \quad y_0 \gets y_1, \quad z_0 \gets z_1$ 
    \State $u_k \gets \|z_0-x\|_2$, $v_k \gets \|x_0-x\|_2$  
\EndFor
\State \textbf{plot} $\log_{10}(u)$ blue, $\log_{10}(v)$ red
{\State  \textbf{legend}(`Kaczmarz',`KGSM')}
\end{algorithmic} 
\end{minipage} &
\raisebox{-.5\height}{\includegraphics[width=.43\textwidth,trim=50 0 0 0]{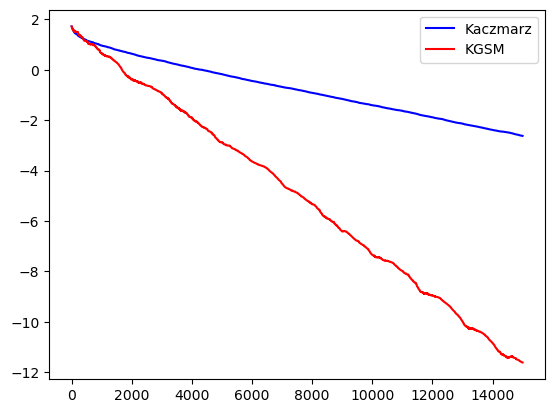}}
\end{tabular}
\end{algorithm} 

Note that the pseudocode in Algorithm \ref{figgauss} chooses rows uniformly at random instead of proportional to their squared $2$-norm; this was done intentionally for simplicity: since the squared $\ell^2$-norm of the rows concentrates to $n$, choosing rows uniformly at random does not make a large practical difference. Further, we note that the ratio $m/n$ was specifically chosen so that the distribution of singular values includes a range of big and small values. Tall Gaussian matrices are well-conditioned and thus will not provide an interesting example for KGSM; indeed, recall that if the width $n$ of a random Gaussian matrix is fixed, then the condition $\sigma_1(A)/\sigma_n(A) \to 1$ as $m  \rightarrow \infty$. 

\subsection{Batch size $1$ Heavy Ball Momentum} \label{heavyballmomentumbatchsize1fails}
In this section, we provide two numerical examples illustrating the ineffectiveness of batch size $1$ Heavy Ball Momentum \eqref{beta0}. In particular, we compare Kaczmarz, KGSM, and batch size $1$ Heavy Ball Momentum for the system with one small singular value described in \S \ref{basicex} and the system whose singular values decay linearly described in \S \ref{exlinear}, see Figure \ref{fig_heavyball}.

\begin{figure}[h!]
    \centering
    \includegraphics[width=.4\textwidth]{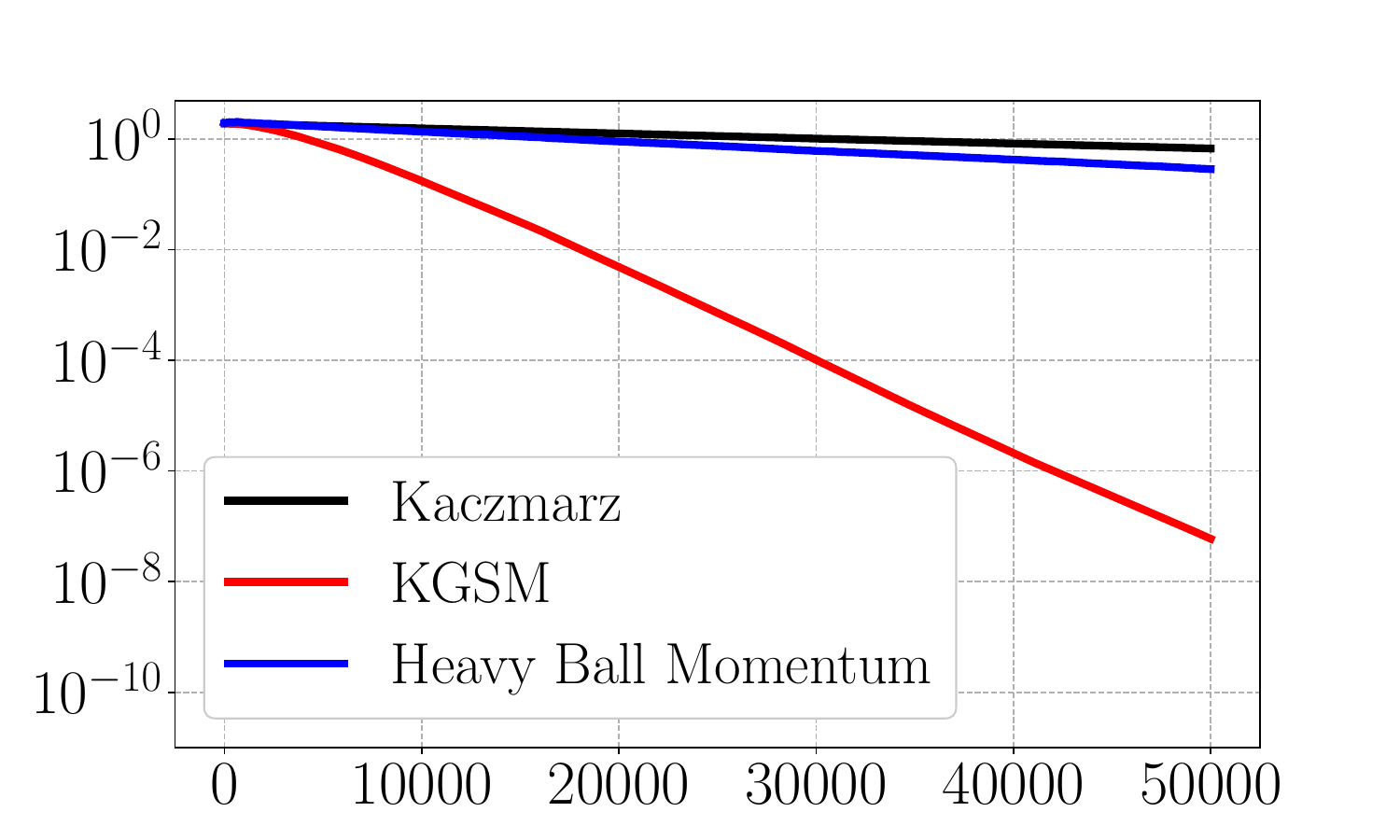}
    \includegraphics[width=.4\textwidth]{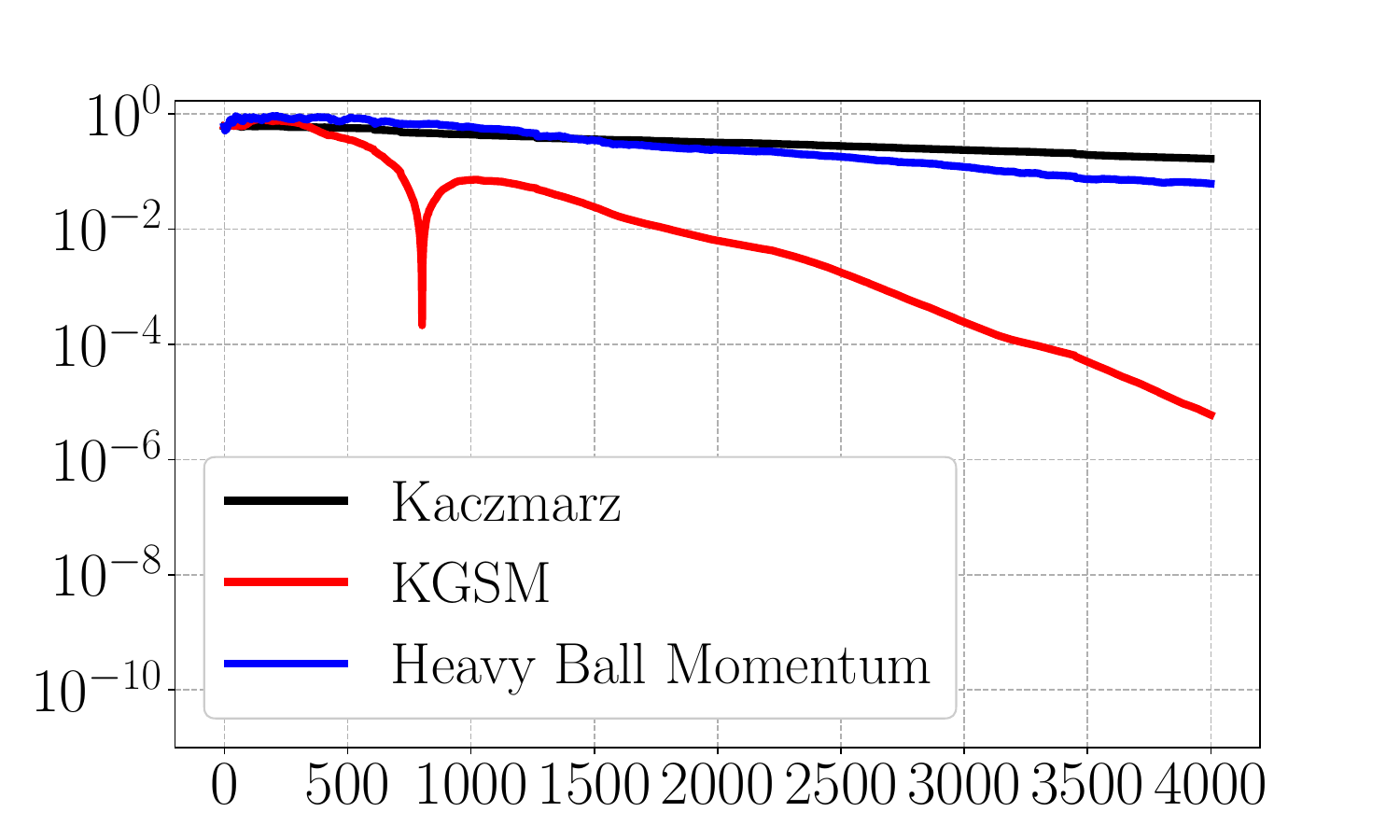}
    \caption{ The error $\|x-x_k\|_2$ for Kaczmarz, KGSM, and batch size $1$ Heavy Ball Momentum \eqref{beta0} with $M=0.5$ for the linear systems described in \S \ref{basicex} (left) and \S \ref{exlinear} (right). }
    \label{fig_heavyball}
\end{figure}

While batch size $1$ Heavy Ball Momentum is sometimes slightly better than randomized Kaczmarz, we were unable to identify any situation or parameter settings where it was better in any significant way.

\subsection{Additional spectral decays} \label{additionalspectraldecays} In this section, 
we consider systems whose singular values decay based on convex and concave functions. In particular, using the procedure described in \S \ref{numericsprelim} we construct $100\times 20$ matrices $A_{\mu}$ and $A_{\sigma}$ with singular values 
\begin{equation} \label{singularvalues}
\mu_i = \left(1-c_1 \left(\frac{i-1}{20}\right)^6\right) \quad \text{and} \quad \sigma_i = \left(1-c_2 \left(\frac{i-1}{20}\right)\right)^6,
\end{equation}
for $i=1,\ldots,20$, where $c_1$ and $c_2$ are constants chosen so that $\sigma_{20}=\mu_{20}=1/50$, see Figure \ref{fig_mu_sigma1}.

\begin{figure}[h]
    \centering
    \begin{tabular}{cc}
    \includegraphics[width = 0.4\textwidth,trim=30 0 0 0]{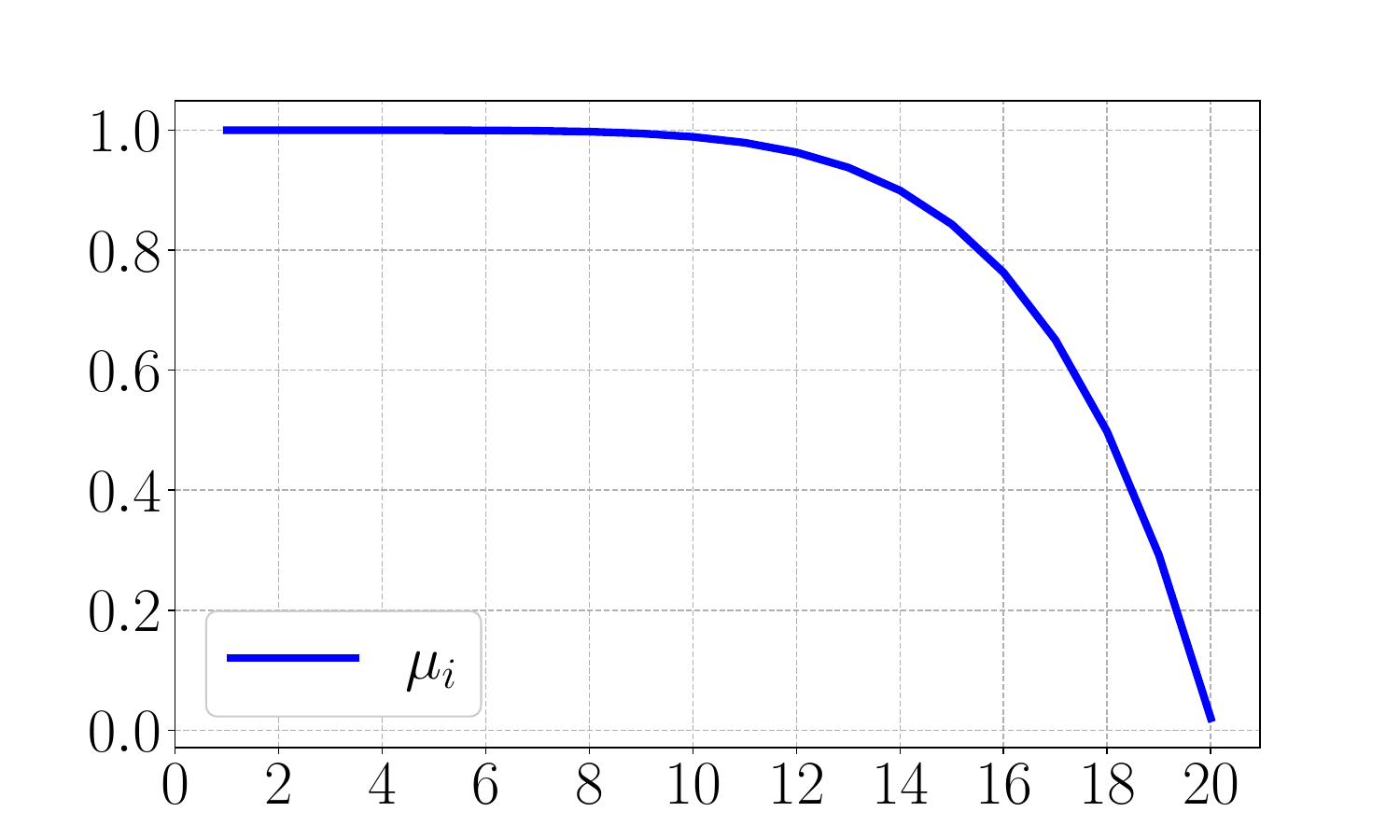}  &
    \includegraphics[width = 0.4\textwidth,trim=30 0 0 0]{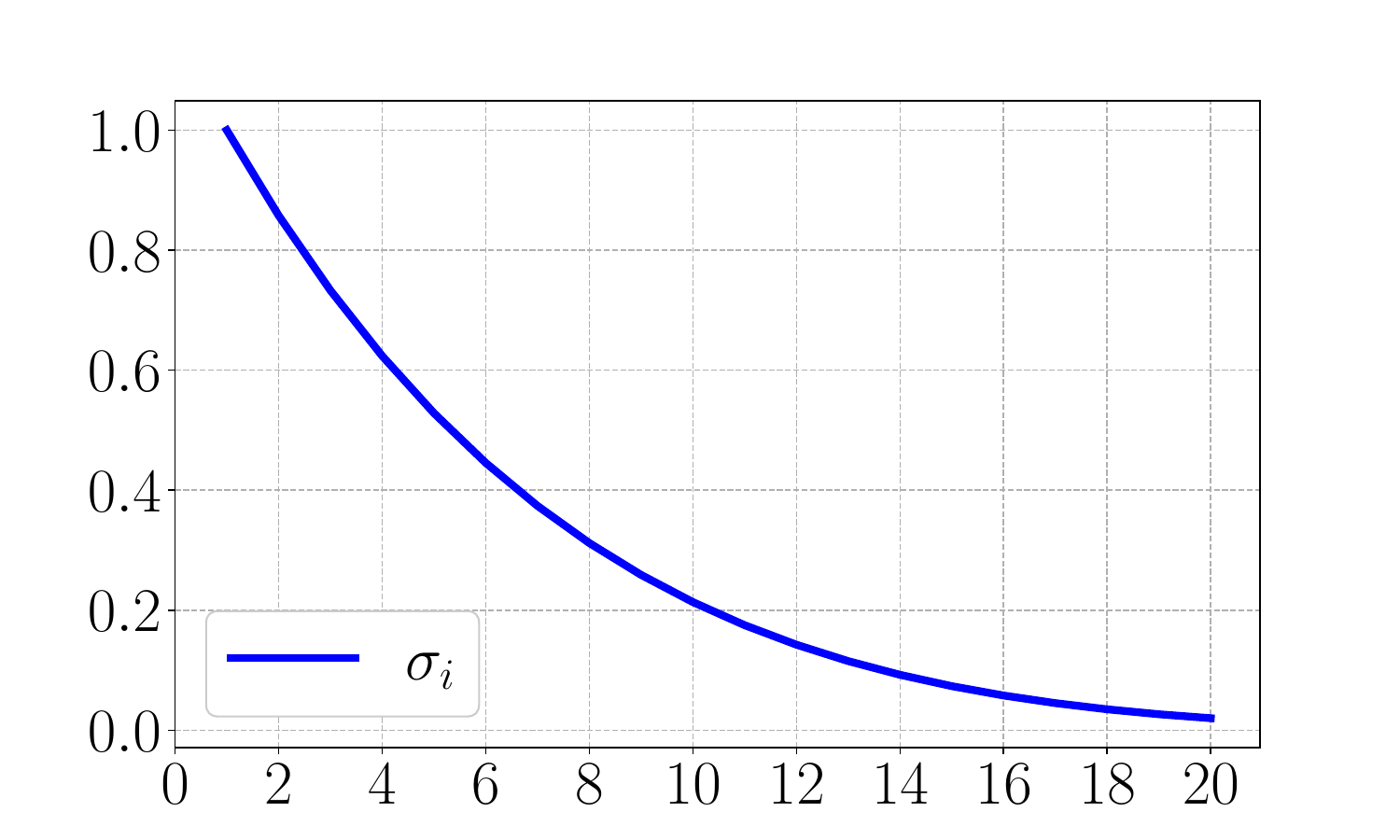}
    \end{tabular}
    \caption{ The singular values $\mu_i$ (left) and $\sigma_i$ (right)
    defined in \S \ref{singularvalues}.
    }
    \label{fig_mu_sigma1}
\end{figure}

We define $M_\mu = 0.95$, $\beta_{\mu} = 1 - \frac{\mu_{20}^2/\|A_{\mu}\|_{F}^2}{(1 - \sqrt{M})^2}$, $M_{\sigma} = 0.91$, and $\beta_{\sigma} = 1 - \frac{\sigma_{20}^2/\|A_{\sigma}\|_{F}^2}{(1 - \sqrt{M})^2}$
to be the momentum and smoothing parameters for each system, respectively; see Remark \ref{settingM} for a discussion about setting the momentum parameter $M$. We choose a random initial vector $x_0$ for each system (with independent standard normal entries), run randomized Kaczmarz \eqref{kaczmarz} and KGSM \eqref{eq:our-method}, and compute the $\ell_2$-errors, see Figure \ref{fig_mu_sigma}. 

\begin{figure}[h!]
    \centering
    \begin{tabular}{cc}
    \includegraphics[width = 0.4\textwidth,trim=30 0 0 0]{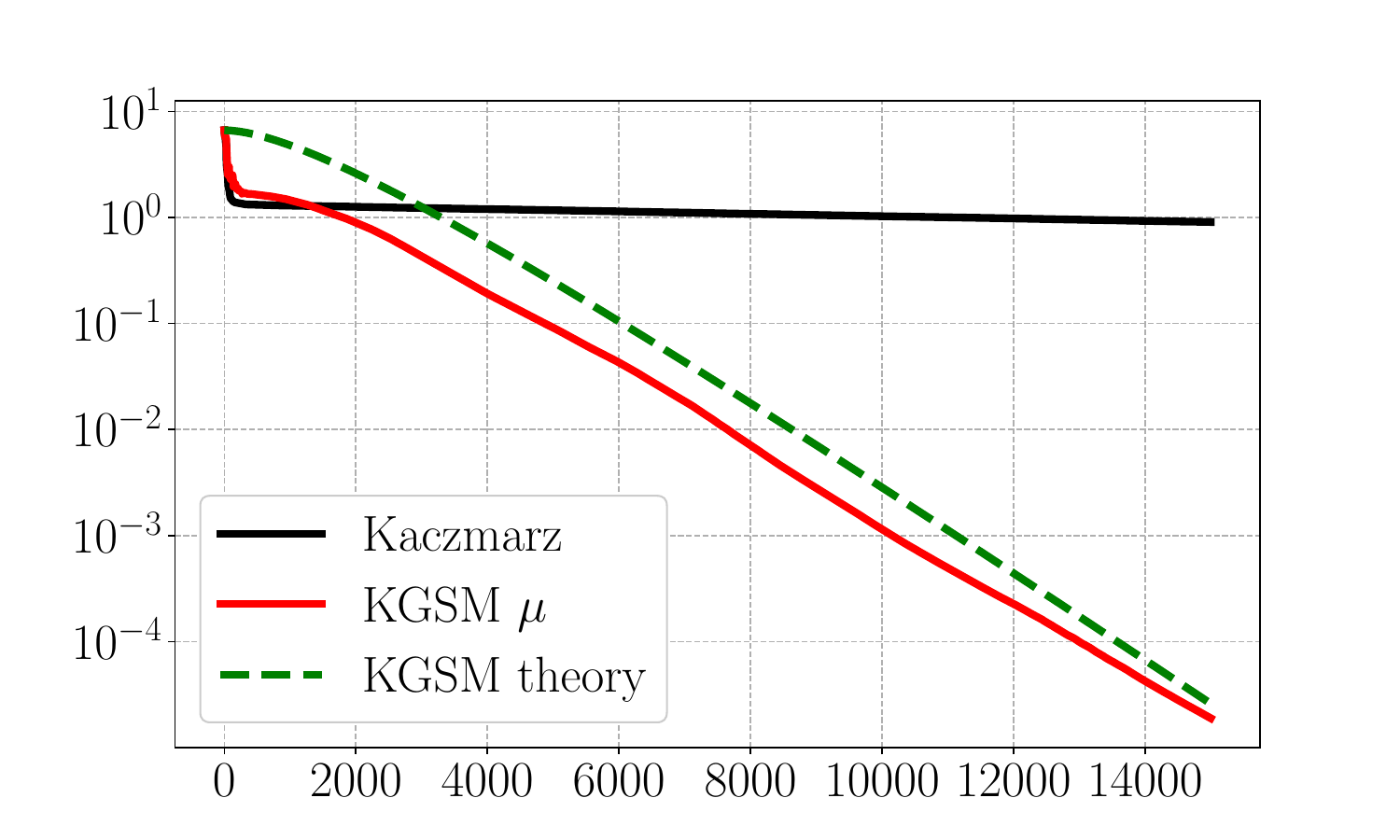} &
    \includegraphics[width = 0.4\textwidth,trim=30 0 0 0]{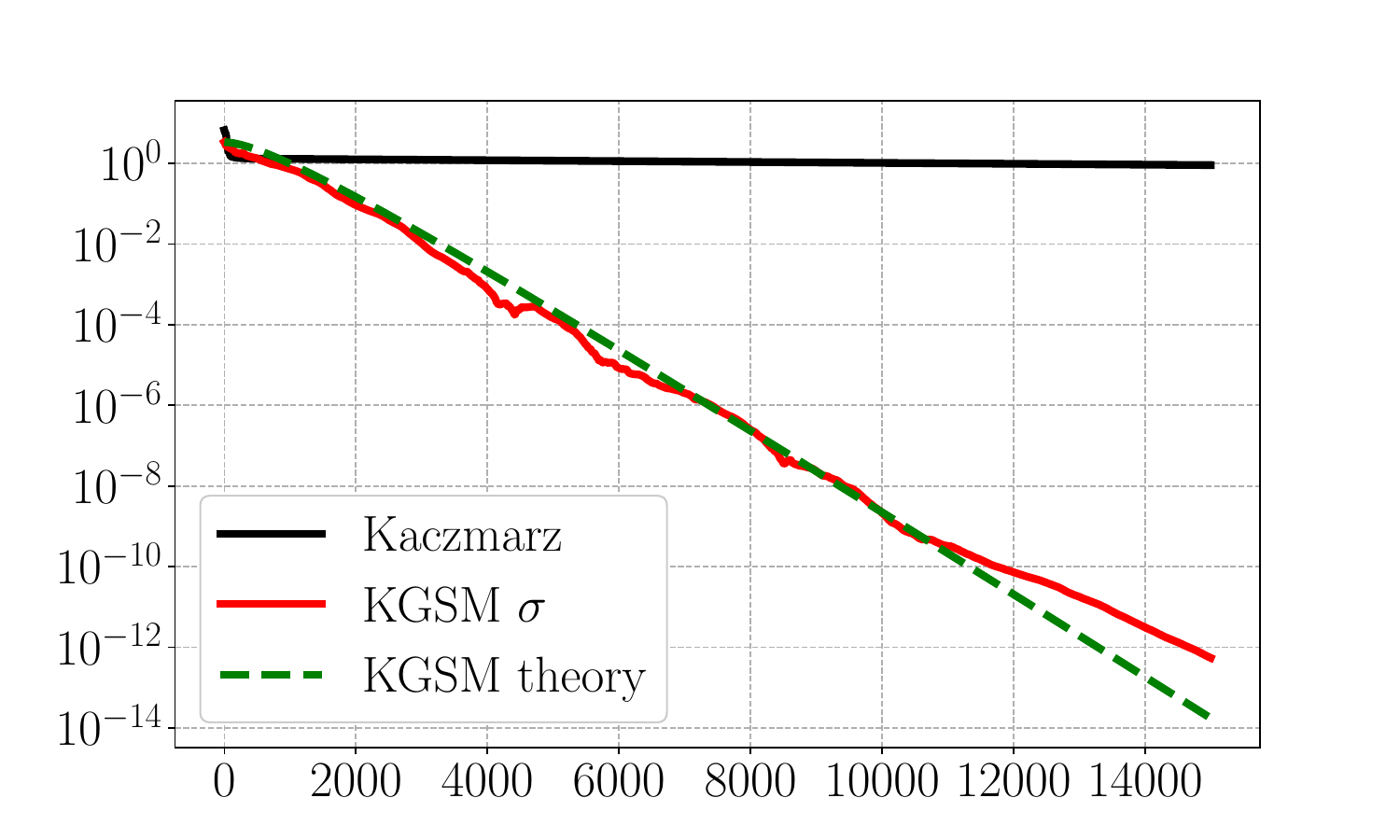} \\
    \end{tabular}
    \caption{The error $\|x_k - x\|_2$ for randomized Kaczmarz \eqref{kaczmarz} and KGSM { \eqref{eq:our-method}} for the system defined by $A_{\mu}$ (left) and $A_{\sigma}$ (right).
    For reference, we plot the value of $|\mathbb{E} \langle x-x_k, v_{20} \rangle |$ for KGSM computed via Theorem \ref{thm1}.}
    \label{fig_mu_sigma}
\end{figure}

In both cases, the theoretical estimates from \S \ref{sec:mainresult} seem to match the numerical error. This contrasts with some of the behavior seen in the example with many small singular values in \S \ref{manybad}, where the theory largely deviated from the numerical error. Additionally, we see that KGSM converges significantly faster than randomized Kaczmarz in both cases. { Interestingly, this example seems to indicate that KGSM converges faster for convex spectral decays than concave spectral decays. Note that $\sigma_1 = \mu_1$ and $\sigma_n = \mu_n$. Since $\sigma$ is convex and $\mu$ is concave, it follows that $\|A_{\sigma}\|_F \leq \|A_{\mu}\|_F$. Since Theorem \ref{thm1} approximately predicts the behavior of both plots, the difference in convergence rates is explained by the dependence of Theorem \ref{thm1} on the quantity $\eta_l = \sigma_l/\|A\|_F$.} These two numerical examples, in combination with the numerical examples given in \S \ref{numerics} and \S \ref{gaussiansystem}, illustrate the performance benefit of KGSM over randomized Kaczmarz for a variety of different spectral decays. 

\subsection{Further phase plots} \label{additionalbetaplots} In Example \ref{explorembeta}, we explored the 
diverse dynamics of KGSM { \eqref{eq:our-method}} for four values in different regions of the $(M,\beta)$ parameter space (see Figure \ref{fig04}). In this section, we provided additional plots of $|\langle x-x_k, v_{19} \rangle|$ in Figure \ref{fig05_v19}, and $\|x -x_k\|_2$  in Figure \ref{fig05_l2}.

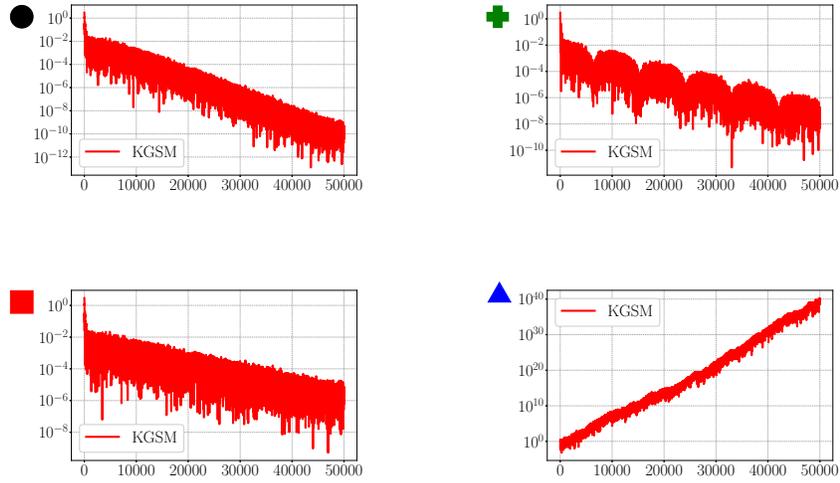
\begin{figure}[h!]
\centering
\begin{tikzpicture}
\node[anchor=north west] at (0,0) {\includegraphics[width=0.355\textwidth,trim=30 0 30 0]{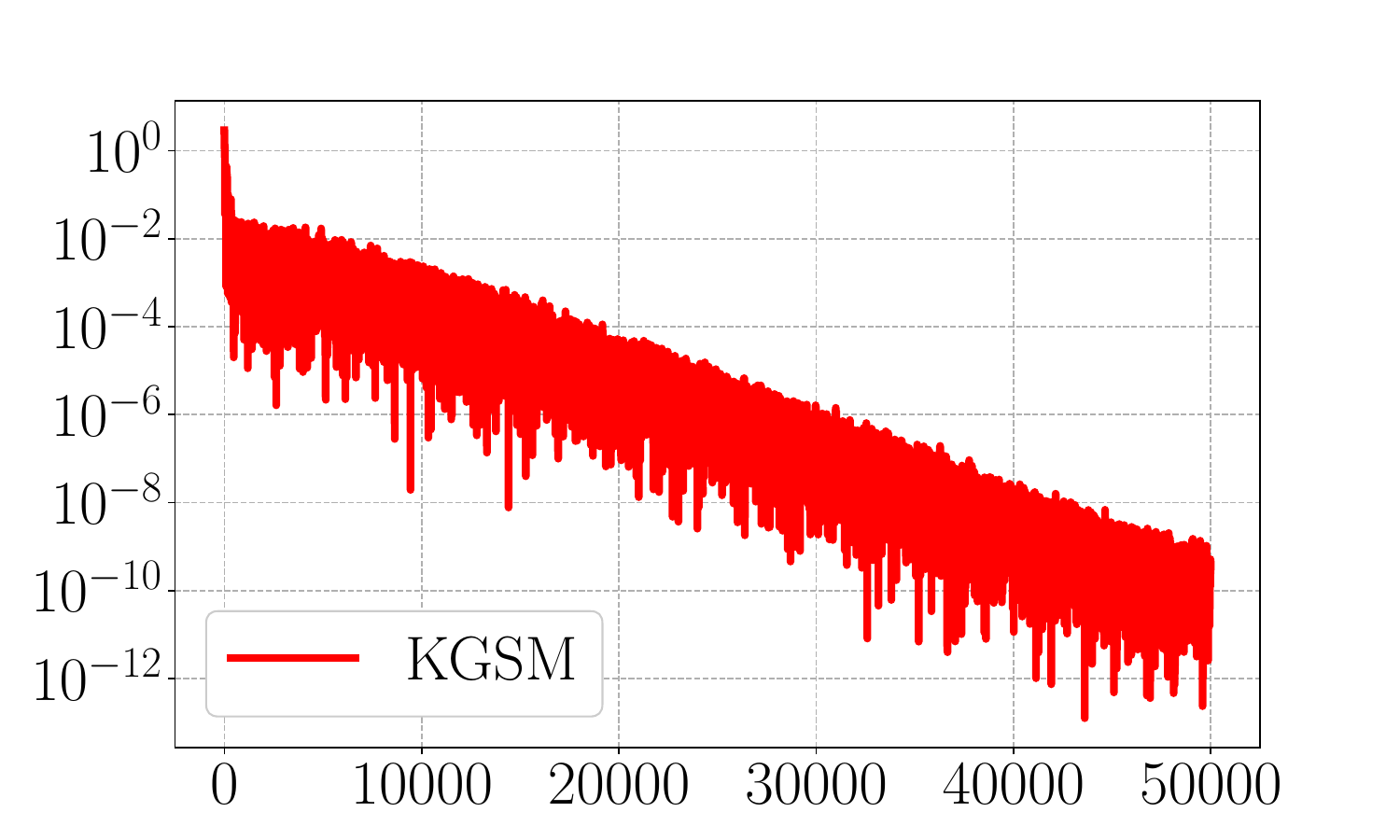}};
\node[anchor=north west] at (0.5\textwidth,0) {\includegraphics[width=0.355\textwidth,trim=30 0 30 0]{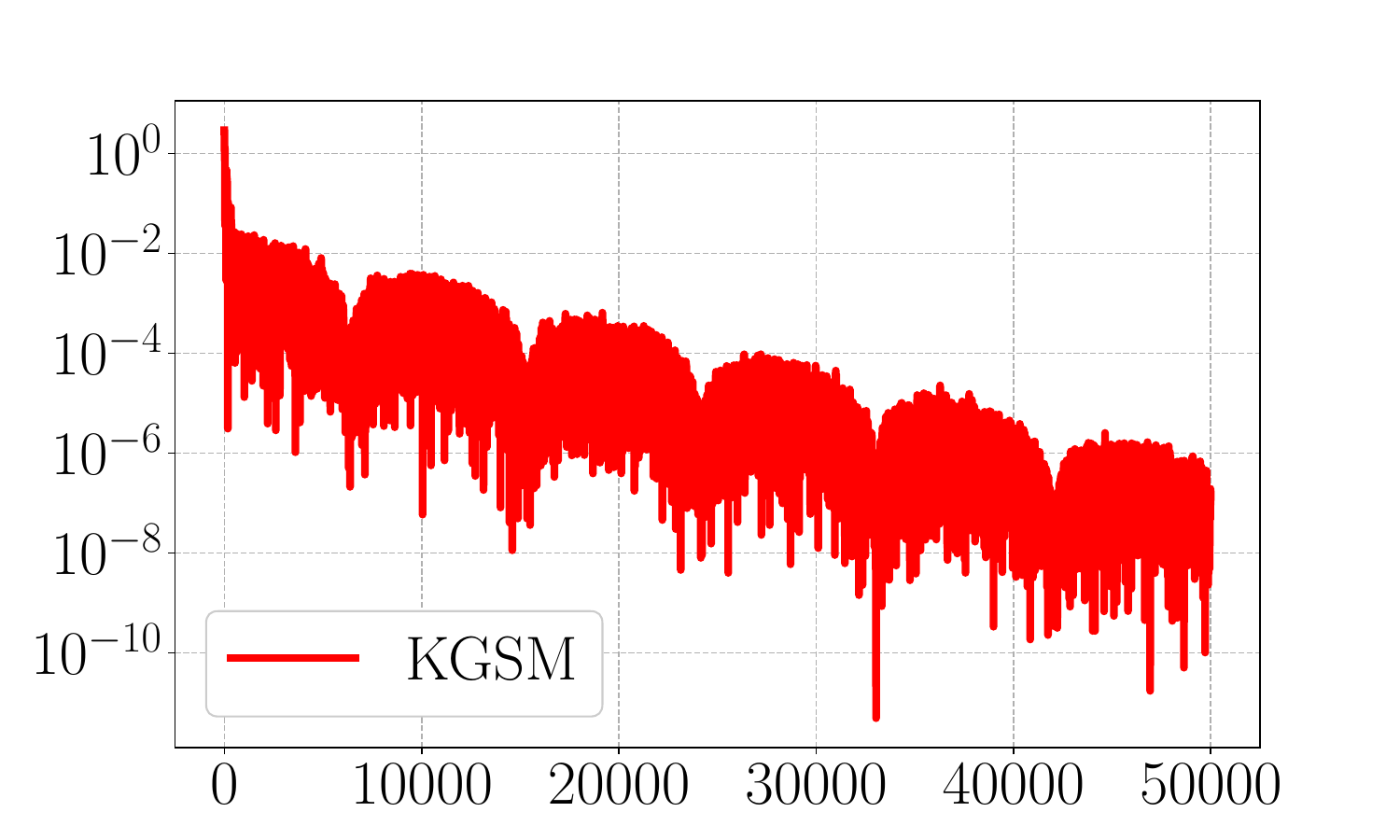}};
\node[anchor=north west] at (0,-0.3\textwidth) {\includegraphics[width=0.355\textwidth,trim=30 0 30 0]{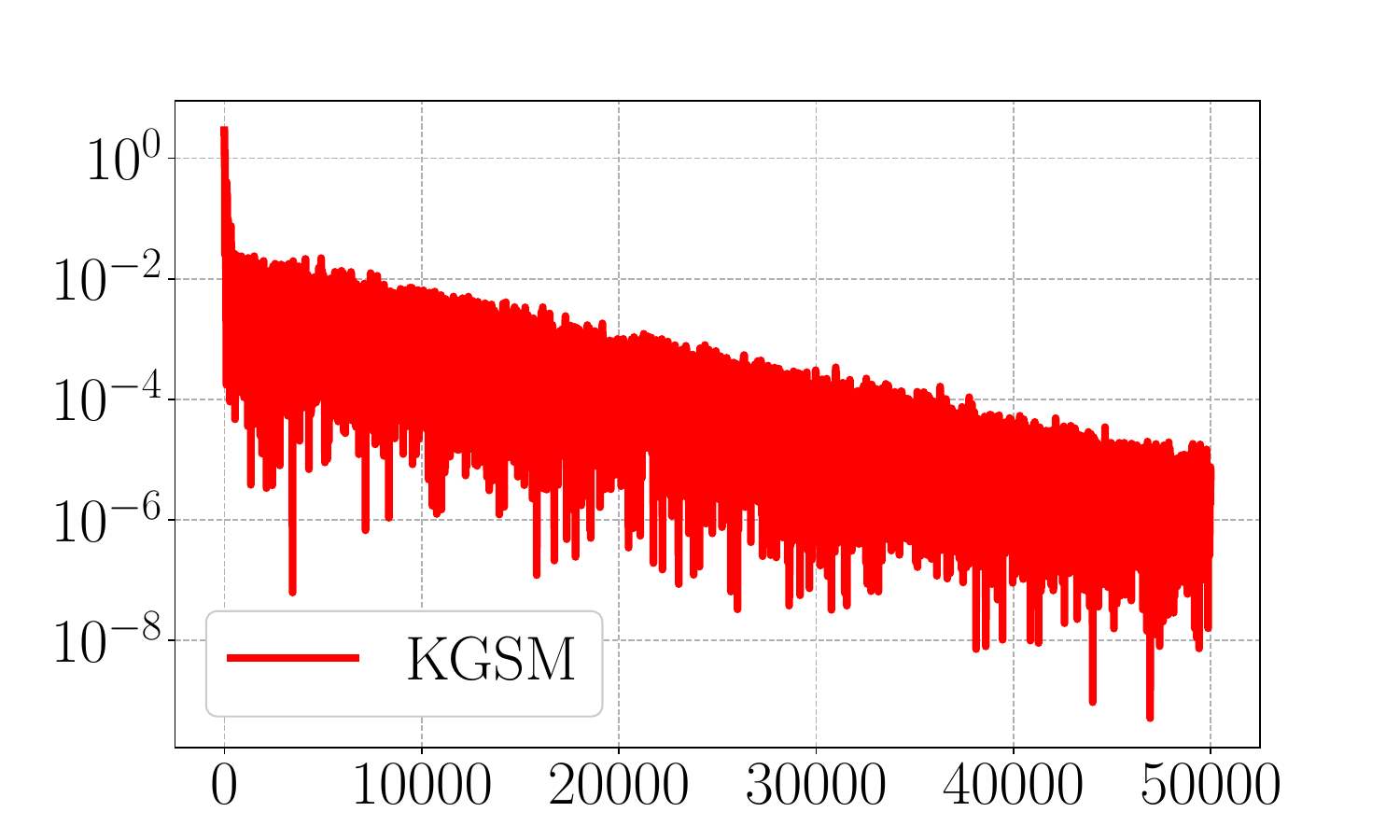}};
\node[anchor=north west] at (0.5\textwidth,-0.3\textwidth) {\includegraphics[width=0.355\textwidth,trim=30 0 30 0]{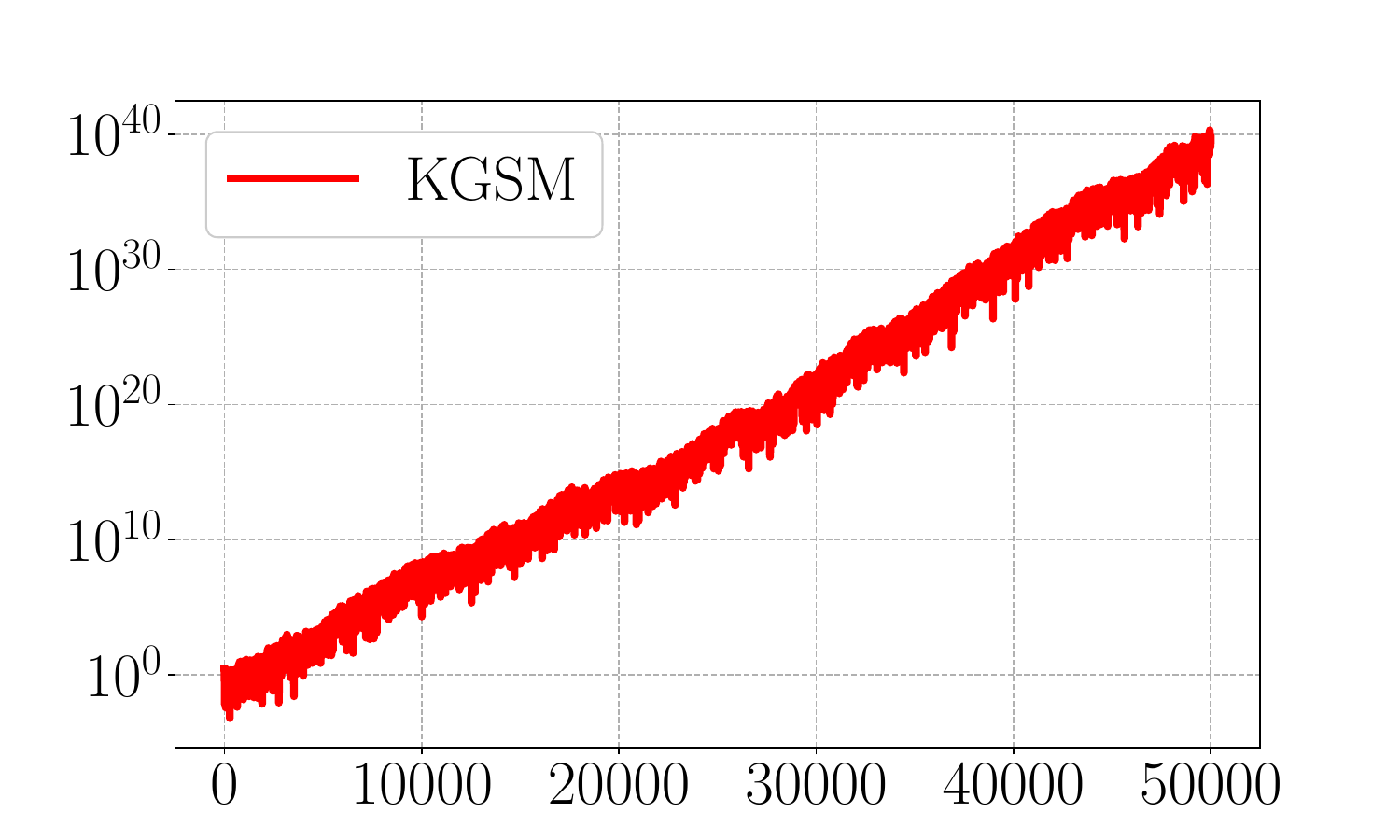}};

\node[draw=black, circle, fill=black, inner sep=0pt, minimum size=3mm] at (-0.01\textwidth, -0.05\textwidth) {};

\definecolor{mygreen}{HTML}{008000}
\begin{scope}[shift={(0.49\textwidth, -0.05\textwidth)}]
    \draw [mygreen, line width=1.3mm] (-0.15,0) -- (0.15,0); 
    \draw [mygreen, line width=1.3mm] (0,-0.15) -- (0,0.15); 
\end{scope}

\definecolor{myred}{HTML}{FF0000}
\node[draw=myred, rectangle, fill=myred, inner sep=0pt, minimum size=3mm] at (-0.01\textwidth, -0.35\textwidth) {};

\definecolor{myblue}{HTML}{0000FF}
\draw[draw=myblue, fill=myblue] (0.48\textwidth, -0.35\textwidth) -- ++(60:3mm) -- ++(-60:3mm) -- cycle;
\end{tikzpicture}
\caption{ The error $|\langle x_k -x,v_{19}\rangle|$  for KGSM \eqref{eq:our-method2} with parameters $(M,\beta)$ indicated by markers labeling each plot, which correspond to the markers in Figure \ref{fig04}.}
\label{fig05_v19}
\end{figure}

\begin{figure}[h!]
\centering
\begin{tikzpicture}
\node[anchor=north west] at (0,0) {\includegraphics[width=0.355\textwidth,trim=30 0 30 0]{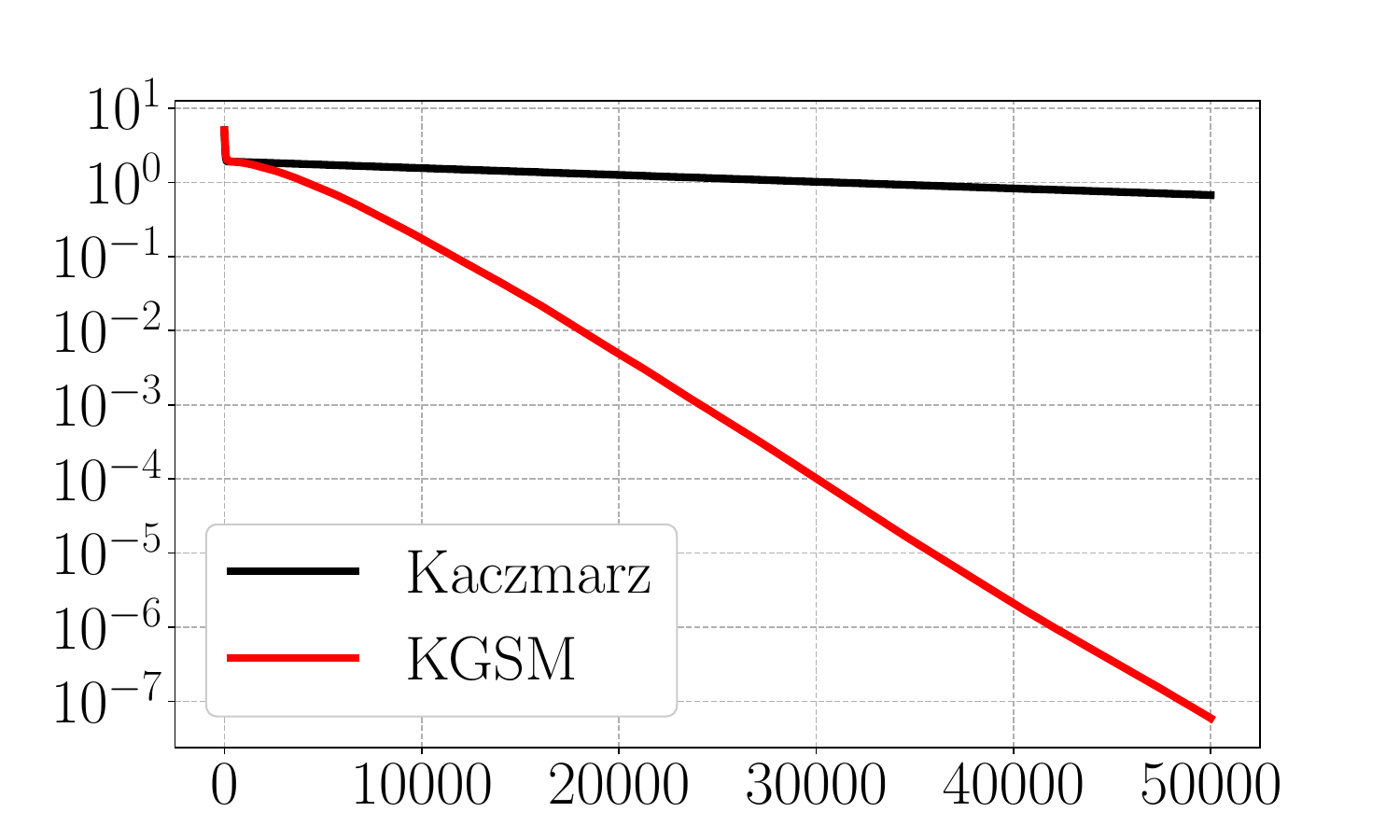}};
\node[anchor=north west] at (0.5\textwidth,0) {\includegraphics[width=0.355\textwidth,trim=30 0 30 0]{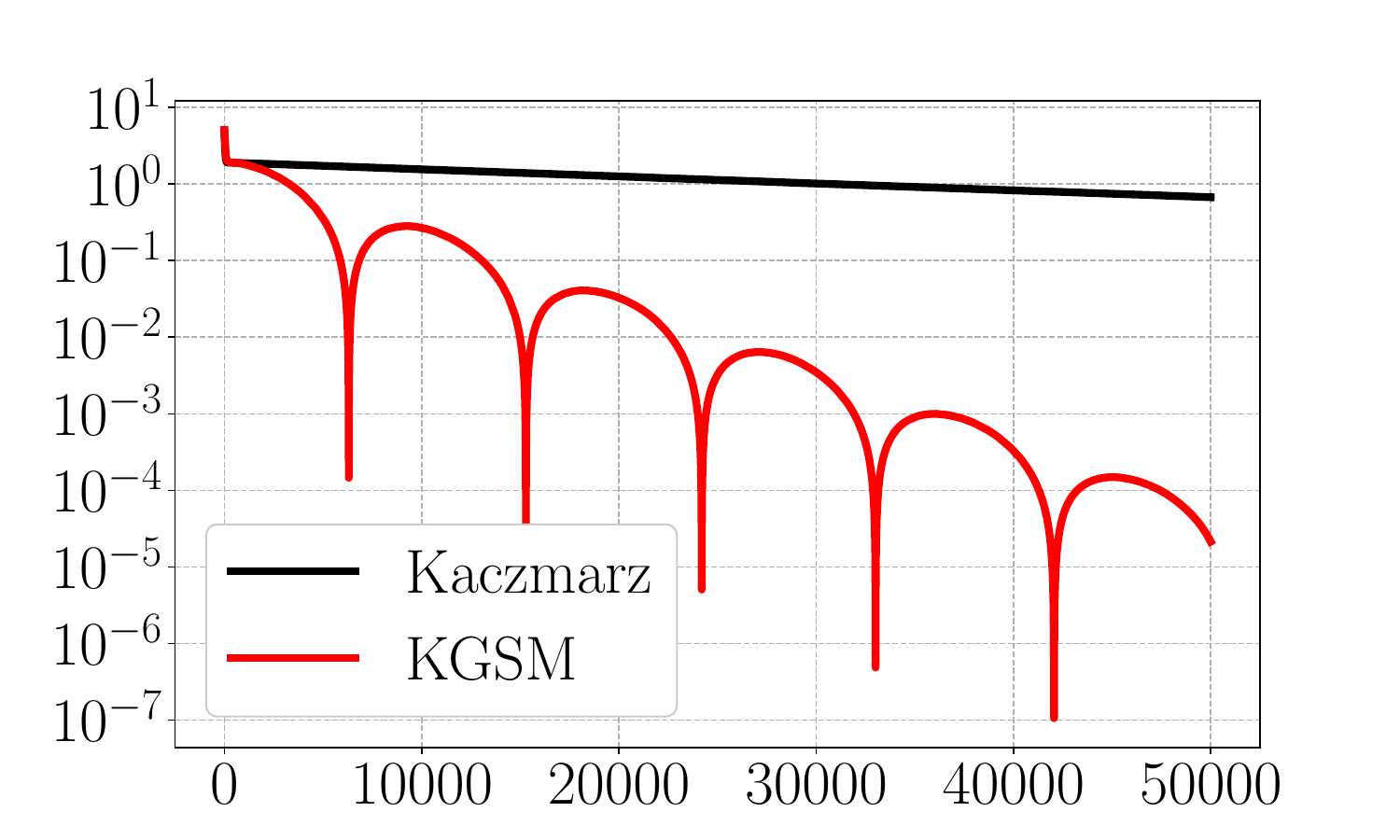}};
\node[anchor=north west] at (0,-0.3\textwidth) {\includegraphics[width=0.355\textwidth,trim=30 0 30 0]{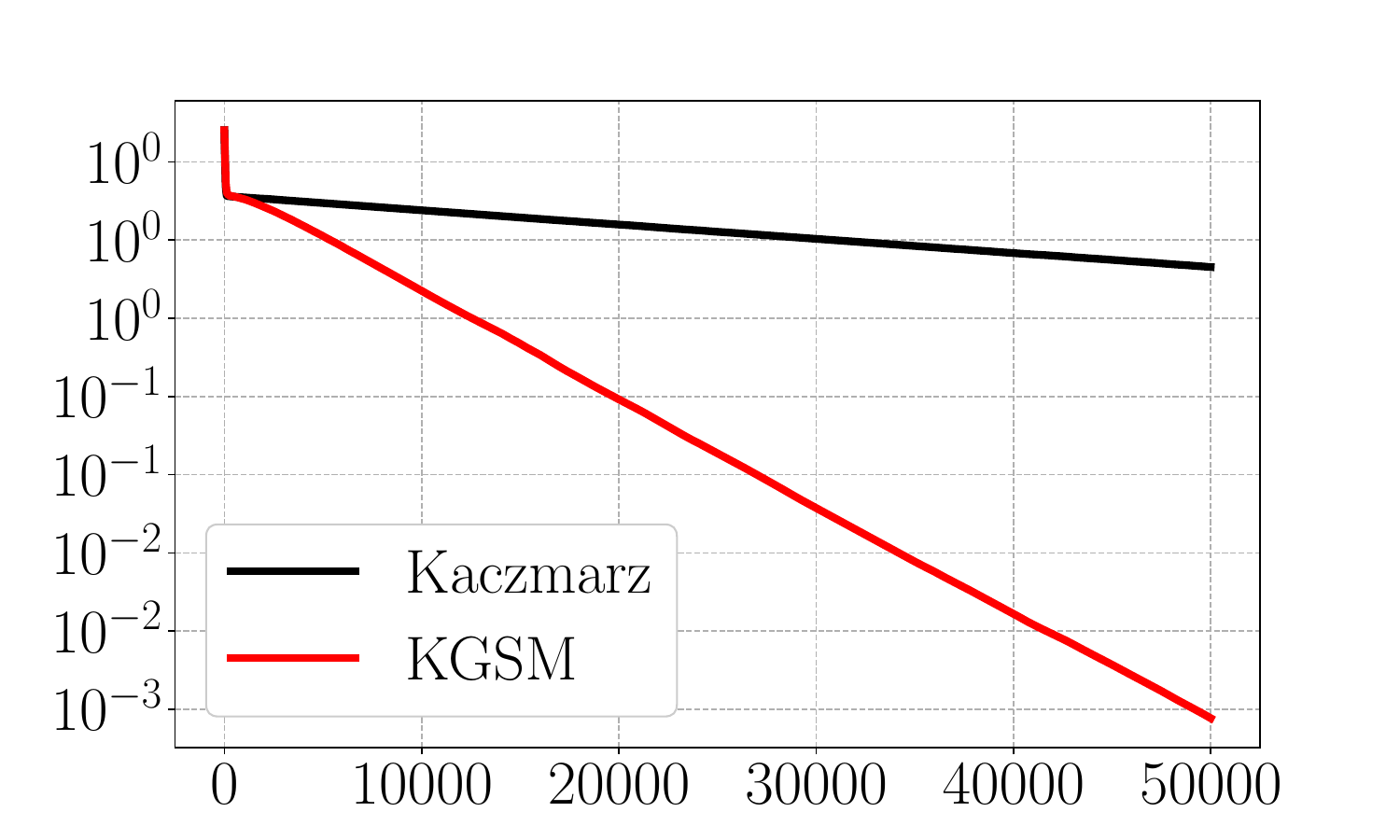}};
\node[anchor=north west] at (0.5\textwidth,-0.3\textwidth) {\includegraphics[width=0.355\textwidth,trim=30 0 30 0]{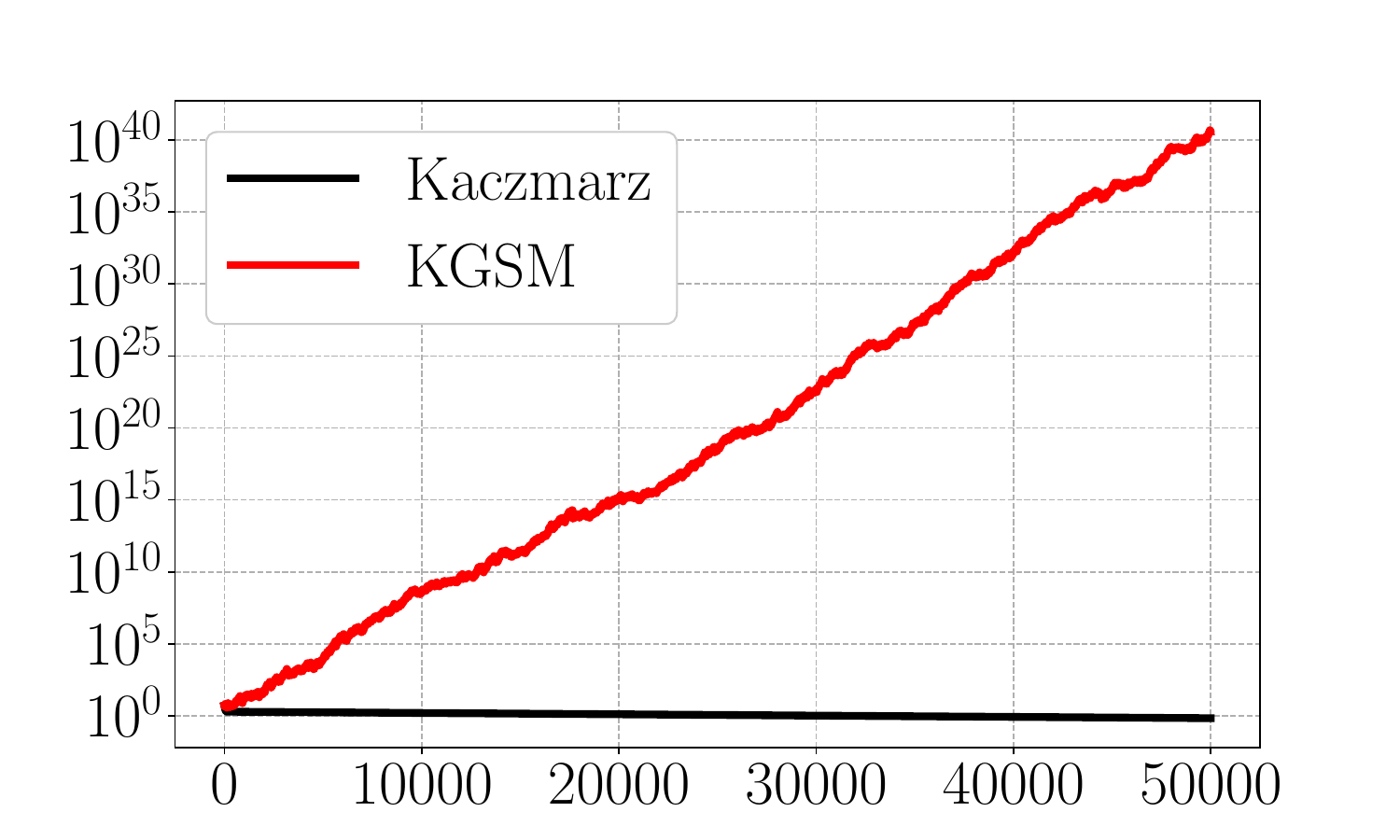}};

\node[draw=black, circle, fill=black, inner sep=0pt, minimum size=3mm] at (-0.01\textwidth, -0.05\textwidth) {};

\definecolor{mygreen}{HTML}{008000}
\begin{scope}[shift={(0.49\textwidth, -0.05\textwidth)}]
    \draw [mygreen, line width=1.3mm] (-0.15,0) -- (0.15,0); 
    \draw [mygreen, line width=1.3mm] (0,-0.15) -- (0,0.15); 
\end{scope}

\definecolor{myred}{HTML}{FF0000}
\node[draw=myred, rectangle, fill=myred, inner sep=0pt, minimum size=3mm] at (-0.01\textwidth, -0.35\textwidth) {};

\definecolor{myblue}{HTML}{0000FF}
\draw[draw=myblue, fill=myblue] (0.48\textwidth, -0.35\textwidth) -- ++(60:3mm) -- ++(-60:3mm) -- cycle;
\end{tikzpicture}
\caption{The error $\|x-x_k\|_2$  for randomized Kaczmarz \eqref{kaczmarz} and KGSM \eqref{eq:our-method2} for parameters $(M,\beta)$ indicated by markers labeling each plot which correspond to the markers in Figure \ref{fig04}.}
\label{fig05_l2}
\end{figure}
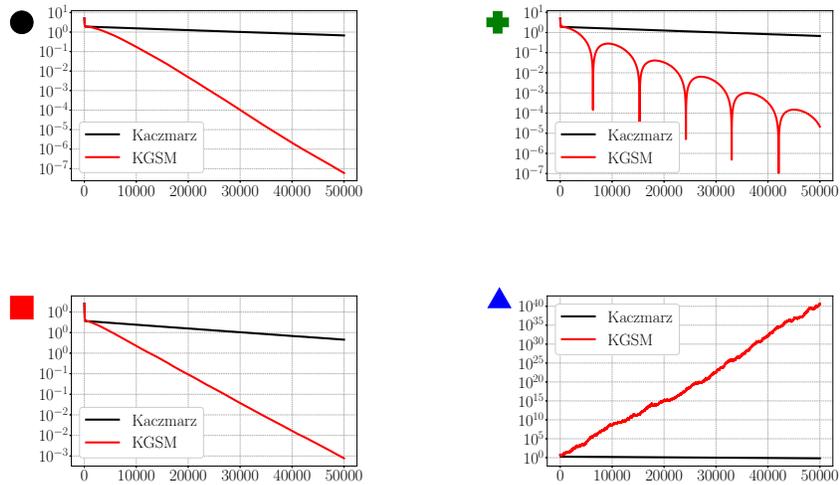

Observe that the error in the direction of $v_{19}$ appears noisy, while the $\ell^2$ error seems to be noise-free and roughly matches $|\langle x_k - x, v_{20} \rangle|$. Roughly speaking, the noise in the direction $v_{19}$ is due to the fact that the error convergences rapidly in this direction but is subject to perturbations of magnitude on the order of $|\langle x_k -x,v_{20}\rangle|$ as the algorithm runs. Finally we note that since $v_1,v_2,\cdots, v_{19}$ have the same singular value $\sigma_1=\cdots = \sigma_{19}=1$, plots of the error in the directions $v_1,\ldots,v_{18}$ all appear similar to those of $v_{19}$ in Figure \ref{fig05_l2}.
\newpage
\end{appendix}

\end{document}